\documentclass[11pt,letterpaper,reqno]{amsart}
\usepackage{tikz}
\usetikzlibrary{positioning, shapes.geometric, arrows.meta}
\usepackage{amssymb}
\usepackage{amsmath}
\usepackage{amsthm}
\usepackage{amsfonts}
\usepackage{bbm}
\usepackage{enumitem} 
\usepackage{pgfplots}
\pgfplotsset{compat=1.18} 
\usepackage{booktabs}

\usepackage{graphicx}
\usepackage[T1]{fontenc}
\usepackage{doi}
\usepackage{float} 
\usepackage{tikz}
\usepackage{pgfplots}
\pgfplotsset{compat=1.18}
\usetikzlibrary{arrows.meta}

\usepackage{bookmark}
\usepackage{hyperref}
\hypersetup{pdfstartview={FitH}}
\newcommand{\C}{\mathbb{C}}

\newcommand{\bv}{\mathbf{v}}
\newcommand{\bw}{\mathbf{w}}
\newcommand{\bz}{\mathbf{z}}
\newcommand{\tr}{\operatorname{tr}}

\newtheorem{thm}{Theorem}[section]
\newtheorem{lem}[thm]{Lemma}
\newtheorem{prop}[thm]{Proposition}
\newtheorem{cor}[thm]{Corollary}
\newtheorem{claim}{Claim}
\newtheorem{prob}[thm]{Problem}

\newtheorem{conj}[thm]{Conjecture}
\theoremstyle{definition}

\newtheorem{remark}[thm]{Remark}
\newtheorem{defn}[thm]{Definition}
\numberwithin{equation}{section}

\newcommand{\diag}{\operatorname{diag}}
\makeatother

\makeatother

\begin{document}

\title[Schoenberg inequalities and de Bruin--Sharma problem]
{Sharp Schoenberg type inequalities and the de Bruin--Sharma problem}

\author[Q.~Tang]{Quanyu Tang}
\author[T.~Zhang]{Teng Zhang}

\address{School of Mathematics and Statistics, Xi'an Jiaotong University, Xi'an 710049, P. R. China}
\email{tangquanyu@stu.xjtu.edu.cn (corresponding); tang\_quanyu@163.com}
\address{School of Mathematics and Statistics, Xi'an Jiaotong University, Xi'an 710049, P. R. China}
\email{teng.zhang@stu.xjtu.edu.cn}









\begin{abstract}
In this paper, we confirm two conjectures proposed by Georgiev, G\'{o}mez-Serrano, Tao, and Wagner~\cite{GGTW25} on Schoenberg
type inequalities of order~$4$, thereby providing a complete solution to the de Bruin--Sharma problem. We further develop a new interpolation framework to study Schoenberg type inequalities and, in particular, give a new proof of Pereira's result. Motivated by Sendov's conjecture, we then derive sharp Schoenberg type inequalities of orders $-1$ and $-2m$ (with $m \in \mathbb{N}$), as well as non-sharp inequalities valid for all negative orders $p \le -1$. Finally, we discuss a dual counterpart of the Schoenberg type inequalities of order~$-2$.
\end{abstract}

\maketitle
\tableofcontents
\section{Introduction}
\subsection{The de Bruin--Sharma problem}
The study of relations between the zeros and critical points of complex polynomials is a classical theme in the geometry of polynomials, with origins tracing back to the works of Gauss and Lucas \cite{BE95}, and modern developments motivated by conjectures of Sendov, Schoenberg, and others \cite{Mar66, QS02, Tao22}. 
Let 
\[
p(z) = z^n + a_1 z^{n-1} + \cdots + a_{n-1} z + a_n
\]
be a complex polynomial of degree $n \ge 2$, with zeros $z_1, \dots, z_n$ and \emph{critical points} $w_1, \dots, w_{n-1}$, i.e., the zeros of $p'(z)$.

In 1986, Schoenberg \cite{Sch86} conjectured that if the centroid of the zeros lies at the origin, i.e.\ $\sum_{j=1}^n z_j = 0$, then the quadratic inequality
\begin{equation}\label{eq:schoenberg}
	\sum_{k=1}^{n-1} |w_k|^2 \le \frac{n-2}{n} \sum_{j=1}^n |z_j|^2
\end{equation}
holds, with equality if and only if all zeros are collinear in the complex plane. This was later confirmed in an equivalent, more general form~\cite{BIS99, Kat99, Mal04, Per03}, and has since inspired a series of higher-order generalizations, collectively referred to as \emph{Schoenberg type inequalities}. As an interesting application of \eqref{eq:schoenberg}, Zhang~\cite{Zha25} presented an improvement of Schmeisser's conjecture \cite{Sch77} for the barycenter case.

For the quartic case, that is, when the order is $4$, de~Bruin and Sharma~\cite{dS99} conjectured that if $\sum_{j=1}^{n} z_j = 0$, then 
\begin{equation}\label{eq:dbs}
	\sum_{k=1}^{n-1} |w_k|^4 \le \frac{n-4}{n} \sum_{j=1}^{n} |z_j|^4 + \frac{2}{n^2} \left( \sum_{j=1}^{n} |z_j|^2 \right)^2.
\end{equation}
This conjecture was first proved by Cheung and Ng~\cite{CN06}, and later sharpened by Kushel and Tyaglov~\cite{KT16}, who established
\begin{equation}\label{eq:kt}
	\sum_{k=1}^{n-1} |w_k|^4 \le \frac{n-4}{n} \sum_{j=1}^{n} |z_j|^4 + \frac{1}{n^2} \left( \sum_{j=1}^{n} |z_j|^2 \right)^2 + \frac{1}{n^2} \left| \sum_{j=1}^{n} z_j^2 \right|^2.
\end{equation}

As a variant of \eqref{eq:dbs} in which the parameters to be optimized range over a two-dimensional space, we recall the 
\href{https://google-deepmind.github.io/alphaevolve_repository_of_problems/problems/24.html}{de~Bruin--Sharma problem}, 
which appears as \cite[p. 32, Problem~6.24]{GGTW25}.


\begin{prob}[de Bruin--Sharma]\label{prob:dbs_12345}
For $n \ge 4$, let $\Omega(n)$ be the set of pairs $(\alpha,\beta) \in \mathbb{R}_{\ge 0}^2$ such that, whenever $p(z)$ is a degree $n$ polynomial with zeros $z_1,\dots,z_n$ satisfying $\sum_{j=1}^n z_j = 0$, and $w_1,\dots,w_{n-1}$ are the critical points (zeros of $p'(z)$), that
\begin{equation}\label{eq:dbs1212}
  \sum_{k=1}^{n-1} |w_k|^4
  \le
  \alpha \sum_{j=1}^n |z_j|^4
  +
  \beta \left(\sum_{j=1}^n |z_j|^2\right)^{2}.
\end{equation}
What is $\Omega(n)$?
\end{prob}

In the discussion of Problem~\ref{prob:dbs_12345} in \cite[p.~33]{GGTW25}, 
Georgiev, G\'{o}mez-Serrano, Tao, and Wagner proposed the following two conjectures.

\begin{conj}[\cite{GGTW25}]\label{conj1:p=4}
	Let \( z_1, z_2, \ldots, z_n \) be the zeros of a complex polynomial \( p(z) \) of degree $n$, and let \( w_1, w_2, \ldots, w_{n-1} \) be its critical points. 
	Assume further that $\sum_{j=1}^n z_j=0$. 
	Then for every odd integer $n>4$, we have
	\[
	\sum_{k=1}^{n-1} |w_k|^4 \le \frac{n^3 - 2n^2 + 3n - 14}{n(n^2 + 3)} 
	\sum_{j=1}^n |z_j|^4 .
	\]
\end{conj}

\begin{conj}[\cite{GGTW25}]\label{conj2:p=4}
	Let \( z_1, z_2, \ldots, z_n \) be the zeros of a complex polynomial \( p(z) \) of degree $n$, and let \( w_1, w_2, \ldots, w_{n-1} \) be its critical points. Assume further that $\sum_{j=1}^n z_j=0$. Then for all $n \geq 4$, we have
	\[
	\sum_{k=1}^{n-1} |w_k|^4 \le \frac{(n-2)^4+n-2}{n^2(n-1)^2} \left(\sum_{j=1}^n |z_j|^2\right)^2 .
	\]
\end{conj}

In this paper, we confirm both Conjecture~\ref{conj1:p=4} and Conjecture~\ref{conj2:p=4}, thereby providing a complete solution to Problem~\ref{prob:dbs_12345}; see Theorem~\ref{thm:solution of de Bruin--Sharma problem} below. In fact, we show that these two conjectures follow by combining the quartic estimates of Cheung--Ng \eqref{eq:dbs} and Kushel--Tyaglov \eqref{eq:kt}
with new sharp inequalities for certain zero configuration.

\begin{thm}\label{thm:solution of de Bruin--Sharma problem}
	Let $n\ge 4$, and introduce the four anchor points in the plane:
	\[
	\left(\alpha_{\mathrm{CN}}(n),\beta_{\mathrm{CN}}(n)\right)
	=\left(\frac{n-4}{n},\frac{2}{n^2}\right),\qquad
    \left(0,\beta_{\mathrm{all}}(n)\right)
	=\left(0,\frac{(n-2)^4+n-2}{n^2(n-1)^2}\right),
	\]
	\[
	\left(\alpha_{\mathrm{even}}(n),0\right)
	=\left(\frac{n-2}{n}, 0\right),\qquad
	\left(\alpha_{\mathrm{odd}}(n),0\right)
	=\left(\frac{n^3-2n^2+3n-14}{n(n^2+3)}, 0\right).
	\]
	Let $\operatorname{co}(\cdot)$ denote the convex hull. Define the two closed, convex "upper-right closures" by
	\[
	\widetilde{\Omega}_{\mathrm{odd}}(n)
	:=\Big\{(\alpha,\beta)\in\mathbb{R}_{\ge0}^2: \exists (\alpha',\beta')\in
	\operatorname{co}\big\{(\alpha_{\mathrm{CN}},\beta_{\mathrm{CN}}),
	(\alpha_{\mathrm{odd}},0),(0,\beta_{\mathrm{all}})\big\}
	\text{ with }\alpha\ge \alpha', \beta\ge \beta'\Big\},
	\]
	\[
	\widetilde{\Omega}_{\mathrm{even}}(n)
	:=\Big\{(\alpha,\beta)\in\mathbb{R}_{\ge0}^2: \exists (\alpha',\beta')\in
	\operatorname{co} \big\{(\alpha_{\mathrm{CN}},\beta_{\mathrm{CN}}),
	(\alpha_{\mathrm{even}},0),(0,\beta_{\mathrm{all}})\big\}
	\text{ with }\alpha\ge \alpha', \beta\ge \beta'\Big\}.
	\]
	Then the set $\Omega(n)$ defined in Problem~\ref{prob:dbs_12345} satisfies
	\[
	\Omega(n)=
	\begin{cases}
		\widetilde{\Omega}_{\mathrm{odd}}(n), & \text{if $n>4$ is odd},\\
		\widetilde{\Omega}_{\mathrm{even}}(n), & \text{if $n\ge 4$ is even}.
	\end{cases}
	\]
\end{thm}
Theorem \ref{thm:solution of de Bruin--Sharma problem} for $n=7$ see Figure \ref{Fig1}.  
\begin{figure}[ht] 
	\centering
	\pgfmathsetmacro{\aOdd}{9/13}      
	\pgfmathsetmacro{\aCN}{3/7}        
	\pgfmathsetmacro{\bCN}{2/49}       
	\pgfmathsetmacro{\bAll}{5/14}      
	\pgfmathsetmacro{\Kodd}{84/13}     
	\pgfmathsetmacro{\Call}{31/42}     
	
	\begin{tikzpicture}
		\begin{axis}[
			width=11.5cm,height=7.8cm,
			axis lines=left,
			xlabel={$\alpha$}, ylabel={$\beta$},
			xmin=0, xmax=0.75,
			ymin=0, ymax=0.40,
			xtick={0,0.2,0.4,0.6},
			ytick={0,0.1,0.2,0.3,0.4},
			grid=both,
			grid style={line width=.1pt, draw=gray!30},
			major grid style={line width=.2pt, draw=gray!45},
			clip=false
			]
			
			\coordinate (Aodd) at (axis cs:\aOdd,0);
			\coordinate (ACN)  at (axis cs:\aCN,\bCN);
			\coordinate (Ball) at (axis cs:0,\bAll);
			
			\path[fill=gray!18]
			(Aodd) --
			(axis cs:0.75,0) --
			(axis cs:0.75,0.40) --
			(axis cs:0,0.40) --
			(Ball) -- (ACN) -- cycle;
			
			\addplot[very thick,black] coordinates {(\aOdd,0) (\aCN,\bCN)};
			\addplot[very thick,black] coordinates {(\aCN,\bCN) (0,\bAll)};
			
			\addplot[only marks,mark=*,mark size=2.2pt,
			mark options={fill=red!80!black,draw=red!80!black}]
			coordinates {(\aOdd,0)};
			\addplot[only marks,mark=*,mark size=2.2pt,
			mark options={fill=blue!80!black,draw=blue!80!black}]
			coordinates {(\aCN,\bCN)};
			\addplot[only marks,mark=*,mark size=2.2pt,
			mark options={fill=green!60!black,draw=green!60!black}]
			coordinates {(0,\bAll)};
			
			\node[anchor=south west, font=\scriptsize, xshift=2pt, yshift=2pt] at (Aodd)
			{$\big(\alpha_{\mathrm{odd}}(7),0\big)$};
			\node[anchor=south west, font=\scriptsize] at (ACN)
			{$(\alpha_{\mathrm{CN}}(7),\beta_{\mathrm{CN}}(7))$};
			\node[anchor=west, font=\scriptsize] at (Ball)
			{$(0,\beta_{\mathrm{all}}(7))$};
			
			\node[anchor=south east, font=\footnotesize]
			at (axis cs:0.75,0.40) {$\widetilde{\Omega}_{\mathrm{odd}}(7)$ (shaded)};
			
		\end{axis}
	\end{tikzpicture}
	\caption{The region $\widetilde{\Omega}_{\mathrm{odd}}(7)$ in the $(\alpha,\beta)$-plane:
		the feasible set is the region above and to the right of
		the convex hull of the three points 
		$(\alpha_{\mathrm{odd}}(7),0)$, $(\alpha_{\mathrm{CN}}(7),\beta_{\mathrm{CN}}(7))$, 
		and $(0,\beta_{\mathrm{all}}(7))$.}\label{Fig1}
\end{figure}

\subsection{Schoenberg type inequalities of general orders} Apart from the quadratic and quartic orders, sharp results for other positive integer exponents are virtually nonexistent, particularly for non-even orders. Kushel and Tyaglov~\cite[Section~5]{KT16} highlighted the technical difficulty of expressing certain matrix-trace expressions in terms of the polynomial's zeros for even orders $2m \ge 6$ as well as for odd orders $2m+1$, and in particular they raised the following problem:
\begin{prob}[\cite{KT16}]\label{prob:kt}
Is it possible to estimate $\sum_{k=1}^{n-1} |w_k|^m$ for $m \in \mathbb{N}$\,?
\end{prob}

Tang~\cite{Tang25} resolved the next even-order case $m=6$ (see Theorem~\ref{thm:inequality_of_order_6_1} in Appendix~\ref{sec:appendix_case_p_6}) and established the first odd-order case $m=1$ (see~\cite[Theorem~3.4]{Tang25}). In particular, in the case $m=1$ it was shown that  
\begin{equation}\label{eq:p=1}
	\sum_{k=1}^{n-1} |w_k|
	\leq \sqrt{\frac{n-2}{n}} \sum_{j=1}^n |z_j|,
\end{equation}
whenever $\sum_{j=1}^n z_j = 0$. Nevertheless, no general method has been available to extend these sharp results to higher integer orders, let alone to non-integer orders.

In~\cite{LXZ21}, Lin, Xie, and Zhang proved an inequality for general orders \(p \ge 2\) and explicitly raised the question of whether further improvements are possible, since \cite[Theorem~3.2]{LXZ21} gives a non-sharp bound even when \(p = 4\). In~\cite[Proposition~5.2]{Tang25}, Tang refined this result, although the resulting bound remains non-sharp for all \(p \ge 4\).


It is worth noting that, for every real number $p \ge 1$, Pereira~\cite[Corollary~5.5]{Per03} established a sharp inequality relating the $\ell^p$-norms of the zeros and the critical points:
\begin{equation}\label{eq:Pereira1}
  \sum_{k=1}^{n-1} |w_k|^p \le \frac{n-1}{n}\sum_{j=1}^n |z_j|^p,    
\end{equation}
which holds without the assumption that the centroid of the zeros is at the origin. 

Pereira's proof relies on a majorization relation between the zeros of a polynomial and its critical points.
In this paper we develop a new approach to \eqref{eq:Pereira1} based on complex interpolation for compatible Banach couples, applied simultaneously to $\ell^p$ spaces and Schatten $p$-classes $S_p$.
This framework allows us to pass directly from the endpoint cases $p = 1$, $p = 2$, and $p = \infty$ to all intermediate exponents, and along the way we obtain a partial refinement of \eqref{eq:Pereira1} under the additional assumption $\sum_{j=1}^n z_j = 0$. In addition, we provide a new proof of the Schoenberg type inequality of order~$1$, namely~\eqref{eq:p=1}.

In the course of our study of Schoenberg type inequalities of general orders, we also propose the following two conjectures, which represent two central puzzles arising from Problem~\ref{prob:kt}.

\begin{conj}\label{conj: sum z_j=0}
	Let \( n \ge 4 \) be an integer and let \( p \ge 2 \) be a real number. Let \( z_1, z_2, \ldots, z_n \) be the zeros of a complex polynomial \( p(z) \) of degree $n$, and let \( w_1, w_2, \ldots, w_{n-1} \) be its critical points. 
	Assume further that $\sum_{j=1}^n z_j=0$. Then
	\[
	\sum_{k=1}^{n-1} |w_k|^p \le \frac{n-2}{n} \sum_{j=1}^n |z_j|^p .
	\]Moreover, the equality can be attained by \( p(z) = (z-1)^{\frac{n}{2}}(z+1)^{\frac{n}{2}} \) with \( n \ge 4 \) even.
\end{conj}
\begin{conj}\label{conj:two-point-odd12}
	Let \( n > 4 \) be an odd integer and let \( p \ge 2 \) be a real number. Let \( z_1, z_2, \ldots, z_n \) be the zeros of a complex polynomial \( p(z) \) of degree $n$, and let \( w_1, w_2, \ldots, w_{n-1} \) be its critical points. 
	Assume further that $\sum_{j=1}^n z_j=0$. Then
	\[
	\sum_{k=1}^{n-1} |w_k|^p \le \frac{(n-3)(n+1)^{p}+(n-1)^{p+1}+2^{p+1}}
	{(n-1)(n+1)^{p}+(n+1)(n-1)^{p}} \sum_{j=1}^n |z_j|^p .
	\]Moreover, equality is attained for the polynomial
\[
p(z) = \left(z-(n+1)\right)^{\frac{n-1}{2}} \left(z+(n-1)\right)^{\frac{n+1}{2}}.
\]
\end{conj}

\begin{remark}\label{remark:two_conj_solved_case_11}
Clearly, when \( p = 2 \), Conjectures~\ref{conj: sum z_j=0} and~\ref{conj:two-point-odd12} both reduce to Schoenberg's inequality~\eqref{eq:schoenberg}. When \( p = 4 \), Conjecture~\ref{conj: sum z_j=0} can be deduced directly from~\eqref{eq:dbs} together with the Cauchy--Schwarz inequality, while Conjecture~\ref{conj:two-point-odd12} reduces to Conjecture~\ref{conj1:p=4}. When \( p = 6 \), Conjecture~\ref{conj: sum z_j=0} is also valid; this is proved in Corollary~\ref{cor:inequality_of_order_6_11} in Appendix~\ref{sec:appendix_case_p_6}.
\end{remark}

\subsection{Normalized polynomials and related conjectures}

A complex polynomial $P(z)$ is said to be \emph{normalized} if it is monic and satisfies $P(0) = 0$. Such polynomials were employed by Smale~\cite{Sma81} in the context of his ``mean value conjecture''. Motivated by Smale's work, Kostrikin~\cite{Kos84} put forward several conjectures concerning normalized polynomials. Notably, when formulated in terms of normalized polynomials, the well-known \href{https://google-deepmind.github.io/alphaevolve_repository_of_problems/problems/20.html}{Sendov's conjecture} (see, e.g.,~\cite{Hay19}) is equivalent to the following statement.

\begin{conj}[Sendov]\label{conj:Sendov}
    Let \[
	P(z) = z \prod_{j=1}^{n} (z - z_j)
	\] with all $z_j$ lying in the closed disk \(\{z \in \C : |z + a| \le 1\}\), where $a \in [0,1]$. Let $w_1,\dots,w_n$ denote the zeros of the derivative $P'(z)$. Then
	\[
	\min_{1\le k\le n}|w_k|\le 1.
	\]
\end{conj}
Numerical evidence suggests that a stronger version of Sendov's conjecture should hold.
\begin{conj}\label{conj1.2}
 Let
	\[
	P(z) = z \prod_{j=1}^{n} (z - z_j)
	\]
with all $z_j$ lying in the closed disk $\{z \in \C : |z + a| \le 1\}$, where $a \in [0,1]$. Let $w_1,\dots,w_n$ denote the zeros of the derivative $P'(z)$. Then for any $\lambda\ge 1$,
\begin{align}\label{eq:conj -p}
  \sum_{k=1}^n\frac{1}{|w_k|^\lambda}\ge n.  
\end{align}
\end{conj} 

It is clear that, as $\lambda \to \infty$, Conjecture~\ref{conj1.2} reduces to Conjecture~\ref{conj:Sendov}. We also remark that, when $0 < \lambda < 1$, inequality~\eqref{eq:conj -p} fails in general. Indeed, take $a = 1$ and $z_j = -1 - a$ for all $j$. Then
\[
\sum_{k=1}^n \frac{1}{\lvert w_k \rvert^\lambda} 
= \frac{(n+1)^\lambda + (n - 1)}{2^\lambda} 
< n,
\]
for every $0 < \lambda < 1$.

\subsection{Towards Schoenberg type inequalities of negative orders}

As indicated earlier, several elegant bounds have been established for power sums of the form
\[
\sum_{k} |w_k|^{p}, \qquad p\ge 1.
\]
This naturally leads to the following question:
\begin{quote}
	\emph{Do analogous estimates continue to hold when the exponent \( p \) is negative?}
\end{quote}
This question is, moreover, closely connected to Conjecture~\ref{conj1.2}.

A basic difficulty is that the expression
\(
\sum_{k} |w_k|^{p} 
\)
is meaningful for \(p<0\) only when none of the critical points \(w_k\) vanish. To ensure this, we restrict our attention to a subclass of normalized polynomials that vanish at the
origin and have all other zeros nonzero. Specifically, we consider
\begin{equation*}\label{eq:main-1}
	P(z) = z \prod_{j=1}^{n} (z - z_j),
	\qquad z_j \in \C \setminus \{0\},
\end{equation*} 
and denote by \(w_1,\dots,w_n\) the zeros of the derivative \(P'(z)\), i.e.,
the critical points of \(P(z)\). Since \(P(0)=0\) while \(P'(0) = \prod_{j=1}^{n} (-z_j) \neq 0\), which follows that none of the critical points vanish:
\[
w_k \neq 0 \qquad (1\le k\le n).
\]

In this paper we establish Schoenberg type inequalities of negative orders for such polynomials.
We obtain a sharp inequality for the order $p = -1$, together with a sharp family for all even negative orders $p = -2m$ ($m \in \mathbb{N}$), which in particular includes the case $p = -2$, and admits a further sharpening in the special case $p = -4$.
In addition, we derive non-sharp inequalities that are valid for the full range of exponents $p \le -1$.

Let \( \bv = (v_1, \ldots, v_m) \) with \( v_1, \ldots, v_m \in \mathbb{C} \). We denote by \( |\bv| = (|v_1|, \ldots, |v_m|) \) the coordinate-wise modulus of \( \bv \), by \( \bv^{-1} = (v_1^{-1}, \ldots, v_m^{-1}) \) its coordinate-wise inverse, and by \( e_k(\bv) \) the \( k \)th elementary symmetric function of \( v_1, \ldots, v_m \) (see, e.g.,~\cite[p.~54]{HJ13}).

In contrast to Lemma~\ref{lem:ek}, our first result concerning elementary symmetric polynomials is the following.
\begin{thm}[Theorem~\ref{eq:ele-minus-one}]\label{mthm1}Let
\[
  P(z) = z \prod_{j=1}^{n} (z - z_j),
  \qquad z_j \in \C \setminus \{0\},
\]
and let $w_1,\dots,w_n$ denote the zeros of $P'(z)$. Then
\begin{equation}\label{eq:intro-ele-minus-one}
  e_k(|\bw|^{-1})
  \le (k+1) e_k(|\bz|^{-1}),  \qquad  k=1, \ldots, n,
\end{equation}
where \( \bz = (z_1,\dots,z_n) \) and \( \bw = (w_1,\dots,w_n) \).
Moreover, equality in \eqref{eq:intro-ele-minus-one} holds for all $k=1,\dots,n$ whenever
$z_1,\dots,z_n$ all lie on the same ray emanating from the origin.
\end{thm}
Taking $k=1$ in Theorem \ref{mthm1}, we derive a sharp Schoenberg type inequality of order $-1$:
\[
\sum_{k=1}^{n} \frac{1}{|w_k|}
\le 2 \sum_{j=1}^{n} \frac{1}{|z_j|}.
\]

For general negative even orders $-2m$, we have the following result.
\begin{thm}[Theorem \ref{thm:general-2m}]\label{mthm2}
	Let
	\[
	P(z) = z \prod_{j=1}^n (z - z_j), \qquad z_j \in \mathbb{C} \setminus \{0\},
	\]
	and let $w_1,\dots,w_n$ be the zeros of $P'(z)$. Then for every integer $m\ge 1$ we have
	\begin{align*}
		\sum_{k=1}^n \frac{1}{|w_k|^{2m}}
		\le
		&\sum_{j=1}^n \frac{1}{|z_j|^{2m}}
		\notag\\&+
		\sum_{\ell=1}^{m} (n+2)^{\ell}
		\sum_{\substack{r_0,\dots,r_\ell \ge 0\\ r_0+\cdots+r_\ell = m-\ell}}
		\left(\sum_{i=1}^n \frac{1}{|z_i|^{2(r_0+r_\ell+1)}}\right) \cdot \prod_{j=1}^{\ell-1} \left(\sum_{i=1}^n \frac{1}{|z_i|^{2(r_j+1)}}\right).
	\end{align*}
	Moreover, equality holds if and only if $z_1=\cdots=z_n$.
\end{thm}

In particular, taking $m=1$ in Theorem~\ref{mthm2} yields a sharp Schoenberg type inequality of order $-2$:
\[
\sum_{k=1}^{n}\frac{1}{|w_k|^{2}}
\le
(n+3)\sum_{j=1}^{n}\frac{1}{|z_j|^{2}}.
\]Furthermore, as a corollary of Theorem~\ref{mthm2} (see Corollary~\ref{cor:sharp-2m}), we obtain the following sharp Schoenberg type inequalities of order $-2m$:
\[
\sum_{k=1}^n \frac{1}{|w_k|^{2m}}
\le
\frac{(n+1)^{2m}+n-1}{n}\sum_{j=1}^n \frac{1}{|z_j|^{2m}},
\qquad m \ge 1.
\]
This leads us to propose the following conjecture.
\begin{conj}\label{conj:negative_order_p_11_1}
  Let
	\[
	P(z) = z \prod_{j=1}^n (z - z_j), \qquad z_j \in \mathbb{C}\setminus\{0\},
	\]
	and let $w_1,\dots,w_n$ be the zeros of $P'(z)$. Then for every real number $p \ge 1$ we have
    \[
\sum_{k=1}^n \frac{1}{|w_k|^{p}}
		\le
		\frac{(n+1)^{p}+n-1}{n}\sum_{j=1}^n \frac{1}{|z_j|^{p}}.
    \]Moreover, equality holds if and only if $z_1=\cdots=z_n$.
\end{conj}

Using complex interpolation between the endpoint cases \(p = 1, 2, \infty\), we obtain the following full-range bounds for all \(p \ge 1\), which are in general non-sharp, to be compared with Conjecture~\ref{conj:negative_order_p_11_1}.

\begin{thm}[Theorems~\ref{thm:negative-interpolation_for_p_2_infinity} and~\ref{thm:negative-1<p<2}]
	Let
	\[
	P(z) = z \prod_{j=1}^n (z - z_j), \qquad z_j \in \mathbb{C}\setminus\{0\},
	\]
	and let $w_1,\dots,w_n$ be the zeros of $P'(z)$. Then for 
	$1\le p\le 2$ we have
	\begin{equation*}
		\sum_{k=1}^n \frac{1}{|w_k|^{p}}
		\le
		(n+3)^{\frac{p}{2}}\sum_{j=1}^n \frac{1}{|z_j|^{p}}.
	\end{equation*} For  
	$p\ge 2$ we have
	\begin{equation*}
		\sum_{k=1}^n \frac{1}{|w_k|^{p}}
		\le
		(n+1)^{p-2}(n+3)\sum_{j=1}^n \frac{1}{|z_j|^{p}}.
	\end{equation*}
\end{thm}
Finally, turning our attention to general complex polynomials, we obtain a dual counterpart of the Schoenberg type inequality of order $-2$. Let \( \Re z \) denote the real part of a complex number $z$.

\begin{thm}[Corollary~\ref{cor:dual -2}]
Let
\[
  P(z) = \prod_{j=1}^{n+1} (z - z_j),
\]
and let $w_1, \ldots, w_n$ be the zeros of $P'(z)$.
Assume that all $z_j,w_k$ lie in $\C \setminus \{0\}$ and that $w_1,\ldots,w_n$ are distinct. Then
\[
\sum_{j=1}^{n+1} \frac{1}{|z_j|^{2}}
\le
\sum_{k=1}^n \frac{1}{|w_k|^2}
+ |c|^2 \left( 1 + \sum_{k=1}^n \frac{|r_k|}{|w_k|^2} \right)^2
+ 2\Re \left( c \sum_{k=1}^n \frac{r_k}{w_k |w_k|^2} \right), 
\]
where
\[
c = \left(G - \sum_{k=1}^n\frac{r_k}{w_k}\right)^{-1}, \qquad
G = \frac{1}{n}\sum_{k=1}^n w_k, \qquad
r_k = \frac{(n+1)P(w_k)}{P''(w_k)}.
\]
Moreover, equality is attained if and only if $r_k(\overline{w_k-G})^2$ is non-negative for every $k=1,\ldots,n$.
\end{thm}

\medskip

\noindent\textbf{Organization of the paper.} Section~\ref{sec:de_Bruin_Sharma_Problem1} is devoted to the proofs of Conjectures~\ref{conj1:p=4} and~\ref{conj2:p=4} and of Theorem~\ref{thm:solution of de Bruin--Sharma problem}. In Section \ref{sec:tools}, we present several preliminary results required for the subsequent analysis. Section~\ref{sec:new proof} offers a new proof of Pereira's result and the Schoenberg type inequality of order 1; we also examine separately the ranges $1 \le p \le 2$ and $p \ge 2$ under the zero barycenter condition. Sections~\ref{sec:order_negative_1}--\ref{sec:negative_interpolation} establish Schoenberg type inequalities for negative orders. Therein, we derive sharp inequalities for the specific orders $-1$, $-2$, $-4$, and $-2m$ (with $m \in \mathbb{N}$), as well as non-sharp inequalities valid for all exponents $p \leq -1$. Finally, Section~\ref{sec: dual cases for order -2} develops a dual counterpart of the Schoenberg type inequality of order $-2$.


\section{Solution of the de Bruin--Sharma problem}\label{sec:de_Bruin_Sharma_Problem1}


In this section we first confirm Conjecture~\ref{conj1:p=4}. 
More precisely, we show that Conjecture~\ref{conj1:p=4} follows from the result of
Kushel and Tyaglov~\eqref{eq:kt}, together with a sharp inequality for the zero
configuration in the case where $n>4$ is odd (see Theorem~\ref{thm:complex-correct11}). Second, we confirm Conjecture~\ref{conj2:p=4}. We prove that
Conjecture~\ref{conj2:p=4} follows from the result of Cheung and Ng~\eqref{eq:dbs},
combined with another sharp inequality for the zero configuration, valid for every
$n\ge4$ (see Theorem~\ref{thm:complex-correct22}). Finally, we use these ingredients to obtain a complete solution to
Problem~\ref{prob:dbs_12345}.

\begin{thm}\label{thm:complex-correct11}
	Let $n>4$ be an odd integer, and let $z_1,\dots,z_n\in\mathbb{C}$ satisfy $\sum_{j=1}^n z_j=0$. Then
	\[
	\left(\sum_{j=1}^n |z_j|^2\right)^2
	+\left|\sum_{j=1}^n z_j^2\right|^2
	\le
	\frac{2n(n^2-1)}{n^2+3} \sum_{j=1}^n |z_j|^4.
	\]
	Moreover, the constant $\frac{2n(n^2-1)}{n^2+3}$ is best possible.
\end{thm}


Assuming Theorem~\ref{thm:complex-correct11}, we can now derive Conjecture~\ref{conj1:p=4}.

\begin{proof}[Proof of Conjecture~\ref{conj1:p=4} assuming Theorem~\ref{thm:complex-correct11}]
By the Kushel--Tyaglov inequality~\eqref{eq:kt} and Theorem~\ref{thm:complex-correct11}, we have
\begin{align*}
	\sum_{k=1}^{n-1} |w_k|^4 
	&\le 
	\frac{n-4}{n} \sum_{j=1}^{n} |z_j|^4 
	+ \frac{1}{n^2} 
	\left( \sum_{j=1}^{n} |z_j|^2 \right)^2 
	+ \frac{1}{n^2} 
	\left| \sum_{j=1}^{n} z_j^2 \right|^2 \\
	&\le 
	\left( \frac{n-4}{n}
	+ \frac{2(n^2-1)}{n(n^2+3)}
	\right)
	\sum_{j=1}^{n} |z_j|^4
	= 
	\frac{n^3 - 2n^2 + 3n - 14}{n(n^2 + 3)} 
	\sum_{j=1}^n |z_j|^4,
\end{align*}
which is exactly the inequality asserted in Conjecture~\ref{conj1:p=4}. 
\end{proof}

Before proving Theorem~\ref{thm:complex-correct11}, we establish the following sharp real inequality.
\begin{lem}\label{lem:real1}
	Let $n>4$ be an odd integer. For $x_1,\dots,x_n \in\mathbb{R}$ satisfying $\sum_{j=1}^n x_j=0$, one has
	\[
	\left(\sum_{j=1}^n x_j^2\right)^2 \le \frac{n(n^2-1)}{n^2+3}\sum_{j=1}^n x_j^4.
	\]
	Moreover, the constant is best possible, with equality if 
	\[
	x_1=\cdots=x_{\frac{n-1}{2}}=a,\qquad 
	x_{\frac{n+1}{2}}=\cdots=x_n=-\frac{n-1}{n+1}\,a,
	\]
	for some $a\neq 0$.
\end{lem}

\begin{proof}
	By homogeneity, it suffices to prove that the minimum of $S_4(x):=\sum_{j=1}^n x_j^4$
	on the feasible set
	\[
	\mathcal{M}:=\left\{x\in\mathbb{R}^n:\ S_1(x):=\sum_{j=1}^n x_j=0,\ 
	S_2(x):=\sum_{j=1}^n x_j^2=1\right\}
	\]
	equals $\frac{n^2+3}{n(n^2-1)}$. The set $\mathcal{M}$ is the intersection of the unit sphere with
	an affine hyperplane; hence it is compact. Since $S_4$ is continuous, a minimizer $x^\star\in\mathcal{M}$ exists.
	
	At any feasible $x \in \mathcal{M}$, the constraint gradients are
	\[
	\nabla S_1(x)=(1,\dots,1),\qquad \nabla S_2(x)=2(x_1,\dots,x_n).
	\]
	Let $a\nabla S_1(x) + b\nabla S_2(x)=0$, then $a+2b x_i=0$ for all $1 \leq i \leq n$. If $b\neq0$ we obtain that $x_i$ is constant. But $\sum_{j=1}^n x_j=0$ forces these constants to be $0$, which contradicts the condition
	$\sum_{j=1}^n x_j^2=1$. Hence the gradients are linearly independent and the linear independence constraint qualification (LICQ) of \cite[Definition~12.4]{NocedalWright2006} holds at all feasible points $x \in \mathcal{M}$. Thus we are now able to use the first-order necessary conditions. By \cite[Theorem~12.1]{NocedalWright2006} there exist Lagrange multipliers $(\lambda,\mu)\in\mathbb{R}^2$ such that the Lagrangian
	\[
	\mathcal{L}(x,\lambda,\mu):=S_4(x)-\lambda\bigl(S_2(x)-1\bigr)-\mu S_1(x)
	\]
	satisfies
	\begin{equation}\label{eq:KKT-Lemma54}
		\nabla_x \mathcal{L}(x^\star,\lambda,\mu)=0
		\quad\Longleftrightarrow\quad
		4x_i^{\star3}-2\lambda x_i^\star-\mu=0,\qquad i=1,\dots,n.
	\end{equation}In particular, each coordinate $x_i^\star$ is one of real roots of the cubic
\[
f(t):=4t^3-2\lambda t-\mu,
\]
which has at most three real roots. Hence the vector $x^\star$ assumes at most three distinct values among its coordinates.


	By \cite[Eq.~(12.53)]{NocedalWright2006}, the critical cone at $(x^\star,\lambda,\mu)$ is
	\[
	\mathcal{C}(x^\star,\lambda)
	=\{v\in\mathbb{R}^n: \nabla S_1(x^\star)^\top v=0, \nabla S_2(x^\star)^\top v=0\}
	=\left\{v\in\mathbb{R}^n: \sum_{i=1}^n v_i=0, \sum_{i=1}^n x_i^\star v_i=0\right\}.
	\]
	The Hessian of the Lagrangian is diagonal:
	\[
	\nabla_{xx}^2 \mathcal{L}(x^\star,\lambda,\mu)=\mathrm{diag}\bigl(12x_1^{\star2}-2\lambda,\dots,12x_n^{\star2}-2\lambda\bigr).
	\]
	Now we use the second-order necessary conditions. By \cite[Theorem~12.5]{NocedalWright2006},
	\begin{equation}\label{eq:2nd-necessary}
		v^{\top}\nabla_{xx}^2 \mathcal{L}(x^\star,\lambda,\mu) v
		=\sum_{i=1}^n \left(12x_i^{\star2}-2\lambda\right) v_i^2 \ge 0,
		\qquad\forall\, v\in\mathcal{C}(x^\star,\lambda).
	\end{equation} Recall that the coordinates of $x^\star$ can take at most three distinct values. We now rule out the possibility of three-valued minimizers. Arguing by contradiction, suppose that $x^\star$ takes exactly three distinct values $a>b>c$, with multiplicities $p,q,r$ (so that $p+q+r=n$). The cubic polynomial
	$f(t)=4(t-a)(t-b)(t-c)$, so
	\[
	f'(a)=4(a-b)(a-c)>0,\quad f'(b)=4(b-a)(b-c)<0,\quad f'(c)=4(c-a)(c-b)>0,
	\]
	and $12x_i^{\star2}-2\lambda=f'(x_i^\star)$.
	
	\emph{Case 1: $q\ge2$.} Pick distinct indices \(i_1\neq i_2\) with \(x_{i_1}^\star=x_{i_2}^\star=b\), and define \(v\in\mathbb{R}^n\) by
	\[
	v_{i_1}=t,\qquad v_{i_2}=-t,\qquad v_j=0\ \text{for } j\notin\{i_1,i_2\},
	\]
	where \(t\in\mathbb{R}\). Then $\sum_{i=1}^n v_i=0$ and
	$\sum_{i=1}^n x_i^\star v_i=b(t-t)=0$, hence $v\in\mathcal{C}(x^\star,\lambda)$ and
	\[
	v^{\top}\nabla_{xx}^2 \mathcal{L}(x^\star,\lambda,\mu) v =2 f'(b) t^2<0,
	\]
	contradicting \eqref{eq:2nd-necessary}. Thus $q\le1$.
	
	\emph{Case 2: $q=1$.} Without loss of generality, assume $p \geq 2$, since $p+q+r=n>4$. Pick two distinct indices \(i_1\neq i_2\) corresponding to the entries with value \(a\), and one index \(i_3\) with value \(b\) and another index \(i_4\) with value \(c\). Define \(v\in\mathbb{R}^n\) by
	\[
	v_{i_1}=\alpha,\quad 
	v_{i_2}=\alpha,\quad 
	v_{i_3}=-\theta,\quad 
	v_{i_4}=-\kappa,\quad 
	v_j=0\ \text{for } j\notin\{i_1,i_2,i_3,i_4\},
	\]
	where
	\[
	\theta=\frac{2(a-c)}{b-c} \alpha,
	\qquad 
	\kappa=\frac{2(b-a)}{b-c} \alpha,
	\qquad 
	\alpha\in\mathbb{R}.
	\]Then $\sum_{i=1}^n v_i=0$ and $\sum_{i=1}^n x_i^\star v_i=0$, so $v\in\mathcal{C}(x^\star,\lambda)$, and
	a direct computation yields
	\[
	v^{\top}\nabla_{xx}^2 \mathcal{L}(x^\star,\lambda,\mu) v
	=2f'(a)\alpha^2+f'(b)\theta^2+f'(c)\kappa^2
	=-8(a-b)(a-c)\,\alpha^2<0,
	\]
	again contradicting \eqref{eq:2nd-necessary}. Hence \(q = 0\); in particular, a minimizer cannot be three-valued.
	
	\smallskip
	We conclude that the coordinates of a minimizer $x^\star$ take exactly two distinct values: $a>0$ with multiplicity $p$ and $-d<0$ with multiplicity $r$, where $p+r=n$ and, from the constraints,
	\[
	pa=rd,\ \ pa^2+rd^2=1
	\ \Longrightarrow\ 
	a^2=\frac{r}{pn},\ \ d^2=\frac{p}{rn}.
	\]
	Hence
	\[
	S_4(x^\star)=pa^4+rd^4
	=\frac{p^3+r^3}{n^2pr}
	=\frac{n^3-3npr}{n^2pr}
	=\frac{n}{pr}-\frac{3}{n}.
	\]
	For fixed $n$, the last expression is minimized when $pr$ is maximized, that is, when $p$ and $r$ are
	as equal as possible. Since $n$ is odd, the maximum is $pr=(n^2-1)/4$, attained at
	$p=(n-1)/2$, $r=(n+1)/2$. Therefore
	\[
	\min_{\mathcal{M}} S_4=\frac{n}{(n^2-1)/4}-\frac{3}{n}
	=\frac{n^2+3}{n(n^2-1)}.
	\]
	This proves the inequality in Lemma~\ref{lem:real1}. The equality case corresponds to
	the two-point configuration with $p=(n-1)/2$, $r=(n+1)/2$ and $pa=rd$, which is exactly the one stated in the lemma. This completes the proof.
\end{proof}

We are now ready to present the following.

\begin{proof}[Proof of Theorem~\ref{thm:complex-correct11}]
	Write $z_j=x_j+iy_j$ with $x_j,y_j\in\mathbb{R}$, so $\sum_{j=1}^n x_j=\sum_{j=1}^n y_j=0$. After multiplying all $z_j$ by a unimodular constant we may assume $T:=\sum_{j=1}^n z_j^2$ is real. Let $A:=\sum_{j=1}^n x_j^2$ and $B:=\sum_{j=1}^n y_j^2$. Then
	\[
	\sum_{j=1}^n|z_j|^2=A+B,\qquad \sum_{j=1}^n z_j^2=(A-B)+2i\sum_{j=1}^n x_jy_j,
	\]
	and since $T$ is real we have $\sum_{j=1}^n x_jy_j=0$, hence
	\[
	\left(\sum_{j=1}^n|z_j|^2\right)^2+\left|\sum_{j=1}^n z_j^2\right|^2
	=(A+B)^2+(A-B)^2=2(A^2+B^2).
	\]
	Applying Lemma~\ref{lem:real1} separately to $(x_j)$ and $(y_j)$ gives $A^2\le \tfrac{n(n^2-1)}{n^2+3}\sum_{j=1}^n x_j^4$ and $B^2\le \tfrac{n(n^2-1)}{n^2+3}\sum_{j=1}^n y_j^4$. Therefore
	\[
	\left(\sum_{j=1}^n|z_j|^2\right)^2+\left|\sum_{j=1}^n z_j^2\right|^2
	\le \frac{2n(n^2-1)}{n^2+3}\left(\sum_{j=1}^n x_j^4+\sum_{j=1}^n y_j^4\right)
	\le \frac{2n(n^2-1)}{n^2+3}\sum_{j=1}^n |z_j|^4,
	\]
	as claimed. The sharpness follows from the equality case of Lemma~\ref{lem:real1}, which completes the proof.
\end{proof}

We now turn to Conjecture~\ref{conj2:p=4}.

\begin{thm}\label{thm:complex-correct22}
	Let $n\geq 4$ and let $z_1,\dots,z_n\in\mathbb{C}$ satisfy $\sum_{j=1}^n z_j=0$. Then
	\[
	\sum_{j=1}^n |z_j|^4 \le \frac{n^2-3n+3}{n^2-n}\left(\sum_{j=1}^n |z_j|^2\right)^2.
	\]
	Moreover, the constant $\frac{n^2-3n+3}{n^2-n}$ is best possible.
\end{thm}

Assuming Theorem~\ref{thm:complex-correct22}, we can now derive Conjecture~\ref{conj2:p=4}.

\begin{proof}[Proof of Conjecture~\ref{conj2:p=4} assuming Theorem~\ref{thm:complex-correct22}]
By the Cheung--Ng inequality~\eqref{eq:dbs} and Theorem~\ref{thm:complex-correct22}, we have
\begin{align*}
  \sum_{k=1}^{n-1} |w_k|^4 
  &\le 
  \frac{n-4}{n} \sum_{j=1}^{n} |z_j|^4 
  + \frac{2}{n^2} 
    \left( \sum_{j=1}^{n} |z_j|^2 \right)^2 \\
  &\le 
  \left( \frac{n-4}{n}\cdot\frac{n^2-3n+3}{n^2-n}
  + \frac{2}{n^2}
  \right)
  \left( \sum_{j=1}^{n} |z_j|^2 \right)^2 \\
  &=
  \frac{(n-2)^4+n-2}{n^2(n-1)^2} 
  \left( \sum_{j=1}^{n} |z_j|^2 \right)^2,
\end{align*}
which is exactly the inequality asserted in Conjecture~\ref{conj2:p=4}.
\end{proof}




Before proving Theorem~\ref{thm:complex-correct22}, we establish the following sharp real inequality. The proof of Lemma~\ref{lem:real2} can also be obtained via the Lagrange multiplier method, in the same spirit as the proof of Lemma~\ref{lem:real1}. However, we present below a short and elegant argument based on an explicit sum-of-squares identity, which is of independent interest.

\begin{lem}\label{lem:real2}
	Let $n\ge4$ and let $a_1,\dots,a_n\in\mathbb{R}$ satisfy $\sum_{j=1}^n a_j=0$.
	Then
	\[
	\sum_{j=1}^n a_j^4 \le 
	\frac{n^2-3n+3}{n^2-n}
	\left(\sum_{j=1}^n a_j^2\right)^2.
	\]
	Moreover, the constant is best possible, with equality if
	\[
	a_1=\cdots=a_{n-1}=a,\qquad 
	a_n=-(n-1)a,
	\]
	for some $a\ne0$.
\end{lem}

\begin{proof}
	Since $\sum_{i=1}^{n-1}a_i=-a_n$, we know that
	\[
	\sum_{j=1}^n a_j^2=\sum_{i=1}^{n-1}a_i^2+\left(\sum_{i=1}^{n-1}a_i\right)^2,\qquad
	\sum_{j=1}^n a_j^4=\sum_{i=1}^{n-1}a_i^4+\left(\sum_{i=1}^{n-1}a_i\right)^4.
	\]
	We claim the following identity:
	\begin{equation}\label{eq:sos-identity}
		\begin{aligned}
			&\frac{n^2-3n+3}{n^2-n}\left(\sum_{i=1}^{n-1}a_i^2+\left(\sum_{i=1}^{n-1}a_i\right)^2\right)^2
			-\left(\sum_{i=1}^{n-1}a_i^4+\left(\sum_{i=1}^{n-1}a_i\right)^4\right)\\
			&\quad=\frac{2}{n(n-1)(n-3)}\sum_{1\le i<j\le n-1}(a_i-a_j)^2\biggl[
			\bigl(a_i+a_j+(2-n)\sum_{i=1}^{n-1}a_i\bigr)^2\\
			&\hspace{5.0cm}+\frac{n-1}{2}\Bigl((n-3) \sum_{\substack{1\le k\le n-1\\ k\ne i,j}}a_k^2
			-\Bigl(\sum_{\substack{1\le k\le n-1\\ k\ne i,j}} a_k\Bigr)^{2}\Bigr)\biggr].
		\end{aligned}
	\end{equation}
	A straightforward verification by comparing coefficients is given in Appendix~\ref{app:identity-proof}.
	
	Since, for each fixed pair $(i,j)$, the Cauchy--Schwarz inequality applied to the $(n-3)$ numbers $\{a_k\}_{k\ne i,j}$ yields
	\[
	(n-3)\sum_{\substack{1\le k\le n-1\\ k\ne i,j}} a_k^2 \ge \Bigl(\sum_{\substack{1\le k\le n-1\\ k\ne i,j}} a_k\Bigr)^2.
	\]Moreover, the outer
factor $\tfrac{2}{n(n-1)(n-3)}$ is positive for $n\ge4$. Therefore the right-hand
	side of
	\eqref{eq:sos-identity} is nonnegative, proving the desired inequality in Lemma~\ref{lem:real2}. The equality case can be verified by a direct computation, completing the proof.
\end{proof}

We are now ready to present the following.

\begin{proof}[Proof of Theorem~\ref{thm:complex-correct22}]
	By homogeneity it suffices to prove the inequality under the normalization
	$\sum_{j=1}^n |z_j|^2=1$; the general case then follows by scaling.

	Write $z_k=x_k+iy_k$ and $r_k^2:=x_k^2+y_k^2$, then the condition implies $\sum_{k=1}^n x_k=\sum_{k=1}^n y_k=0$ and $\sum_{k=1}^n r_k^2=1$. Define $G(x,y):=-\sum_{k=1}^n r_k^4$, $g_1(x,y):=\sum_{k=1}^n x_k$, $g_2(x,y):=\sum_{k=1}^n y_k$, and $g_3(x,y):=\sum_{k=1}^n r_k^2-1$. Consider the equality-constrained minimization
	of $G$ on $\mathbb{R}^{2n}$:
	\[
	\min_{(x,y)} G(x,y)
	\quad\text{s.t.}\quad
	g_1(x,y)=0,\ 
	g_2(x,y)=0,\ 
	g_3(x,y)=0.
	\]
	This is of the form \cite[Eq.~(12.1)]{NocedalWright2006}. The feasible set 
	\[
	\mathcal{M}
	=\left\{(x,y)\in\mathbb{R}^{2n}:\ \sum_{j=1}^n x_j=0,\ \sum_{j=1}^n y_j=0,\ \sum_{j=1}^n (x_j^2+y_j^2)=1\right\}
	\]
	is compact and $G$ is
	continuous, so a minimizer $(x^\star,y^\star)$ exists.
	
	At any feasible $(x,y) \in \mathcal{M}$ the constraint gradients are \[
\nabla g_1 = (\underbrace{1,\dots,1}_{n},\underbrace{0,\dots,0}_{n}),
\quad
\nabla g_2 = (\underbrace{0,\dots,0}_{n},\underbrace{1,\dots,1}_{n}),
\quad
\nabla g_3 = 2(x_1,\dots,x_n,y_1,\dots,y_n).
\] Let $a\nabla g_1+b\nabla g_2+c\nabla g_3=0$, then $a+2cx_i=0$ and $b+2cy_i=0$ for all $1\leq i \leq n$. If $c=0$, then $a=b=0$. So we only need to consider the case $c \neq 0$, which means $x_i,y_i$ are constants. But $\sum_{k=1}^n x_k=\sum_{k=1}^n y_k=0$ forces these constants to be $0$, which contradicts the condition $\sum_{j=1}^n (x_j^2+y_j^2)=1$. Hence the gradients are linearly independent and the linear independence constraint qualification (LICQ) of \cite[Definition~12.4]{NocedalWright2006} holds at all feasible points. Thus we are now able to use the first-order necessary conditions. By \cite[Theorem~12.1]{NocedalWright2006} there exist Lagrange multipliers $(\alpha,\beta,\lambda)\in\mathbb{R}^3$
	such that
	\[
	\nabla G(x^\star,y^\star)+\alpha\nabla g_1(x^\star,y^\star)+\beta\nabla g_2(x^\star,y^\star)+\lambda\nabla g_3(x^\star,y^\star)=0.
	\]
	Computing the partial derivatives gives, for each $1 \leq k \leq n$,
	\[
	4r_k^2 x_k^\star-2\lambda x_k^\star-\alpha=0,\qquad
	4r_k^2 y_k^\star-2\lambda y_k^\star-\beta=0.
	\]
	Let $\mu:=\tfrac12(\alpha+i\beta)$ and $z_k^\star:=x_k^\star+iy_k^\star$. The two real relations combine into the single complex equation
	\begin{equation}\label{eq:LM-complex}
		(2|z_k^\star|^2-\lambda) z_k^\star=\mu,\qquad k=1,\dots,n.
	\end{equation}
	
	We first claim that $\mu\neq 0$. Indeed, if $\mu=0$, then \eqref{eq:LM-complex}
	implies that each nonzero $z_k^\star$ satisfies $|z_k^\star|^2=\lambda/2$, hence all nonzero
	coordinates have the same modulus. Let $m$ denote the number of nonzero coordinates of $z^\star$. Because
	$\sum_{j=1}^n z_j^\star=0$, necessarily $m\ge 2$, and from $\sum_{j=1}^n|z_j^\star|^2=1$
	we get $|z_k^\star|^2=1/m$ for each nonzero index. Therefore
	\[
	\sum_{j=1}^n |z_j^\star|^4 = m\cdot \left(\frac{1}{m}\right)^2 = \frac{1}{m} \le \frac12 .
	\]
	However, the vector $v=(a,\dots,a,-(n-1)a)$ with $a>0$ chosen so that
	$\sum_{j=1}^n|v_j|^2=1$ (namely $a^2=1/(n(n-1))$) is admissible and satisfies
	\[
	\sum_{j=1}^n |v_j|^4
	= (n-1)a^4 + (n-1)^4 a^4
	= \frac{1}{n^2(n-1)} + \frac{(n-1)^2}{n^2}
	> \tfrac12 \quad (\text{since } n\ge 4),
	\]
	contradicting maximality of $\sum_{j=1}^n |z_j^\star|^4$. Hence $\mu\neq 0$.
	
	With $\mu\neq 0$, \eqref{eq:LM-complex} shows that each $z_k^\star$ is a real multiple
	of the same complex number $\mu$: there exist $t_k\in\mathbb{R}$ and a unit
	complex number $u:=\mu/|\mu|$ such that $z_k^\star=t_k u$ for all $k$.
	Using $\sum_{j=1}^n z_j^\star=0$ and $\sum_{j=1}^n|z_j^\star|^2=1$, we obtain
	\[
	\sum_{j=1}^n t_j=0,\qquad \sum_{j=1}^n t_j^2=1,\qquad
	\sum_{j=1}^n |z_j^\star|^4=\sum_{j=1}^n t_j^4 .
	\]
	Applying Lemma~\ref{lem:real2} to the real numbers $t_1,\dots,t_n$ yields
	\[
	\sum_{j=1}^n |z_j^\star|^4
	=\sum_{j=1}^n t_j^4
	\le \frac{n^2-3n+3}{n^2-n}\left(\sum_{j=1}^n t_j^2\right)^2
	=\frac{n^2-3n+3}{n^2-n}.
	\]
	This proves the inequality in Theorem~\ref{thm:complex-correct22}. The sharpness follows from the equality case of Lemma~\ref{lem:real2}, which completes the proof.
\end{proof}

As we have now established Conjecture~\ref{conj1:p=4} and Conjecture~\ref{conj2:p=4}, 
we are able to determine the set $\Omega(n)$ in Problem~\ref{prob:dbs_12345}, 
that is, to provide a complete solution to the 
\href{https://google-deepmind.github.io/alphaevolve_repository_of_problems/problems/24.html}{de~Bruin--Sharma problem}, see Theorem~\ref{thm:solution of de Bruin--Sharma problem}.

\begin{proof}[Proof of Theorem~\ref{thm:solution of de Bruin--Sharma problem}]
\emph{Case 1: $n>4$ odd.}
The set $\Omega(n)$ is clearly closed and convex. By \eqref{eq:dbs} together with
Conjecture~\ref{conj1:p=4} and Conjecture~\ref{conj2:p=4} (now theorems), we have
$\Omega(n)\supset \widetilde{\Omega}_{\mathrm{odd}}(n)$.

We show the reverse inclusion. If $(\alpha,\beta)\notin \widetilde{\Omega}_{\mathrm{odd}}(n)$, then necessarily either
\begin{equation}\label{eq:left-odd}
	\alpha < \alpha_{\mathrm{odd}}(n) - \frac{n(n^2-1)}{n^2+3}\beta
	\quad\text{with } 0\le \beta\le \frac{2}{n^2},
\end{equation}
or
\begin{equation}\label{eq:right-all}
	\alpha < \frac{\beta_{\mathrm{all}}(n)-\beta}{\frac{n^2-3n+3}{n^2-n}}
	\quad\text{with } \frac{2}{n^2}\le \beta\le \beta_{\mathrm{all}}(n).
\end{equation}
Consider the two test polynomials with zeros
\[
p_1:\ z_1=\cdots=z_{\frac{n-1}{2}}=1,\quad
z_{\frac{n+1}{2}}=\cdots=z_n=-\frac{n-1}{n+1},
\]
\[
p_2:\ z_1=\cdots=z_{n-1}=1,\quad z_n=-(n-1).
\]

\noindent\emph{For $p_1$.}
Let $m=\frac{n-1}{2}$ and $q=\frac{n-1}{n+1}$. Then \(p_1(z)=(z-1)^m(z+q)^{m+1}\). A direct computation gives
\[
\left(\sum_{j=1}^n |z_j|^2\right)^{2}=\frac{n^2(n-1)^2}{(n+1)^2},\quad
\sum_{j=1}^n |z_j|^4
=\frac{n(n-1)(n^2+3)}{(n+1)^3},\quad \sum_{k=1}^{n-1}|\xi_k|^4
=(m-1)+m\cdot q^4+\left(\frac{2}{n+1}\right)^{4}.
\]
Substituting these into \eqref{eq:dbs1212} gives, after dividing by $\sum|z_j|^4$,
\[
\alpha \ge \alpha_{\mathrm{odd}}(n)-\frac{n(n^2-1)}{n^2+3}\beta.
\]
With $0\le \beta\le 2/n^2$ on this branch, this contradicts \eqref{eq:left-odd}.

\medskip
\noindent\emph{For $p_2$.}
Since \(p_2(z)=(z-1)^{n-1}(z+(n-1))\), a direct computation gives
\[
\left(\sum_{j=1}^n |z_j|^2\right)^2=n^2(n-1)^2,\quad
\sum_{j=1}^n |z_j|^4=(n-1)+(n-1)^4,\quad \sum_{k=1}^{n-1}|\xi_k|^4=(n-2)+(n-2)^4.
\]Plugging into \eqref{eq:dbs1212} and dividing by $\sum|z_j|^4$ yields
\[
\alpha \ge \frac{\beta_{\mathrm{all}}(n)-\beta}{\frac{n^2-3n+3}{n^2-n}}.
\]With $2/n^2\le\beta\le \beta_{\mathrm{all}}(n)$ on this branch, this contradicts \eqref{eq:right-all}. Hence no $(\alpha,\beta)\notin \widetilde{\Omega}_{\mathrm{odd}}(n)$ can belong to $\Omega(n)$, and we conclude
$\Omega(n)=\widetilde{\Omega}_{\mathrm{odd}}(n)$.

\medskip
\emph{Case 2: $n\ge 4$ even.}
Again $\Omega(n)$ is closed and convex. By applying the Cauchy--Schwarz inequality to \eqref{eq:dbs} and combining it with Conjecture~\ref{conj2:p=4} (which is now a theorem), we obtain $\Omega(n)\supset \widetilde{\Omega}_{\mathrm{even}}(n)$.

For the reverse inclusion, if $(\alpha,\beta)\notin \widetilde{\Omega}_{\mathrm{even}}(n)$, then either
\begin{equation}\label{eq:left-even}
	\alpha + n \beta  <  \frac{n-2}{n}
	\quad\text{with } 0\le \beta\le \frac{2}{n^2},
\end{equation}
or \eqref{eq:right-all} holds. Consider the test polynomials with zeros
\[
\textstyle
p_3:\ z_1=\cdots=z_{n/2}=1,\quad
z_{n/2+1}=\cdots=z_n=-1,
\]
\[
p_2:\ z_1=\cdots=z_{n-1}=1,\quad z_n=-(n-1).
\]

\noindent\emph{For $p_3$.} Since $p_3(z)=(z-1)^{\frac{n}{2}}(z+1)^{\frac{n}{2}}$, a direct calculation gives
\[
\left(\sum_{j=1}^n |z_j|^2\right)^2=n^2,\quad
\sum_{j=1}^n |z_j|^4=n,\quad \sum_{k=1}^{n-1}|\xi_k|^4=n-2.
\]
Plugging these values into \eqref{eq:dbs1212} yields
\[
\alpha + n \beta \ge \frac{n-2}{n},
\]
With $0\le \beta\le 2/n^2$ on this branch, this contradicts \eqref{eq:left-even}.

As in Case~1, we next consider the polynomial $p_2$, which similarly leads to a contradiction with~\eqref{eq:right-all}. 
Hence no $(\alpha,\beta)\notin \widetilde{\Omega}_{\mathrm{even}}(n)$ can belong to $\Omega(n)$, and we conclude that $\Omega(n)=\widetilde{\Omega}_{\mathrm{even}}(n)$. This completes the proof.
\end{proof}

\section{Notation and tools for the sequel}\label{sec:tools}

In this section, we introduce the notation and review several basic results from matrix analysis and complex interpolation theory that will be used in the proofs of our main results.

\subsection{Matrix analysis}
Let \(A\) be an \(n \times n\) complex matrix. The eigenvalues and singular values of \(A\) are denoted by \(\lambda_j(A)\) and \(\sigma_j(A)\), respectively, for \(j = 1,\ldots,n\), arranged in nonincreasing order of their moduli. 
We write \(\operatorname{tr}(A)\) for the trace of \(A\), and denote by \(I\) the identity matrix and by \(J\) the all-ones matrix. 
We also write \(\mathbf{1} = (1,\dots,1)^{\top} \in \mathbb{C}^n\) for the all-ones column vector. For a (not necessarily square) \(m \times n\) complex matrix \(X\), we write \(X^*\) for its conjugate transpose and set \(|X| := (X^* X)^{1/2}\), the positive semidefinite square root of \(X^* X\). For \(a_1,\dots,a_n \in \mathbb{C}\) we denote by \(\mathrm{diag}(a_1,\dots,a_n)\) the diagonal matrix with diagonal entries \(a_1,\dots,a_n\).  
If \(a = (a_1,\dots,a_n)^{\top} \in \mathbb{C}^n\), we also write \(\mathrm{diag}(a)\) for \(\mathrm{diag}(a_1,\dots,a_n)\).

The following spaces and norm definitions are standard in matrix analysis; see~\cite{Bhatia13, HJ13, HJ91}.

\begin{defn}
Let $V \subset \mathbb{C}^n$ be a linear subspace. 
\begin{itemize}
 \item For $1 \le p \le \infty$ and $z=(z_1,\ldots,z_n)^\top \in V$, we define $\ell^p(V)$ as the space $V$ equipped with the \emph{$\ell^p$--norm}
    \[
        \|z\|_{\ell^p} = \left( \sum_{j=1}^n |z_j|^p \right)^{1/p}.
    \]
    \item For $1 \leq p \leq \infty$, let $S_p$ denote the Schatten von Neumann trace class, consisting of compact linear operators $A: \ell^2 \to \ell^2$, where $\ell^2$ denotes the Hilbert space of all square-summable complex sequences. This space is endowed with the \emph{Schatten $p$-norm}, defined by
$$\|A\|_{S_p} = \left( \sum_j \sigma_j(A)^p \right)^{1/p}.$$
\end{itemize}
\end{defn}

The following observation, due to Komarova and Rivin~\cite[Lemma~5.7]{KR03}, plays an important role in our derivation.

\begin{lem}[\cite{KR03}]\label{lem: zp'(z)}
Let \( n \geq 2 \) and \( z_1, z_2, \ldots, z_n \in \mathbb{C} \). Let \( p(z) = \prod_{j=1}^n (z - z_j) \), and define \( D = \operatorname{diag}(z_1, z_2, \ldots, z_n) \). Then the characteristic polynomial of the matrix \( D \left( I - \frac{1}{n} J \right) \) has the same set of zeros as the polynomial \( z p'(z) \). 
\end{lem}

\begin{remark}
Clearly, \(I - \tfrac{1}{n} J\) is a projection matrix. Since \(AB\) and \(BA\) have the same spectrum for any \(n\times n\) complex matrices \(A,B\), all eigenvalues of
\(D \left(I - \tfrac{1}{n} J\right)\) coincide with those of \( \left(I - \tfrac{1}{n} J\right) D \left(I - \tfrac{1}{n} J\right)\).
\end{remark}


The following lemma comes from Cheung and Ng ~\cite[Theorem~1.1]{CN06}.
\begin{lem}[\cite{CN06}]\label{lem:CN-matrix-model}
	Let $p(z)=\prod_{j=1}^{n}(z-z_j)$ be a polynomial of degree $m\ge2$ with complex zeros
	$z_1,\dots,z_n$ (listed with multiplicities).  
	Let
	\[
	D=\operatorname{diag}(z_1,\dots,z_{n-1}),
	\]
	and let $I$ and $J$ be the identity matrix and the all--ones matrix
	of order $n-1$, respectively.  Then the set of all eigenvalues of the
	$(n-1)\times(n-1)$ matrix
	\[
	D\left(I-\frac1n J\right)+\frac{z_n}{n}J
	\]
	is the same as the set of all critical points of the polynomial $p(z)$.
\end{lem}

The following lemma is standard in matrix analysis.
\begin{lem}[{\cite[p.~176]{HJ91}}]\label{lem:weyl}  Let $A$ be any $n\times n$ complex matrix. Then for any $r>0$, $$\sum_{j=1}^n|\lambda_j(A)|^r\le \sum_{j=1}^n\sigma_j^r(A)=\tr \left|A\right|^r,$$ with equality if and only if $A$ is normal \cite[p.~455]{HJ13}. The case $r=2$ of the inequality is called Schur's inequality~\cite[p.~102]{HJ13}. 
\end{lem}

We will also need the following comparison for singular values.

\begin{lem}\label{lem:XDx_vs_absD}
Let \(D=\mathrm{diag}(d_1,\dots,d_n)\) be diagonal and let \(X\in\mathbb{C}^{n\times n}\).
Then, for every \(k=1,\dots,n\),
\[
\prod_{j=1}^k \sigma_j\left(X^*  D X\right)
 \le
\prod_{j=1}^k \sigma_j\left(X^*  |D| X\right),
\]
where singular values are arranged in nonincreasing order.
\end{lem}

\begin{proof}
Consider the \(2n\times 2n\) block matrix
\[
C
=\begin{pmatrix}
X^* |D|X & X^*  D X\\
X^*  D^*  X & X^* |D|X
\end{pmatrix}
=\begin{pmatrix}X^* &0\\0&X^* \end{pmatrix}
\begin{pmatrix}
|D| & D\\ D^*  & |D|
\end{pmatrix}
\begin{pmatrix}X&0\\0&X\end{pmatrix}.
\]
By \cite[Corollary~1.3.8]{Bhatia09}, we know that
\(
\begin{pmatrix}|D| & D\\ D^*  & |D|\end{pmatrix}
\) is positive semidefinite, thus $C$ is positive semidefinite. Let \(A=X^* |D|X\) and \(B=X^* DX\). By \cite[Proposition~1.3.2]{Bhatia09} (see also \cite[Lemma~3.5.12]{HJ91}), there exists a contraction \(K\) (that is, $\sigma_1(K) \leq 1$) such that \(B = A^{1/2} K A^{1/2}\). Therefore, in the same way as in \cite[p.~208]{HJ91}, for \(k=1,2,\ldots,n\), we have
\begin{align*}
\prod_{j=1}^k \sigma_j(B)
&= \prod_{j=1}^k \sigma_j\left(A^{1/2} K A^{1/2}\right) \le \prod_{j=1}^k \sigma_j\left(A^{1/2}\right)\sigma_j(K)\sigma_j \left(A^{1/2}\right)
\\
&\le \prod_{j=1}^k \sigma_j(A)^{1/2}\sigma_j(A)^{1/2}
= \prod_{j=1}^k \sigma_j(A).
\end{align*}
This completes the proof.\end{proof}

\subsection{Complex interpolation theory} We now introduce the tools from complex interpolation needed in this paper; for the basic theory of interpolation, we refer the interested reader to \cite[Chapters~1--5]{Bergh76}.

\begin{defn}
Let $A_0$ and $A_1$ be two topological vector spaces. We say that $A_0$ and $A_1$ are \emph{compatible} 
if there exists a Hausdorff topological vector space $\mathcal{Z}$ such that $A_0$ and $A_1$ are subspaces of $\mathcal{Z}$.  
In this case, we can form their sum $A_0 + A_1$ and their intersection $A_0 \cap A_1$.
\end{defn}

\begin{defn}
Let $X$ and $Y$ be normed spaces, and let $T : X \to Y$ be a bounded linear operator.  
The \emph{operator norm} of $T$ is defined by
\[
\|T\|_{X \to Y} := \sup_{\substack{x \in X \\ x \neq 0}} \frac{\|Tx\|_Y}{\|x\|_X}.
\]
\end{defn}

\begin{defn}
Let $(X_0, X_1)$ be a compatible Banach couple and $0 < \theta < 1$.  
The space $(X_0, X_1)_{[\theta]}$ is defined as
\[
(X_0, X_1)_{[\theta]} := \{ f(\theta) : f \in \mathcal{F}(X_0, X_1) \},
\]
where $\mathcal{F}(X_0, X_1)$ denotes the space of all bounded continuous functions 
$f : \{ z \in \mathbb{C} : 0 \le \Re z \le 1 \} \to X_0 + X_1$ 
which are analytic on $\{ z \in \mathbb{C} : 0 < \Re z < 1 \}$, 
satisfy $f(it) \in X_0$ for all $t \in \mathbb{R}$ and $f(1+it) \in X_1$ for all $t \in \mathbb{R}$,  
and are bounded on the boundary lines in the respective norms of $X_0$ and $X_1$.  
The norm on $(X_0, X_1)_{[\theta]}$ is given by
\[
\|x\|_{(X_0, X_1)_{[\theta]}} := \inf\big\{\, \|f\|_{\mathcal{F}(X_0, X_1)} : f(\theta) = x \,\big\}.
\]
\end{defn}

We recall the abstract complex interpolation theorem for Banach couples; see \cite[Theorem~4.1.2]{Bergh76}.

\begin{thm}[\cite{Bergh76}]\label{thm:complex-interpolation}
Let $(X_0, X_1)$ and $(Y_0, Y_1)$ be two compatible couples of Banach spaces, and let $0 < \theta < 1$.  
Suppose that 
\[
T : X_0 + X_1 \longrightarrow Y_0 + Y_1
\]
is a linear operator such that $T : X_j \to Y_j$ is bounded for $j = 0,1$.  
Then $T$ is bounded from $(X_0, X_1)_{[\theta]}$ to $(Y_0, Y_1)_{[\theta]}$, and 
\[
\|T\|_{\theta} \le \|T\|_{X_0 \to Y_0}^{\,1-\theta} \, \|T\|_{X_1 \to Y_1}^{\,\theta}.
\]
\end{thm}

We also need the following Riesz--Thorin interpolation theorem for \(L^p\) spaces; see \cite[Theorem~5.1.1]{Bergh76}.

\begin{thm}[\cite{Bergh76}]\label{thm:RT_interpolation1}
Assume that \(p_0 \ge 1\), \(p_1 \ge 1\) and \(0 < \theta < 1\). Then
\[ 
(L^{p_0}, L^{p_1})_{[\theta]} = L^p \quad \text{(with equal norms)},
\] 
if
\[
\frac{1}{p} = \frac{1-\theta}{p_0} + \frac{\theta}{p_1}.
\]
\end{thm}

\begin{remark}\label{rmk:RT_interpolation1}
For certain noncommutative \(L^p\) spaces, such as the Schatten \(p\)-classes \(S_p\), a Riesz--Thorin interpolation theorem is also known to hold; see \cite[Eq.~(2.1)]{Pisier03}, since matrix algebras are semifinite von Neumann algebras.  
See also \cite[Theorem~13.1]{GohbergKrein69}, \cite[Theorems~2.9--2.10]{Simon05}, and \cite[Remark~1, p.~23]{Simon05}.
\end{remark}

\section{Revisiting  the Schoenberg type inequality of order 1 and Pereira's result}\label{sec:new proof}

In this section, we first give a new proof of the Schoenberg type inequality of order~1, namely~\eqref{eq:p=1}. We then use complex interpolation to provide a new proof of Pereira's result~\eqref{eq:Pereira1}, together with a partial refinement under the additional assumption that the centroid of the zeros is at the origin.


\subsection{A new proof of the Schoenberg type inequality of order~1}\label{sec:p=1}
In~\cite[Theorem~3.4]{Tang25}, Tang established the Schoenberg type inequality of order~1, namely~\eqref{eq:p=1}. 
In this subsection we obtain a strengthened version of this inequality, from which a new proof of~\eqref{eq:p=1} follows.


\begin{thm}\label{thm:p1}
Let \(n \ge 2\) and \(z = (z_1,\ldots,z_n)^{\top} \in \mathbb{C}^n\) satisfy \(\sum_{j=1}^n z_j = 0\). Set \(Q = I - \frac{1}{n}J\), and define $T_0(z) = Q \mathrm{diag}(z) Q$. Then
\[
\|T_0(z)\|_{S_1} \le \sqrt{\frac{n-2}{n}} \|z\|_{\ell^1}.
\]
\end{thm}

Assuming Theorem~\ref{thm:p1}, we can now derive inequality~\eqref{eq:p=1}.

\begin{proof}[Proof of inequality~\eqref{eq:p=1} assuming Theorem~\ref{thm:p1}]
Let \( z_1, z_2, \ldots, z_n \) be the zeros of \( p(z) \), and let \( w_1, w_2, \ldots, w_{n-1} \) be its critical points. Set $Q:=I-\frac1nJ$, $D:=\mathrm{diag}(z_1,\dots,z_n)$ and $A:=QDQ$. It follows from Lemma~\ref{lem: zp'(z)} that the spectrum of \(A\) is 
\begin{equation}\label{eq:specA}
\operatorname{spec}(A)=\{w_1,\dots,w_{n-1},0\}.
\end{equation} By Lemma~\ref{lem:weyl}, for every $p\ge1$, $\sum_{i=1}^n |\lambda_i(A)|^p\le\sum_{i=1}^n \sigma_i(A)^p=\|A\|_{S_p}^p$.
Using \eqref{eq:specA} we get
\begin{equation}\label{eq:eval-to-sval}
   \sum_{k=1}^{n-1} |w_k|^p \le \|A\|_{S_p}^p .
\end{equation}
Taking $p=1$ in \eqref{eq:eval-to-sval} and applying Theorem~\ref{thm:p1} to $z=(z_1,\dots,z_n)^\top$ yields
\[
\sum_{k=1}^{n-1} |w_k|
\le \|A\|_{S_1}
\le \sqrt{\frac{n-2}{n}} \|z\|_{\ell^1},
\]
which is precisely inequality~\eqref{eq:p=1}. 
\end{proof}



Before proving Theorem~\ref{thm:p1}, we first establish the following proposition. The proof of Proposition~\ref{prop:elementary_sym1} adapts ideas from the proof of \cite[Theorem~3.1]{Tang25}. For the basic theory of majorization, we refer the reader to~\cite{MOA11}. 

\begin{prop}\label{prop:elementary_sym1}
Let \(D=\mathrm{diag}(z_1,\ldots,z_n)\) with \(n\ge 2\), and set \(A=QDQ\) where \(Q=I-\frac{1}{n}J\). Denote by \(\sigma_1, \ldots, \sigma_{n-1}, 0\) the singular values of \(A\), arranged in nonincreasing order of their moduli. Then
\[
e_k\left(\sigma_1,\ldots,\sigma_{n-1}\right) \le \frac{n-k}{n}  e_k\left(|z_1|,\ldots,|z_n|\right),
\quad k=1,\ldots,n-1.
\]
\end{prop}

\begin{proof}
Let \(q(z)\) be the polynomial with zeros \(|z_j|\), \(j=1,\ldots,n\), and critical points \(\xi_j\), \(j=1,\ldots,n-1\), arranged in nonincreasing order of their moduli. By Lemma~\ref{lem:XDx_vs_absD}, for every \(k=1,\dots,n\) we have
\[
\prod_{j=1}^k \sigma_j \left(Q D Q\right)
 \le
\prod_{j=1}^k \sigma_j \left(Q |D| Q\right).
\]
By Lemma \ref{lem: zp'(z)}, the eigenvalues of \(Q |D| Q\) coincide with its singular values and are given by \(\xi_1,\ldots,\xi_{n-1},0\). Therefore,
\[
\prod_{j=1}^k \sigma_j
 \le
\prod_{j=1}^k \xi_j, \quad k=1,\dots,n-1.
\]

Without loss of generality, we assume \(|\sigma_{n-1}|>0\). The general case follows by a continuity argument. We may choose an entrywise positive vector \(\boldsymbol{x}=(x_1, \ldots, x_{n-1})\) with \(x_j=\xi_j\) for \(j=1, \ldots, n-2\) and \(\prod_{j=1}^{n-1}\sigma_j= \prod_{j=1}^{n-1}x_j\). Then \(\left(\log \sigma_1, \ldots, \log \sigma_{n-1} \right)\) is majorized by \(\left(\log x_1, \ldots, \log x_{n-1}\right)\). By \cite[Theorem~9]{Rov12}, the function \(e_{k}(e^{t_1}, \ldots, e^{t_{n-1}})\) is Schur-convex over \(t_i\in  \mathbb{R}\), hence
\[
e_k \left(\sigma_1,\ldots,\sigma_{n-1}\right)\le e_k \left(x_1, \ldots, x_{n-1}\right).
\]
Since \(e_k:\mathbb{R}^{n-1}_+\to \mathbb{R}\) is monotone increasing in each variable, we have \(e_k(x_1, \ldots, x_{n-1})\le e_k(\xi_1, \ldots, \xi_{n-1})\), as clearly \(x_{n-1}\le \xi_{n-1}\). Now observe that
\[
q(z) = \prod_{j=1}^n \big(z - |z_j|\big) 
= z^n + \sum_{k=1}^n (-1)^k e_k(|\boldsymbol{z}|)\, z^{n-k},
\]
where \(\boldsymbol{z} = (z_1,\ldots,z_n)\). Differentiating gives
\[
q'(z) = nz^{n-1} + \sum_{k=1}^{n-1} (n-k)(-1)^k e_k(|\boldsymbol{z}|) z^{n-1-k}.
\]
On the other hand,
\[
q'(z) = n \prod_{j=1}^{n-1} \big(z - \xi_j\big)
= nz^{n-1} + n \sum_{k=1}^{n-1} (-1)^k e_k(\boldsymbol{\xi})\, z^{\,n-1-k},
\]
where \(\boldsymbol{\xi} = (\xi_1,\ldots,\xi_{n-1})\). Comparing coefficients yields
\[
e_k(\boldsymbol{\xi}) = \frac{n-k}{n} e_k(|\boldsymbol{z}|), \qquad k = 1, \ldots, n-1.
\]
Therefore,
\[
e_k(\sigma_1,\ldots,\sigma_{n-1})  \le e_k(\boldsymbol{\xi}) = \frac{n-k}{n} e_k(|\boldsymbol{z}|), \qquad k = 1, \ldots, n-1.
\]
This completes the proof.
\end{proof}
We are now ready to present the following.
\begin{proof}[Proof of Theorem~\ref{thm:p1}]Denote by \(\sigma_1, \ldots, \sigma_{n-1}, 0\) the singular values of \(T_0(z)\), arranged in nonincreasing order of their moduli. It suffices to prove that
\[
\left(\sum_{k=1}^{n-1}\sigma_k \right)^2  \leq   \frac{n-2}{n}  \left(\sum_{j=1}^n\left|z_j\right|\right)^2,
\]
which expands as
\begin{equation}\label{eq:p1_eqeq1}
\sum_{k=1}^{n-1}\sigma_k^2
+ 2  \sum_{1\leq i<j\leq n-1}\sigma_i \sigma_j
 \leq \frac{n-2}{n} \sum_{k=1}^n\left|z_k\right|^2
+ \frac{2(n-2)}{n}\sum_{1\leq i<j\leq n}\left|z_i z_j\right|.
\end{equation} Since $\sum_{j=1}^n z_j = 0$, it is clear that $J \mathrm{diag}(z) J = 0$. A direct computation using the fact that $Q$ is an orthogonal projection and that, for any diagonal matrix $E$, one has $\operatorname{tr}(EJ) = \operatorname{tr}(E)$, gives
\begin{equation}\label{eq:ppppppppp2}
\begin{aligned}
   \|T_0(z)\|_{S_2}^2
   &= \tr\left(Q \mathrm{diag}(z) Q Q \mathrm{diag}(z)^* Q\right) = \tr\left(Q \mathrm{diag}(z) Q \mathrm{diag}(z)^* \right)\\
   &= \tr\left(|\mathrm{diag}(z)|^2-\frac{1}{n}\mathrm{diag}(z) J \mathrm{diag}(z)^* -\frac{1}{n}J|\mathrm{diag}(z)|^2\right)
     = \frac{n-2}{n} \|z\|_{\ell^2}^2 .
\end{aligned}
\end{equation}From \eqref{eq:ppppppppp2} we have
\begin{equation}\label{eq:p1_eqeq2}
\sum_{k=1}^{n-1}\sigma_k^2
= \frac{n-2}{n} \sum_{k=1}^n\left|z_k\right|^2.
\end{equation}
Moreover, taking the case \(k=2\) in Proposition~\ref{prop:elementary_sym1} yields
\begin{equation}\label{eq:p1_eqeq3}
\sum_{1\leq i<j\leq n-1}\sigma_i \sigma_j
 \leq \frac{n-2}{n}\sum_{1\leq i<j\leq n}\left|z_i z_j\right|.
\end{equation} Combining \eqref{eq:p1_eqeq2} and \eqref{eq:p1_eqeq3} yields \eqref{eq:p1_eqeq1}, thereby completing the proof.
\end{proof}

\begin{remark}
   The Schoenberg type inequality of order $1$ proved in \cite[Theorem~3.4]{Tang25} is in fact an immediate corollary of Malamud~\cite[Theorem~4.10]{Mal04} (more precisely, its special case \cite[Theorem~3.1]{Tang25}) together with the quadratic Schoenberg inequality~\eqref{eq:schoenberg}. In contrast, Theorem~\ref{thm:p1} in the present paper yields a conclusion that goes beyond the scope of Malamud~\cite[Theorem~4.10]{Mal04}.
\end{remark}

\subsection{A new proof of Pereira's result}
Let \( z_1, z_2, \ldots, z_n \) be the zeros of \( p(z) \), and let \( w_1, w_2, \ldots, w_{n-1} \) be its critical points. Set $Q:=I-\frac1nJ$, $D:=\mathrm{diag}(z_1,\dots,z_n)$ and $A:=QDQ$. 


Define the linear map
$T: \mathbb C^{n}\to \mathbb C^{n\times n}$ by $T(z)=Q \mathrm{diag}(z) Q$.
We have two endpoint estimates:

(i) $p=\infty$: since $Q$ is an orthogonal projection, we know that
\begin{equation}\label{eq:pinfty}
   \|T(z)\|_{S_\infty} = \|Q \mathrm{diag}(z) Q\|_{S_\infty} \le \|\mathrm{diag}(z)\|_{S_\infty}=\|z\|_{\ell^\infty}.
\end{equation}

(ii) $p = 2$:  a direct computation using the fact that $Q$ is an orthogonal projection and that, for any diagonal matrix $E$, one has $\operatorname{tr}(EJ) = \operatorname{tr}(E)$, gives\begin{align}\label{eq:p2,n-1/n}
\|T(z)\|_{S_2}^2&= \tr\left(Q \mathrm{diag}(z) Q Q \mathrm{diag}(z)^* Q\right) = \tr\left(Q \mathrm{diag}(z) Q \mathrm{diag}(z)^* \right)\nonumber\nonumber\\
		&=\tr\left(|\mathrm{diag}(z)|^2-\frac{1}{n}\mathrm{diag}(z) J\mathrm{diag}(z)^* -\frac{1}{n}J |\mathrm{diag}(z)|^2+\frac{1}{n^2}J\mathrm{diag}(z)J\mathrm{diag}(z)^*\right)\nonumber\\
		&=\frac{n-2}{n} \|z\|_{\ell^2}^2+\left|\frac{\sum_{j=1}^nz_j}{n}\right|^2\le \frac{n-1}{n} \|z\|_{\ell^2}^2.
\end{align}By the Riesz--Thorin interpolation theorem for $L^p$ spaces (see Theorem~\ref{thm:RT_interpolation1}), we have 
\[
\big(\ell^\infty(\mathbb C^{n}), \ell^2(\mathbb C^{n})\big)_{[\theta]} = \ell^p(\mathbb C^{n}) \quad \text{(with equal norms)},
\]
where $\theta = \frac{2}{p}$ and $p \in (2,\infty)$. Since the Riesz--Thorin interpolation theorem holds for the Schatten \(p\)-classes (see Remark~\ref{rmk:RT_interpolation1}), we also have
\[
\big(S_\infty, S_2\big)_{[\theta]} = S_p \quad \text{(with equal norms)},
\]
with $\theta = \frac{2}{p}$ for $p \in (2,\infty)$.  From \eqref{eq:pinfty}--\eqref{eq:p2,n-1/n}, we apply the abstract complex interpolation theorem (see Theorem~\ref{thm:complex-interpolation}) to the two compatible Banach couples $(\ell^\infty(\mathbb C^{n}),\ell^2(\mathbb C^{n}))$ and $(S_\infty,S_2)$ and obtain\[
\|T\|_{\ell^p(\mathbb C^{n}) \to S_p} \le \|T\|_{\ell^\infty(\mathbb C^{n}) \to S_\infty}^{\,1-\theta} \, \|T\|_{\ell^2(\mathbb C^{n}) \to S_2}^{\,\theta} \leq 1^{1-\frac{2}{p}} \cdot \left(\sqrt{\frac{n-1}{n}}\right)^{\frac{2}{p}} = \left(\frac{n-1}{n}\right)^{\frac{1}{p}}.
\] This means
\[
   \|T(z)\|_{S_p}\le \left(\frac{n-1}{n}\right)^{\frac{1}{p}}\|z\|_{\ell^p}.
\]
Taking $z=(z_1,\dots,z_n)$ and recalling $A=T(z)$ yields
$\|A\|_{S_p}^p\le \frac{n-1}{n}\sum_{j=1}^n|z_j|^p$.
Combine this with \eqref{eq:eval-to-sval} to finish the proof of the case $p\ge 2$.

Similar to the case \(p \ge 2\), we begin by establishing the two endpoint estimates corresponding to $p=1$ and $p=2$, namely
\[
   \|T(z)\|_{S_1}\le c^2\,\|z\|_{\ell^1},\qquad
   \|T(z)\|_{S_2}\le c\,\|z\|_{\ell^2},
   \qquad
   c:=\sqrt{\frac{n-1}{n}},
\]
where the first inequality follows from Proposition~\ref{prop:elementary_sym1} (with \(k=1\)) and the second from~\eqref{eq:p2,n-1/n}. By the Riesz--Thorin interpolation theorem for $L^p$ spaces and for Schatten $p$-classes, we know that for $p \in (1,2)$, we have 
\[
\left(\ell^1(\mathbb C^{n}),\ell^2(\mathbb C^{n})\right)_{[\theta]} = \ell^p(\mathbb C^{n}), 
\qquad 
\left(S_1,S_2\right)_{[\theta]} = S_p \quad \text{(with equal norms)},
\]
where $1/p = (1-\theta)/1 + \theta/2$. Hence, we apply Theorem~\ref{thm:complex-interpolation} to the two compatible Banach couples $(\ell^1(\mathbb C^{n}),\ell^2(\mathbb C^{n}))$ and $(S_1,S_2)$ and obtain\[
   \|T(z)\|_{S_p}\ \le\ c^{\,2-2\theta}\,c^{\,\theta}\ \|z\|_{\ell^p}
   \ =\ c^{2-\theta}\,\|z\|_{\ell^p}
   \ =\ \left(\frac{n-1}{n}\right)^{\frac{1}{p}}\|z\|_{\ell^p}.
\]
Raising both sides to the $p$th power and using the eigenvalue--singular
value comparison \eqref{eq:eval-to-sval} yields
\[
   \sum_{k=1}^{n-1}|w_k|^p
   \ \le\ \|T(z)\|_{S_p}^p
   \ \le\ \frac{n-1}{n}\,\sum_{j=1}^n |z_j|^p,
\]
as desired.\qed

 \subsection{The case where the centroid is at the origin}

Motivated by Problem~\ref{prob:kt} and Pereira's result~\eqref{eq:Pereira1}, it is natural to ask the following question. For a general exponent $p \ge 1$, under the additional assumption that the centroid of the zeros of $p(z)$ is at the origin, can one improve the constant in the inequality relating the $\ell^p$-norm of the critical points to the $\ell^p$-norm of the zeros?  
Inequality~\eqref{eq:p=1} provides the sharp Schoenberg type inequality in the case $p=1$.  
For $1<p<2$, we do not currently have a satisfactory conjecture for the sharp inequality.  
However, for $p \ge 2$ we propose Conjectures~\ref{conj: sum z_j=0} and~\ref{conj:two-point-odd12}.  
The cases of these conjectures that we are presently able to prove are recorded in Remark~\ref{remark:two_conj_solved_case_11}.  
In this subsection we use complex interpolation to obtain a partial refinement of Pereira's result~\eqref{eq:Pereira1} under the additional assumption that the centroid of the zeros is at the origin. 

A natural approach to Conjecture~\ref{conj: sum z_j=0} is to consider the interpolation of the subspace
\[
V := \left\{ z=(z_1,\dots,z_n) \in \mathbb{C}^n : \sum_{j=1}^n z_j = 0 \right\}
\]
of $\mathbb{C}^n$.  
According to Triebel's book~\cite[p.~118, Theorem~1]{Tri78}, if $\{A_0, A_1\}$ is an interpolation couple and $B$ is a complemented subspace of $A_0 + A_1$ with projection $P$ belonging to $L(\{A_0, A_1\}, \{A_0, A_1\})$, and if $F$ is an arbitrary interpolation functor, then $\{A_0 \cap B, A_1 \cap B\}$ is also an interpolation couple and
\[
F\bigl(\{A_0 \cap B, A_1 \cap B\}\bigr)
= F\bigl(\{A_0, A_1\}\bigr) \cap B.
\]
However, this equality holds only up to equivalent norms, rather than with equal norms.  
In other words, $(\ell^\infty(V), \ell^2(V))_{[\theta]}$ coincides with $\ell^p(V)$ as a vector space, and the corresponding norms are equivalent (in general not equal when $n \ge 3$).

As a partial result towards Conjecture~\ref{conj: sum z_j=0}, we can prove the following weaker estimate.




\begin{thm}\label{thm:weaker result for p>= 2}
	For $n \geq 4$, let \( z_1, z_2, \ldots, z_n \) be the zeros of \( p(z) \), and let \( w_1, w_2, \ldots, w_{n-1} \) be its critical points. Assume further that $\sum_{j=1}^n z_j=0$. Then for every $p\ge 2$,
	\[
	\sum_{k=1}^{n-1} |w_k|^p \le \left(\frac{2(n-1)}{n} \right) ^{p-2}  \frac{n-2}{n} \sum_{j=1}^n |z_j|^p .
	\]
\end{thm}
\begin{proof}
	Let $V:=\{z\in\mathbb C^n:\sum_{j=1}^n z_j=0\}$ and define the linear map
	$T_0: V\to \mathbb C^{n\times n}$ by $T_0(z)=Q \mathrm{diag}(z) Q$ and the orthogonal projection: $P_V: \mathbb{C}^n\to V$ by $P_V(z)=Qz$. Let $\widehat{T}:= T_0\circ P_V$ map $\mathbb{C}^n$ to $ \mathbb C^{n\times n}$. Clearly, $\widehat{T}\mid_V=T_0$. Set $G(z)=\frac{1}{n}\sum_{j=1}^n z_j$,
	we have two endpoint estimates:
	
	(i) $p=\infty$: since $Q$ is an orthogonal projection, we know that
	\begin{align}\label{eq:pinfty,0}
		\|\widehat{T}(z)\|_{S_\infty} &= \|Q \mathrm{diag}(z-G(z)\mathbf{1}) Q\|_{S_\infty} \le \|\mathrm{diag}(z-G(z)\mathbf{1})\|_{S_\infty}\nonumber\\
		&=\|z-G(z)\mathbf{1}\|_{\ell^\infty}\le \max_{1\le i\le n} \left\{ \frac{n-1}{n}|z_i|+\frac{1}{n}\sum_{j\neq i}|z_j|  \right\}\le \frac{2(n-1)}{n}\|z\|_{\ell^\infty}.
	\end{align}
	
	(ii) $p = 2$:  A direct computation using the fact that $Q$ is an orthogonal projection and that, for any diagonal matrix $E$, one has $\operatorname{tr}(EJ) = \operatorname{tr}(E)$, gives
	\begin{equation}\label{eq:p2}
		\begin{aligned}
			\|\widehat{T}(z)\|_{S_2}^2
			&= \tr\left(Q \mathrm{diag}(z-G(z)\mathbf{1}) Q Q \mathrm{diag}(z-G(z)\mathbf{1})^* Q\right)\\
			&= \frac{n-2}{n} \|z-G(z)\mathbf{1}\|_{\ell^2}^2=\frac{n-2}{n}\left(\|z\|_{\ell^2}^2-n\left|G(z)\right|^2 \right) \\
			&\le \frac{n-2}{n}\|z\|_{\ell^2}^2.
		\end{aligned}
	\end{equation} By the Riesz--Thorin interpolation theorem for $L^p$ spaces (see Theorem~\ref{thm:RT_interpolation1}), we have 
	\[
	\big(\ell^\infty(\mathbb{C}^n), \ell^2(\mathbb{C}^n)\big)_{[\theta]} = \ell^p(\mathbb{C}^n)\quad \text{(with equal norms)},
	\]
	where $\theta = \frac{2}{p}$ and $p \in (2,\infty)$. Since the Riesz--Thorin interpolation theorem holds for the Schatten \(p\)-classes (see Remark~\ref{rmk:RT_interpolation1}), we also have
	\[
	\big(S_\infty, S_2\big)_{[\theta]} = S_p\quad \text{(with equal norms)},
	\]
	with $\theta = \frac{2}{p}$ for $p \in (2,\infty)$. From \eqref{eq:pinfty,0}--\eqref{eq:p2}, we apply the abstract complex interpolation theorem (see Theorem~\ref{thm:complex-interpolation}) to the two compatible Banach couples $(\ell^\infty(\mathbb{C}^n)),\ell^2(\mathbb{C}^n))$ and $(S_\infty,S_2)$ and obtain\[
	\|\widehat{T}\|_{\ell^p(\mathbb{C}^n) \to S_p} \le \|\widehat{T}\|_{\ell^\infty(\mathbb{C}^n) \to S_\infty}^{\,1-\theta} \, \|\widehat{T}\|_{\ell^2(\mathbb{C}^n) \to S_2}^{\,\theta} \leq \left(\frac{2(n-1)}{n} \right) ^{1-\frac{2}{p}} \cdot \left(\sqrt{\frac{n-2}{n}}\right)^{\frac{2}{p}} .
	\] This means
	\[
	\|\widehat{T}(z)\|_{S_p}\le  \left(\frac{2(n-1)}{n} \right) ^{1-\frac{2}{p}}  \left(\frac{n-2}{n}\right)^{\frac{1}{p}}\|z\|_{\ell^p}
	\]
    for $ p> 2$ and $ z\in \mathbb{C}^n$.
	Taking $z=(z_1,\dots,z_n)$ yields
	$$\|T_0(z)\|_{S_p}^p=\|\widehat{T}|_V(z)\|_{S_p}^p\le	\|\widehat{T}(z)\|_{S_p}^p\le \left(\frac{2(n-1)}{n} \right) ^{p-2}  \frac{n-2}{n}\sum_{j=1}^n|z_j|^p.$$
	Recalling $A=T(z)$ in~\eqref{eq:eval-to-sval} to finish the proof of the case for $p\ge 2$.
\end{proof}
\begin{remark}
	Under the assumption $\sum_{j=1}^n z_j=0$ and $p\in \left(2,2+\log\frac{n-1}{n-2}/\log\frac{2(n-1)}{n}\right)$, Theorem~\ref{thm:weaker result for p>= 2} is stronger than Pereira's result~\eqref{eq:Pereira1}.
\end{remark}


Similarly to Theorem~\ref{thm:weaker result for p>= 2}, we can also obtain estimates in the range \(1 \le p \le 2\), although these bounds are non-sharp when \(1 \le p < 2\). 
Before doing so, we first establish the following auxiliary result.

\begin{prop}\label{prop:p=1-endpoint-sharp}
Let
\[
Q := I - \frac{1}{n}J,\qquad 
G(z) := \frac{1}{n}\sum_{j=1}^n z_j,\qquad
\widehat{T}(z) := Q\mathrm{diag}\left(z - G(z)\mathbf{1}\right)Q.
\]Then for every $z\in\mathbb{C}^n$ we have
\begin{equation}\label{eq:sharp-p=1}
  \|\widehat{T}(z)\|_{S_1}
  \;\le\; \frac{2(n-2)}{n}\,\|z\|_{\ell^1}.
\end{equation}
Moreover, the constant $\frac{2(n-2)}{n}$ is best possible, and equality is attained,
for instance, at $z=(1,0,\dots,0)^\top$.
\end{prop}

\begin{proof}
Define
\[
f(z) := \|\widehat{T}(z)\|_{S_1},
\quad z\in\mathbb{C}^n.
\]
Since $\widehat{T}$ is linear and $\|\cdot\|_{S_1}$ is a norm, $f$ is convex and
$1$-homogeneous:
\[
f(\alpha z) = |\alpha|\,f(z),\qquad \alpha\in\mathbb{C},\ z\in\mathbb{C}^n.
\]
Thus, to prove \eqref{eq:sharp-p=1}, it suffices to show that
\[
\sup_{\|z\|_{\ell^1}=1} f(z) = \frac{2(n-2)}{n}.
\]


It is well known that the extreme points of the $\ell^1$-unit ball
\[
B_1 := \{z\in\mathbb{C}^n : \|z\|_{\ell^1}\le 1\}
\]
are precisely the vectors of the form $\lambda e_j$ with $|\lambda|=1$ and
$1\le j\le n$, where $e_j$ denotes the $j$-th standard basis vector in $\mathbb{C}^n$.
By a standard consequence of the Krein--Milman theorem in finite dimensions
(see, e.g., \cite[Theorem~II.3.3]{Barvinok25}), a continuous convex function on a compact convex set
attains its maximum at an extreme point of the set. Thus we have
\[
\sup_{\|z\|_{\ell^1}=1} f(z)
=
\max_{\substack{1\le j\le n\\ |\lambda|=1}} f(\lambda e_j).
\]

Next we use the symmetries of $\widehat{T}$. For any unimodular scalar
$|\lambda|=1$ we have
\[
G(\lambda z) = \lambda G(z),\qquad
\diag(\lambda z - G(\lambda z)\mathbf{1}) = \lambda\,\diag(z-G(z)\mathbf{1}),
\]
and thus
\[
\widehat{T}(\lambda z)
= Q\,\diag(\lambda z - G(\lambda z)\mathbf{1})\,Q
= \lambda\,Q\,\diag(z-G(z)\mathbf{1})\,Q
= \lambda\,\widehat{T}(z).
\]
Therefore $f(\lambda z)=f(z)$ for all $|\lambda|=1$. Similarly, for any permutation matrix $\Pi$ we have $J\Pi=\Pi J=J$ and hence
$Q\Pi=\Pi Q$. Writing $z'=\Pi z$, we obtain
\[
\widehat{T}(z')
= Q\,\diag(z'-G(z')\mathbf 1)\,Q
= Q\,\Pi\,\diag(z-G(z)\mathbf 1)\,\Pi^\top Q
= \Pi\,\widehat{T}(z)\,\Pi^\top,
\]
so $f(z')=\|\widehat{T}(z')\|_{S_1}=\|\widehat{T}(z)\|_{S_1}$. Consequently,
\[
f(\lambda e_j) = f(e_1)\qquad
\text{for all }|\lambda|=1,\ 1\le j\le n,
\]
and hence
\[
\sup_{\|z\|_{\ell^1}=1} f(z) = f(e_1).
\]

We now compute $\|\widehat{T}(e_1)\|_{S_1}$. Clearly, $G(e_1) = \frac1n$ and $e_1-G(e_1)\mathbf 1 = \left(\frac{n-1}{n},-\frac1n,\dots,-\frac1n\right)$. Set
\[
D := \diag(e_1-G(e_1)\mathbf 1)
= \diag\left(\tfrac{n-1}{n},-\tfrac1n,\dots,-\tfrac1n\right).
\]
Then
\[
\widehat{T}(e_1) = Q D Q =: A.
\]
Both $D$ and $Q$ are real symmetric matrices, hence $A$ is real symmetric as well,
and its Schatten $1$-norm equals the sum of the absolute values of its eigenvalues.

Let $u := (-(n-1),1,\dots,1)^\top$. Consider the orthogonal decomposition
\[
\mathbb{R}^n
= E_0 \oplus E_1 \oplus E_2,
\]
where
\[
E_0 := \operatorname{span}\{\mathbf 1\},\qquad
E_1 := \operatorname{span}\{u\},\quad
E_2 := \{x\in\mathbb{R}^n : x_1=0,\ \sum_{j=2}^n x_j=0\}.
\]One checks that $E_0,E_1,E_2$ are mutually orthogonal and invariant under $A$.

\emph{(i) On $E_0$.}
Since $Q\mathbf 1=0$, we have $A\mathbf 1 = Q D Q \mathbf 1 = 0$, so $0$ is an eigenvalue.

\emph{(ii) On $E_2$.}
If $x\in E_2$, then $x_1=0$ and $\sum_{j=2}^n x_j=0$. Thus
\[
(Dx)_1 = \frac{n-1}{n} x_1 = 0,\qquad
(Dx)_j = -\frac1n x_j\quad (2\le j\le n),
\]
so $\sum_{j=1}^n (Dx)_j = 0$, i.e.\ $Dx\in V:=\{y:\sum_{j=1}^n y_j=0\}$, and hence $QDx = Dx$.
Since also $x\in V$, we have $Qx=x$, and therefore
\[
Ax = QDQx = QDx = Dx = -\frac1n x.
\]
Thus $-\frac1n$ is an eigenvalue of $A$ with multiplicity $\dim E_2 = n-2$.

\emph{(iii) On $E_1$.}
The vector $u$ satisfies $\sum_{j=1}^n u_j=0$, so $u\in V$ and $Qu=u$.
A direct computation gives
\[
Du = \left(-\tfrac{(n-1)^2}{n}, -\tfrac1n,\dots,-\tfrac1n\right)^\top,
\]
and $\sum_{j=1}^n (Du)_j
= -\frac{(n-1)^2}{n} + (n-1)\cdot\left(-\frac1n\right)
= -(n-1)$. Hence
\[
QDu = Du - \frac1n JDu = Du + \frac{n-1}{n}\mathbf 1,
\]
that is,
\[
(QDu)_1 = -\frac{(n-1)(n-2)}{n},\qquad
(QDu)_j = \frac{n-2}{n}\quad (j\ge2).
\]
Thus
\[
QDu = \left(-\tfrac{(n-1)(n-2)}{n}, \tfrac{n-2}{n}, \dots, \tfrac{n-2}{n}\right)
= \frac{n-2}{n}(- (n-1), 1, \dots, 1)
= \frac{n-2}{n} u.
\]
Since $Qu=u$, we obtain
\[
Au = QDQu = QDu = \frac{n-2}{n} u,
\]
so $\frac{n-2}{n}$ is an eigenvalue of $A$, corresponding to $E_1$.

Collecting the above, the eigenvalues of $A$ are
\[
\lambda_1 = \frac{n-2}{n},\quad
\lambda_2 = -\frac1n\ \text{(with multiplicity $n-2$)},\quad
\lambda_0 = 0.
\]
Therefore
\[
\|\widehat{T}(e_1)\|_{S_1}
= \|A\|_{S_1}
= |\lambda_1| + (n-2)|\lambda_2| + |\lambda_0|
= \frac{n-2}{n} + (n-2)\frac1n
= \frac{2(n-2)}{n}.
\]
Since $\|e_1\|_{\ell^1}=1$, this shows that
\[
\sup_{\|z\|_{\ell^1}=1}\|\widehat{T}(z)\|_{S_1}
= \|\widehat{T}(e_1)\|_{S_1}
= \frac{2(n-2)}{n},
\]
and hence \eqref{eq:sharp-p=1} holds for all $z\in\mathbb{C}^n$, with best possible
constant. In particular, equality is attained at $z=e_1$.
\end{proof}

We are now ready to present the following.
\begin{thm}\label{thm:weaker result for 1<=p<= 2}
	For $n \geq 4$, let \( z_1, z_2, \ldots, z_n \) be the zeros of \( p(z) \), and let \( w_1, w_2, \ldots, w_{n-1} \) be its critical points. Assume further that $\sum_{j=1}^n z_j=0$. Then for $1\le p\le 2$, 
		\[
	\sum_{k=1}^{n-1} |w_k|^p \le  \frac{2^{2-p} (n-2)}{n} \sum_{j=1}^n |z_j|^p .
	\]
\end{thm}

\begin{proof}
Similar to the proof of case $p\ge 2$, we begin by using the two endpoint estimates corresponding to $p=1$ and $p=2$, namely
\[
   \|\widehat{T}(z)\|_{S_1}\le  \frac{2(n-2)}{n} \|z\|_{\ell^1},\qquad
   \|\widehat{T}(z)\|_{S_2}\le \sqrt{\frac{n-2}{n}} \|z\|_{\ell^2},
\]
where the first inequality is given by Proposition~\ref{prop:p=1-endpoint-sharp}  and the second by~\eqref{eq:p2}. By the Riesz--Thorin interpolation theorem for $L^p$ spaces and for Schatten $p$-classes, we know that when $p \in (1,2)$, we have 
\[
\left(\ell^1(\mathbb C^{n}),\ell^2(\mathbb C^{n})\right)_{[\theta]} = \ell^p(\mathbb C^{n}), 
\qquad 
\left(S_1,S_2\right)_{[\theta]} = S_p \quad \text{(with equal norms)},
\]
where $1/p = (1-\theta)/1 + \theta/2$. Hence, we apply Theorem~\ref{thm:complex-interpolation} to the two compatible Banach couples $(\ell^1(\mathbb C^{n}),\ell^2(\mathbb C^{n}))$ and $(S_1,S_2)$ and obtain\[
   \|\widehat{T}(z)\|_{S_p} \le \left(\frac{2(n-2)}{n}\right)^{\frac{2}{p}-1} \left(\sqrt{\frac{n-2}{n}}\right)^{2-\frac{2}{p}} \|z\|_{\ell^p}
    = 2^{\frac{2}{p}-1}\left(\frac{n-2}{n}\right)^\frac{1}{p}\|z\|_{\ell^p}.
\]
Raising both sides to the $p$th power and using the eigenvalue--singular value comparison~\eqref{eq:eval-to-sval} yields
\[
   \sum_{k=1}^{n-1}|w_k|^p \le \|T_0(z)\|_{S_p}^p=\|\widehat{T}|_V(z)\|_{S_p}^p
    \le \|\widehat{T}(z)\|_{S_p}^p
    \le \frac{2^{2-p} (n-2)}{n} \sum_{j=1}^n |z_j|^p,
\]
as desired.
\end{proof}

\begin{remark}
	Under the assumption $\sum_{j=1}^n z_j=0$ and $p\in \left(2-\log\frac{n-1}{n-2}/\log2,2\right)$, Theorem~\ref{thm:weaker result for 1<=p<= 2} is stronger than Pereira's result~\eqref{eq:Pereira1}.
\end{remark}

\section{Schoenberg type inequality of order $-1$}\label{sec:order_negative_1}

Throughout Sections~\ref{sec:order_negative_1}--\ref{sec:negative_interpolation} we consider normalized polynomials of the form
\begin{equation*}
  P(z) = z \prod_{j=1}^{n} (z - z_j),
  \qquad z_j \in \C \setminus \{0\},
\end{equation*}
and denote by \(w_1,\dots,w_n\) the zeros of the derivative \(P'(z)\).
A straightforward computation shows that the zeros and the critical points satisfy the
identities
\begin{equation}\label{eq:reciprocal-1}
  \sum_{k=1}^{n} \frac{1}{w_k}
  = 2 \sum_{j=1}^{n} \frac{1}{z_j},
\end{equation}
and
\begin{equation}\label{eq:reciprocal-2}
  \sum_{k=1}^{n} \frac{1}{w_k^{2}}
  = 3 \sum_{j=1}^{n} \frac{1}{z_j^{2}}
  + \left( \sum_{j=1}^{n} \frac{1}{z_j} \right)^{2}.
\end{equation}
Proofs of the identities~\eqref{eq:reciprocal-1} and~\eqref{eq:reciprocal-2} can be found in Appendix~\ref{sec:appendix_reciprocal}.

\begin{remark}[A simple consequence for Sendov's conjecture]\label{rem:Sendov-a=1}
As a simple application of~\eqref{eq:reciprocal-1}, we note that the case $a=1$ of
Conjecture~\ref{conj:Sendov} follows directly from this identity.
Indeed, if \(|z_j + 1| \le 1\) for all \(j\), then for each \(1 \le j \le n\), \(\Re\left( \frac{1}{z_j} \right) \le -\frac{1}{2}\). Multiplying by $2$ and summing over $j$ yields
\(\Re\left( \sum_{j=1}^n \frac{2}{z_j} \right) \le -n\).
By \eqref{eq:reciprocal-1}, this implies \(\Re\left( \sum_{k=1}^n \frac{1}{w_k} \right) \le -n\), and hence \(\min_{1 \le k \le n} |w_k| \le 1\).
\end{remark}

Our goal in this section is to compare the sums of negative powers of the moduli of the
critical points \(w_k\) and the zeros \(z_j\) in a Schoenberg-type fashion.
We begin by recalling the following result, which is a consequence of
Malamud~\cite[Theorem~4.10]{Mal04} (see also~\cite[Corollary~3.3]{Tang25} for a different proof).

\begin{lem}[\cite{Mal04,Tang25}]\label{lem:ek}
Let \( z_1, z_2, \ldots, z_n \) be the zeros of a complex polynomial \( p(z) \) and
\( w_1, w_2, \ldots, w_{n-1} \) be its critical points. Then 
\[
  e_k(|\bw|) \le \frac{n-k}{n} e_k(|\bz|), \qquad  k=1, \ldots, n-1,
\]
where \(\bz = (z_1,\dots,z_n)\) and \(\bw = (w_1,\dots,w_{n-1})\).
\end{lem}

The next theorem can be viewed as a negative-order analogue of Lemma~\ref{lem:ek}.

\begin{thm}\label{thm:negative_order_ek_1}
Let
\[
  P(z) = z \prod_{j=1}^{n} (z - z_j),
  \qquad z_j \in \C \setminus \{0\},
\]
and let $w_1,\dots,w_n$ denote the zeros of $P'(z)$. Then
\begin{equation}\label{eq:ele-minus-one}
  e_k(|\bw|^{-1})
  \le (k+1) e_k(|\bz|^{-1}),  \qquad  k=1, \ldots, n,
\end{equation}
where \( \bz = (z_1,\dots,z_n) \) and \( \bw = (w_1,\dots,w_n) \).
Moreover, equality in \eqref{eq:ele-minus-one} holds for all $k=1,\dots,n$ whenever
$z_1,\dots,z_n$ all lie on the same ray emanating from the origin.
\end{thm}

\begin{proof}
By Lemma~\ref{lem:ek}, applied to the polynomial $P$, we have
\[
  (n+1) e_{n-k}(|w_1|,\dots,|w_n|)
  \le (k+1) e_{n-k}(|z_1|,\dots,|z_n|).
\]
On the other hand, since $P$ is monic of degree $n+1$ we have
\[
  w_1 \cdots w_n = \frac{z_1 \cdots z_n}{n+1},
\]
and hence
\[
  |w_1 \cdots w_n| = \frac{|z_1 \cdots z_n|}{n+1}.
\]
Thus,
\begin{align*}
  e_k(|\bw|^{-1})
  &= \frac{e_{n-k}(|w_1|,\dots,|w_n|)}{|w_1\cdots w_n|}
   = \frac{(n+1)\,e_{n-k}(|w_1|,\dots,|w_n|)}{|z_1\cdots z_n|} \\
  &\leq \frac{(k+1)\,e_{n-k}(|z_1|,\dots,|z_n|)}{|z_1\cdots z_n|}
   = (k+1)\,e_k(|\bz|^{-1}),
\end{align*}
which proves \eqref{eq:ele-minus-one}.

We now turn to the equality case. Suppose that $z_1,\dots,z_n$ all lie on the same ray emanating from the origin. After multiplying $P$ by a unimodular constant and performing a rotation of the variable, we may assume without loss of generality that
\[
z_j > 0 \quad\text{for all } j=1,\dots,n.
\]
Then $P$ is a real polynomial and Gauss--Lucas theorem \cite[p. 18] {BE95} implies that all critical points $w_1,\dots,w_n$ are real and lie in the interval $(0,\max_{j=1}^n z_j]$; in particular,
\[
w_k > 0 \quad\text{for all } k=1,\dots,n.
\]Consider the generating functions
\[
E_z(t) := \prod_{j=1}^n \left(1+\frac{t}{z_j}\right)
= \sum_{k=0}^n e_k\bigl(z_1^{-1},\dots,z_n^{-1}\bigr)\,t^k,
\]
and
\[
E_w(t) := \prod_{k=1}^n \left(1+\frac{t}{w_k}\right)
= \sum_{k=0}^n e_k\bigl(w_1^{-1},\dots,w_n^{-1}\bigr)\,t^k.
\]
A direct computation using the factorization
\[
P(z) = z \prod_{j=1}^n (z-z_j), \qquad
P'(z) = (n+1)\prod_{k=1}^n (z-w_k),
\]
shows that
\[
E_w(t) = (t E_z(t))'.
\]
Indeed, one checks that
\(
E_w(t) = (-1)^n P'(-t)/\prod_{j=1}^n z_j
\)
and
\(
tE_z(t) = (-1)^{n+1} P(-t)/\prod_{j=1}^n z_j,
\)
so differentiating the latter with respect to $t$ yields the former. Comparing coefficients of $t^k$ in the identity
\(
E_w(t) = (tE_z(t))'
\)
gives
\[
e_k\bigl(w_1^{-1},\dots,w_n^{-1}\bigr)
= (k+1)\,e_k\bigl(z_1^{-1},\dots,z_n^{-1}\bigr),
\qquad k=0,\dots,n.
\]Since $z_j>0$ and $w_k>0$, we have $|z_j|^{-1}=z_j^{-1}$ and $|w_k|^{-1}=w_k^{-1}$, and hence
\[
e_k(|\bw|^{-1})
= e_k\bigl(w_1^{-1},\dots,w_n^{-1}\bigr)
= (k+1)\,e_k\bigl(z_1^{-1},\dots,z_n^{-1}\bigr)
= (k+1)e_k(|\bz|^{-1}),
\]
for all $k=1,\dots,n$. This establishes the claimed equality case in~\eqref{eq:ele-minus-one}.
\end{proof}


As a direct consequence of Theorem~\ref{thm:negative_order_ek_1}, we obtain the following Schoenberg type inequality of order $-1$, which may be viewed as a modulus analogue of the identity~\eqref{eq:reciprocal-1}.

\begin{cor}\label{cor:minus-one}
Let
\[
  P(z) = z \prod_{j=1}^{n} (z - z_j),
  \qquad z_j \in \C \setminus \{0\},
\]
and let $w_1,\dots,w_n$ denote the zeros of $P'(z)$. Then
\begin{equation}\label{eq:Schoenberg-minus-one}
  \sum_{k=1}^{n} \frac{1}{|w_k|}
  \le 2 \sum_{j=1}^{n} \frac{1}{|z_j|}.
\end{equation}
Moreover, equality in \eqref{eq:Schoenberg-minus-one} holds whenever
$z_1,\dots,z_n$ all lie on the same ray emanating from the origin.
\end{cor}

\begin{remark}
We do not attempt to characterize all cases of equality in
\eqref{eq:Schoenberg-minus-one}. For our purposes it suffices to know that
equality holds for a nontrivial family of configurations, which already
shows that \eqref{eq:Schoenberg-minus-one} is sharp.
\end{remark}


\section{Schoenberg type inequality of order $-2$}\label{sec:order_negative_2}

\begin{thm}\label{thm:BBstar-bound}
	Let 
	\[
	P(z)=z\prod_{j=1}^{n}(z-z_j),
	\qquad z_j\in\mathbb{C}\setminus\{0\},
	\]
	and let $w_1,\dots,w_n$ be the zeros of $P'(z)$. Then
	\begin{equation}\label{eq:BBstar-bound}
		\sum_{k=1}^{n}\frac{1}{|w_k|^{2}} \le (n+3)\sum_{j=1}^{n}\frac{1}{|z_j|^{2}}.
	\end{equation}
	Moreover, equality holds in
	\eqref{eq:BBstar-bound} if and only if $z_1=\cdots=z_n$.
\end{thm}

\begin{proof}
	  Applying Lemma~\ref{lem:CN-matrix-model} to $P$ and using
	$z_{n+1}=0$, we see that $w_1,\dots,w_n$ are
	precisely the eigenvalues of the $n\times n$ matrix
	\[
	A_0:=D\left(I-\frac{1}{n+1}J\right),
	\]where \(D:=\operatorname{diag}(z_1,\dots,z_n)\). Since $D$ is invertible, the matrices $A_0$ and
	\[
	A:=
	\left(I-\frac1{n+1}J\right)D
	= D - \frac1{n+1} \mathbf{1} Z^{\top},
	\qquad Z:=(z_1,\dots,z_n)^{\top},
	\]
	have the same multiset of eigenvalues (this follows from the standard
	fact that $XY$ and $YX$ have the same nonzero eigenvalues for square
	matrices $X,Y$).  Hence the eigenvalues of $A$ are also exactly
	$w_1,\dots,w_n$.
	
	Next we compute $B:=A^{-1}$ explicitly. Set $N:=n+1$ and $U:=\frac1N D^{-1}\mathbf{1}$, we write
	\[
	A = D - \frac1N \mathbf{1} Z^{\top}
	= D(I-UZ^{\top}).
	\] By the Sherman--Morrison formula,
	\[
	(I-UZ^{\top})^{-1}
	=I+U\,(1-Z^{\top}U)^{-1}Z^{\top}.
	\]
	We have
	\[
	Z^{\top}U
	= \frac1N Z^{\top} D^{-1}\mathbf{1}
	= \frac1N\sum_{j=1}^{n}\frac{z_j}{z_j}
	= \frac{n}{N}
	= \frac{n}{n+1},
	\]
	hence \((1-Z^{\top}U)^{-1} = n+1\). Therefore
	\[
	A^{-1}
	= (I-UZ^{\top})^{-1}D^{-1}
	= \bigl(I+(n+1)UZ^{\top}\bigr)D^{-1}.
	\]
	Substituting $U=\frac1N D^{-1}\mathbf{1}$ gives $(n+1)UZ^{\top}= D^{-1}\mathbf{1}Z^{\top}$, and using $Z^{\top}D^{-1}=\mathbf{1}^{\top}$ we obtain
	\[
	B=A^{-1}
	= D^{-1}+D^{-1}\mathbf{1} Z^{\top}D^{-1}
	= D^{-1} + D^{-1}\mathbf{1} \mathbf{1}^{\top}.
	\]
	Set \(b:=D^{-1}\mathbf{1}=\left(\frac{1}{z_1},\dots,\frac{1}{z_n}\right)^{\top}=\left(b_1,\dots,b_n\right)^{\top}\). Then
	\begin{equation}\label{eq:B_inverse_of_A_11111}
	B = D^{-1} + b\,\mathbf{1}^{\top}.
	\end{equation}

	We now estimate the sum $\sum_{k=1}^{n}1/|w_k|^{2}$ via the singular
	values of $B$.  Let $\lambda_1(B),\dots,\lambda_n(B)$ be the eigenvalues of
	$B$ and $\sigma_1(B),\dots,\sigma_n(B)$ its singular values.  Since
	$w_k$ are the eigenvalues of $A$, we have
	\begin{equation}\label{eq:eiganvalues_of_B_1}
		\lambda_k(B)=\frac{1}{w_k},\qquad k=1,\dots,n.
	\end{equation}
	By Schur's inequality (Lemma~\ref{lem:weyl}),
	\begin{equation}\label{eq:BBstar-inequality}
\sum_{k=1}^{n}\frac{1}{|w_k|^{2}}
		= \sum_{k=1}^n |\lambda_k(B)|^2
	\le
	\sum_{k=1}^n \sigma_k(B)^2
	=\operatorname{tr}(BB^*),
	\end{equation}with equality if and only if $B$ is normal.

	Now we compute $\operatorname{tr}(BB^*)$. From~\eqref{eq:B_inverse_of_A_11111} we obtain $B^* = (D^{-1})^* + \mathbf{1}\, b^*$, and hence
	\[
	\begin{aligned}
		BB^*
		&=(D^{-1} + b\,\mathbf{1}^{\top})\left((D^{-1})^* + \mathbf{1}\,b^*\right)\\
		&= D^{-1}(D^{-1})^*
		+ D^{-1}\mathbf{1}\,b^*
		+ b\,\mathbf{1}^{\top}(D^{-1})^*
		+ b\,\mathbf{1}^{\top}\mathbf{1}\,b^*.
	\end{aligned}
	\]
	Using $D^{-1}\mathbf{1}=b$, $\mathbf{1}^{\top}(D^{-1})^*=b^*$ and
	$\mathbf{1}^{\top}\mathbf{1}=n$, this simplifies to
	\begin{equation}\label{eq:Thm_6.1_BBstar_11}
	BB^* = D^{-1}(D^{-1})^* + b\,b^* + b\,b^* + n\,b\,b^*
	= D^{-1}(D^{-1})^* + (n+2)\,b\,b^*.
	\end{equation}
	Since $D^{-1}(D^{-1})^*$ is diagonal with diagonal entries $1/|z_j|^2$,
	we have
	\[
	\operatorname{tr}\left(D^{-1}(D^{-1})^*\right)=\operatorname{tr}(b\,b^*) = b^*b
	= \sum_{j=1}^{n}\frac{1}{|z_j|^2}.
	\]Therefore
	\[
	\operatorname{tr}(BB^*)
	= \sum_{j=1}^{n}\frac{1}{|z_j|^2}
	+ (n+2)\sum_{j=1}^{n}\frac{1}{|z_j|^2}
	= (n+3)\sum_{j=1}^{n}\frac{1}{|z_j|^2}.
	\]
	Combining this with \eqref{eq:BBstar-inequality} yields~\eqref{eq:BBstar-bound}.
	
By~\eqref{eq:BBstar-inequality} and Lemma~\ref{lem:weyl}, equality in \eqref{eq:BBstar-bound} holds if and only if $B$ is normal. 

\begin{claim}\label{claim:B_normal_iff_11}
$B$ is normal if and only if $z_1=\cdots=z_n$.    
\end{claim} 
\begin{proof}[Proof of Claim~\ref{claim:B_normal_iff_11}]
Recalling that $b_i = 1/z_i$, we set
	\[
	F:=(D^{-1})^* D^{-1}=\operatorname{diag}\left(|b_1|^2,\dots,|b_n|^2\right),
	\quad
	f:=F\mathbf{1}=\begin{pmatrix}|b_1|^2\\ \vdots\\ |b_n|^2\end{pmatrix},
	\quad
	s:=\mathbf{1}^{\top}f = \sum_{j=1}^n |b_j|^2.
	\] Then \eqref{eq:Thm_6.1_BBstar_11} gives \(BB^* = F + (n+2)\,bb^*\). On the other hand, by \eqref{eq:B_inverse_of_A_11111} we have
	\[
	B^*B = (I+J)(D^{-1})^* D^{-1}(I+J) = (I+J)F(I+J)
	= F + JF + FJ + JFJ.
	\]
	Using $JF = \mathbf{1}f^{\top}$, $FJ = f\mathbf{1}^{\top}$ and 
	$JFJ = (\mathbf{1}^{\top}f)\,J = sJ$, we obtain
	\begin{equation}\label{eq:Thm_6.1_BstarB_11}
	B^*B = F + \mathbf{1}f^{\top} + f\mathbf{1}^{\top} + sJ.
	\end{equation}Thus $B$ is normal, i.e.\ $BB^*=B^*B$, if and only if
	\begin{equation}\label{eq:matrix-equality}
		\mathbf{1}f^{\top} + f\mathbf{1}^{\top} + sJ = (n+2)\,bb^*.
	\end{equation}
	Writing \eqref{eq:matrix-equality} entrywise, we obtain for all
	indices $1\le i,j\le n$,
	\begin{equation}\label{eq:entrywise}
		|b_i|^2 + |b_j|^2 + s = (n+2)b_i\overline{b_j}.
	\end{equation}Taking $i=j$ in \eqref{eq:entrywise} yields
	\[
	2|b_i|^2 + s = (n+2)|b_i|^2,
	\]
	hence
	\[
	\sum_{j=1}^n |b_j|^2 = n|b_i|^2\qquad\text{for all }1\le i \le n.
	\]
	In particular, all $|b_i|$ are equal: there exists $r>0$ such that
	\[
	|b_i|^2 = r^2,\qquad i=1,\dots,n,
	\]
	and then $s = nr^2$.
	
	Now fix $i\neq j$ in \eqref{eq:entrywise}. Using $|b_i|^2=|b_j|^2=r^2$
	and $s=nr^2$, we get
	\[
	r^2 + r^2 + nr^2 = (n+2)b_i\overline{b_j},
	\]
	that is,
	\[
	(n+2)r^2 = (n+2)b_i\overline{b_j},
	\]
	so
	\[
	b_i\overline{b_j} = r^2\qquad\text{for all }i,j.
	\]
	Write $b_i = r e^{\mathrm{i}\theta_i}$; then
	\[
	b_i\overline{b_j} = r^2 e^{\mathrm{i}(\theta_i-\theta_j)} = r^2
	\quad\Longrightarrow\quad
	e^{\mathrm{i}(\theta_i-\theta_j)} = 1
	\]
	for all $i,j$, hence $\theta_1=\cdots=\theta_n$ modulo $2\pi$.  It
	follows that all $b_i$ coincide:
	\[
	b_1=\cdots=b_n = re^{\mathrm{i}\theta}
	\]
	for some $\theta\in\mathbb{R}$.  Since $b_i=1/z_i$, we conclude that
	\[
	z_1=\cdots=z_n = \frac{1}{re^{\mathrm{i}\theta}}.
	\]
	
	Conversely, if $z_1=\cdots=z_n=z_0\neq0$, then
	\[
	D^{-1} = \frac{1}{z_0}I,\qquad
	B = D^{-1}(I+J) = \frac{1}{z_0}(I+J),
	\]
	and $I+J$ is Hermitian. Hence $B$ is a scalar multiple of a Hermitian
	matrix and thus normal.  This shows that $B$ is normal if and only if
	$z_1=\cdots=z_n$, proving the claim.\end{proof}
Now the equality case of \eqref{eq:BBstar-bound} follows directly from Claim~\ref{claim:B_normal_iff_11}. This completes the proof.
\end{proof}

\begin{remark}
	Theorem~\ref{thm:BBstar-bound} gives a sharp inequality of the form \(
	\sum_{k=1}^{n}\frac{1}{|w_k|^{2}} \le C \sum_{j=1}^{n}\frac{1}{|z_j|^{2}}\). By contrast, combining Corollary~\ref{cor:minus-one},
	\(
	\sum_{k=1}^{n}\frac{1}{|w_k|}
	\le 2\sum_{j=1}^{n}\frac{1}{|z_j|},
	\)
	with the elementary inequality
	$\sum_k |a_k|^{2}\le(\sum_k |a_k|)^2$ and then the Cauchy--Schwarz
	inequality, one only obtains the much weaker bound $\sum_{k=1}^{n}\frac{1}{|w_k|^{2}}
	\le 4\left(\sum_{j=1}^{n}\frac{1}{|z_j|}\right)^{2}
	\le 4n\sum_{j=1}^{n}\frac{1}{|z_j|^{2}}$.
\end{remark}

The identity~\eqref{eq:reciprocal-2} leads us to the following conjectural
refinements of Theorem~\ref{thm:BBstar-bound} in the case of order $-2$.

\begin{conj}\label{conj:C1}
	Let
	\[
	P(z) = z \prod_{j=1}^{n} (z - z_j),
	\qquad z_j \in \C \setminus \{0\},
	\]
	and let $w_1,\dots,w_n$ denote the zeros of $P'(z)$. Then
	\begin{equation}\label{eq:C1}
		\sum_{k=1}^{n} \frac{1}{|w_k|^{2}}
		\le 3 \sum_{j=1}^{n} \frac{1}{|z_j|^{2}}
		+ \left| \sum_{j=1}^{n} \frac{1}{z_j} \right|^{2}.
	\end{equation}
	Moreover, equality in \eqref{eq:C1} holds whenever
	$z_1,\dots,z_n$ all lie on the same line passing through the origin.
\end{conj}



In the case $n=2$ we can verify Conjecture~\ref{conj:C1} by an explicit
computation.

\begin{prop}\label{thm:n2-C1}
	Let $n=2$ and
	\[
	P(z) = z (z-a)(z-b), \qquad a,b \in \C \setminus \{0\}.
	\]
	Let $w_1,w_2$ be the zeros of $P'(z)$. Then
	\[
	\frac{1}{|w_1|^{2}} + \frac{1}{|w_2|^{2}}
	\le 3\left( \frac{1}{|a|^{2}} + \frac{1}{|b|^{2}} \right)
	+ \left| \frac{1}{a} + \frac{1}{b} \right|^{2}.
	\]
	Equivalently, Conjecture~\ref{conj:C1} holds for $n=2$.
\end{prop}

\begin{proof}
	We have
	\[
	P'(z) = 3z^{2} - 2(a+b)z + ab,
	\]
	so the critical points are
	\[
	w_{1,2}
	= \frac{(a+b) \pm \sqrt{a^{2}-ab+b^{2}}}{3}.
	\]
	A direct computation shows that
	\[
	w_1 + w_2 = \frac{2(a+b)}{3}, \qquad
	w_1 w_2   = \frac{ab}{3},     \qquad
	w_1 - w_2 = \frac{2}{3}\sqrt{a^{2}-ab+b^{2}}.
	\]
	Hence
	\begin{align}
		\frac{1}{|w_1|^{2}} + \frac{1}{|w_2|^{2}}
		&= \frac{|w_1|^{2} + |w_2|^{2}}{|w_1 w_2|^{2}} \notag \\
		&= \frac{|w_1 - w_2|^{2} + |w_1 + w_2|^{2}}
		{2\,|w_1 w_2|^{2}} \notag \\
		&= \frac{\tfrac{4}{9}|a^{2}-ab+b^{2}| + \tfrac{4}{9}|a+b|^{2}}
		{2\,|ab|^{2}/9} \notag \\
		&= \frac{2}{|ab|^{2}}
		\left( |a^{2}-ab+b^{2}| + |a+b|^{2} \right).
		\label{eq:w-sum}
	\end{align}
	On the other hand,
	\begin{align}
		3\left( \frac{1}{|a|^{2}} + \frac{1}{|b|^{2}} \right)
		+ \left| \frac{1}{a} + \frac{1}{b} \right|^{2}
		&= \frac{3(|a|^{2} + |b|^{2})}{|ab|^{2}}
		+ \frac{|a+b|^{2}}{|ab|^{2}} \notag \\
		&= \frac{3(|a|^{2} + |b|^{2}) + |a+b|^{2}}{|ab|^{2}}.
		\label{eq:rhs}
	\end{align}
	Comparing \eqref{eq:w-sum} and \eqref{eq:rhs}, we see that
	Conjecture~\ref{conj:C1} for $n=2$ is equivalent to
	\begin{equation}\label{eq:key-ineq}
		|a^{2}-ab+b^{2}|
		\le |a|^{2} + |b|^{2} - \Re(a\overline{b}).
	\end{equation}To prove \eqref{eq:key-ineq}, note that
	\[
	a^{2}-ab+b^{2} = (a-\omega b)(a-\overline{\omega} b),
	\qquad \omega = e^{i\pi/3}.
	\]
	Using the elementary inequality $2|xy| \le |x|^{2} + |y|^{2}$ for $x,y \in \C$,
	we obtain
	\begin{align*}
		|a^{2}-ab+b^{2}|
		&= |a-\omega b|\,|a-\overline{\omega} b| \\
		&\le \frac{|a-\omega b|^{2} + |a-\overline{\omega} b|^{2}}{2}.
	\end{align*}
	Since $|\omega|=1$ and $\omega + \overline{\omega} = 1$, a short computation
	gives
	\begin{align*}
		|a-\omega b|^{2} + |a-\overline{\omega} b|^{2}
		&= 2(|a|^{2} + |b|^{2})
		- 2\Re\left( \overline{a}(\omega+\overline{\omega})b \right) \\
		&= 2(|a|^{2} + |b|^{2}) - 2\Re(\overline{a}b) \\
		&= 2\left( |a|^{2} + |b|^{2} - \Re(a\overline{b}) \right).
	\end{align*}
	Combining the last two displays yields \eqref{eq:key-ineq}, which in turn is
	equivalent to the desired inequality in the statement of the proposition.
\end{proof}

\section{Schoenberg type inequality of order $-4$}\label{sec:order_negative_4}

\begin{thm}\label{thm:order-4-B2B2*}
	Let
	\[
	P(z) = z \prod_{j=1}^n (z - z_j), \qquad z_j \in \mathbb{C} \setminus \{0\},
	\]
	and let $w_1,\dots,w_n$ be the zeros of $P'(z)$. Then 
	\begin{align}\label{eq:order-4-B2B2*}
		\sum_{k=1}^n \frac{1}{|w_k|^4}
		\le
		&(n+5) \sum_{j=1}^n \frac{1}{|z_j|^4}
		+ \left(\sum_{j=1}^n \frac{1}{|z_j|^2}\right)^2+(n+2)\left(\sum_{j=1}^n \frac{1}{|z_j|^2}\right) \left|\sum_{j=1}^n \frac{1}{z_j}\right|^2
		\notag\\&+ 2(n+2)\Re\left\{\left(\sum_{j=1}^n \frac{1}{\bar{z}_j}\right)\left(\sum_{j=1}^n \frac{1}{z_j |z_j|^2}\right)\right\}.
	\end{align}
	Moreover, equality in \eqref{eq:order-4-B2B2*} holds whenever
	$z_1=\cdots=z_n$.
\end{thm}

\begin{proof}

As in the proof of Theorem~\ref{thm:BBstar-bound}, we write
\[
B := D^{-1} + b\mathbf{1}^{\top}, \quad 
D=\operatorname{diag}(z_1,\dots,z_n),\quad
b =\left(b_1,\dots,b_n\right)^{\top} = D^{-1}\mathbf{1}.
\]
By \eqref{eq:eiganvalues_of_B_1} we have $\lambda_k(B)=1/w_k$, and hence applying Lemma~\ref{lem:weyl} to $B^2$ gives
\begin{equation}\label{eq:Schur-B2}
   \sum_{k=1}^n \frac{1}{|w_k|^4}
		= \sum_{k=1}^n |\lambda_k(B^2)|^2
		\le
		\operatorname{tr}\left(B^2(B^2)^*\right).
\end{equation}

It remains to compute $\operatorname{tr}\left(B^2(B^2)^*\right)$ explicitly.
	Recalling that $b_j = 1/z_j$, we write $F := D^{-1}(D^{-1})^* = \operatorname{diag}(|b_1|^2,\dots,|b_n|^2)$, $K := bb^*$ and $\alpha := n+2$. For $r\ge 1$ set $A_r := \sum_{j=1}^n \frac{1}{|z_j|^{2r}}$. By \eqref{eq:Thm_6.1_BBstar_11} and \eqref{eq:Thm_6.1_BstarB_11}, we have
	\[
	BB^* = F + \alpha K,
	\qquad
	B^*B = F + \mathbf{1} f^\top + f\mathbf{1}^\top + A_1 J,
	\]
	where $f:=F\mathbf{1}= (|b_1|^2,\dots,|b_n|^2)^\top$.
	Writing
	\[
	M := BB^* = F + \alpha K,\qquad
	N := B^*B = F + L,\quad L := \mathbf{1} f^\top + f\mathbf{1}^\top + A_1 J,
	\]
	we have
	\begin{eqnarray} \label{eq:trMN}
	    	\operatorname{tr}\big(B^2(B^2)^*\big)
	= \operatorname{tr}(MN)= \operatorname{tr}(F^2)
	+ \operatorname{tr}(FL)
	+ \alpha\,\operatorname{tr}(KF)
	+ \alpha\,\operatorname{tr}(KL).
	\end{eqnarray}
 First, \begin{eqnarray}\label{eq: tr F^2}
     \operatorname{tr}(F^2) = \sum_{j=1}^n |b_j|^4 = A_2.
 \end{eqnarray} Next, 
	\[
	\operatorname{tr}\left(F(\mathbf{1}f^\top)\right) = A_2,
	\qquad
	\operatorname{tr}\left(F(f\mathbf{1}^\top)\right) = A_2,
	\qquad
	\operatorname{tr}\left(F A_1 J\right) = A_1^2,
	\]
	so
\begin{eqnarray}\label{eq:trFL}
   \operatorname{tr}(FL) = 2A_2 + A_1^2. 
\end{eqnarray}
	Moreover,
\begin{eqnarray} \label{eq: trKF}
    	\operatorname{tr}(KF)
	= \operatorname{tr}(FK)
	= b^* F b
	= \sum_{j=1}^n \overline{b_j} |b_j|^2 b_j
	= \sum_{j=1}^n |b_j|^4
	= A_2.
\end{eqnarray}
Finally, for $KL$, we have
	\[
	KL = K(\mathbf{1}f^\top) + K(f\mathbf{1}^\top) + A_1 KJ.
	\]
	Writing $U := \mathbf{1}^\top b = \sum_{j=1}^n \frac{1}{z_j}$ and
	$W := f^\top b = \sum_{j=1}^n  \frac{1}{z_j|z_j|^2} $, a direct computation yields
	\[
	\operatorname{tr}\big(K(\mathbf{1}f^\top)\big)
	= (b^*\mathbf{1})(f^\top b) = \overline{U} W,\qquad
	\operatorname{tr}\big(K(f\mathbf{1}^\top)\big)
	= (b^*f)(\mathbf{1}^\top b) = U \overline{W},
	\]
	and \(\operatorname{tr}(KJ)
	= (\mathbf{1}^\top b)(b^*\mathbf{1})
	= |U|^2\). Hence
\begin{eqnarray} \label{eq:tr KL}
    \operatorname{tr}(KL)
	= 2\Re(\overline{U}W) + A_1 |U|^2.
\end{eqnarray}
Collecting terms \eqref{eq: tr F^2}, \eqref{eq:trFL}, \eqref{eq: trKF} and \eqref{eq:tr KL}, we find in \eqref{eq:trMN}
	\[
	\begin{aligned}
		\operatorname{tr}(MN)
		&= A_2 + (2A_2 + A_1^2) + \alpha A_2
		+ \alpha\big(2\,\Re(\overline{U}W) + A_1 |U|^2\big)\\
		&= (3+\alpha) A_2 + A_1^2
		+ \alpha\big(2\,\Re(\overline{U}W) + A_1 |U|^2\big).
	\end{aligned}
	\]
	Since $\alpha = n+2$, we have $3+\alpha = n+5$, so
	\[
	\operatorname{tr}\left(B^2(B^2)^*\right)
	= (n+5) A_2 + A_1^2
	+ (n+2)\left(2\,\Re(\overline{U}W) + A_1 |U|^2\right),
	\]
	which is exactly the right-hand side of \eqref{eq:order-4-B2B2*}.
	Combining this with \eqref{eq:Schur-B2} proves the inequality.
	
	By \eqref{eq:Schur-B2}, equality in \eqref{eq:order-4-B2B2*} occurs if
	and only if equality holds in Schur's inequality for the matrix $B^2$,
	that is, if and only if $B^2$ is normal. In particular, if
	$z_1=\dots=z_n=z_0\neq 0$, then $D^{-1}=z_0^{-1}I_n$ and
	\[
	B = D^{-1}+b\mathbf{1}^\top = z_0^{-1}(I+J),
	\]
	so $B$ is a scalar multiple of the Hermitian matrix $I+J$ and thus normal. Hence $B^2$ is normal and equality holds in
	\eqref{eq:order-4-B2B2*}. This completes the proof.
\end{proof}

\begin{remark}
	A natural question is whether equality in \eqref{eq:order-4-B2B2*} occurs only in the configuration $z_1=\dots=z_n=z_0\neq 0$, as in Theorem~\ref{thm:BBstar-bound}.
	Unfortunately, the answer is negative.  
	By \eqref{eq:Schur-B2}, equality in \eqref{eq:order-4-B2B2*} holds if and only if 
	$B^{2}$ is a normal matrix. In particular, if $z_1=\dots=z_n=z_0\neq 0$, then 
	$B$ is normal, hence $B^{2}$ is normal, and equality holds in 
	\eqref{eq:order-4-B2B2*}. However, $B^{2}$ may be normal even when $B$ itself is not.  
	For example, when $n=2$ and $z_1=-z_2\neq 0$, one computes
	\[
	B=\frac{1}{z_1}
	\begin{pmatrix}
		2 & 1\\
		-1 & -2
	\end{pmatrix},
	\qquad
	B^{2}=\frac{3}{z_1^{2}}\,I_{2},
	\]
	so $B^{2}$ is a scalar multiple of the identity and therefore normal; 
	hence equality holds in \eqref{eq:order-4-B2B2*} as well.
	A complete characterization of all zero configurations for which $B^{2}$ is normal 
	appears to be rather subtle, and we shall not attempt it here.
\end{remark}

\section{Schoenberg type inequality of order $-2m$
}\label{sec:order_negative_2m}

By the argument in Sections~\ref{sec:order_negative_2} and~\ref{sec:order_negative_4}, we in fact have, for every integer $m\ge1$,
\[
\sum_{k=1}^n \frac{1}{|w_k|^{2m}}
= \sum_{k=1}^n |\lambda_k(B^m)|^2
\;\le\; \operatorname{tr}\big(B^m(B^m)^*\big).
\]
However, even for $m=3$ the explicit expression of
$\operatorname{tr}\big(B^3(B^3)^*\big)$ in terms of the zeros $z_1,\dots,z_n$
becomes extremely complicated and involves many mixed phase-dependent
quantities (such as $\sum 1/z_j$, $\sum 1/(z_j|z_j|^2)$, etc.).
If one is willing to accept a slightly weaker but much cleaner bound
for $\sum_k |w_k|^{-2m}$, it is natural to replace
$\operatorname{tr}\big(B^m(B^m)^*\big)$ by $\operatorname{tr}\big((BB^*)^m\big)$.
This yields
\[
\sum_{k=1}^n \frac{1}{|w_k|^{2m}}
\;\le\; \operatorname{tr}\big((BB^*)^m\big),
\]
and this inequality is indeed tighter, as it follows from a well-known consequence of Horn's inequalities that \(\operatorname{tr}\left( (A^*)^3 A^3 \right) \le \operatorname{tr}\left( (A^*A)^3 \right)\); see, for example, \cite[pp.~177, 190]{HJ91}. The advantage is that $\operatorname{tr}\big((BB^*)^m\big)$ admits a
closed form depending only on the moduli $|z_j|$, which leads to
much more transparent inequalities.



The next theorem shows that, for each integer $m \ge 1$, this approach
produces a sharp upper bound for $\sum_{k=1}^n |w_k|^{-2m}$, with 
Theorem~\ref{thm:BBstar-bound} corresponding to the special case $m=1$.

\begin{thm}\label{thm:general-2m}
	Let
	\[
	P(z) = z \prod_{j=1}^n (z - z_j), \qquad z_j \in \mathbb{C} \setminus \{0\},
	\]
	and let $w_1,\dots,w_n$ be the zeros of $P'(z)$. Then for every integer $m\ge 1$ we have
	\begin{align}\label{eq:general-2m}
		\sum_{k=1}^n \frac{1}{|w_k|^{2m}}
		\le
		&\sum_{j=1}^n \frac{1}{|z_j|^{2m}}
		\notag\\&+
		\sum_{\ell=1}^{m} (n+2)^{\ell}
		\sum_{\substack{r_0,\dots,r_\ell \ge 0\\ r_0+\cdots+r_\ell = m-\ell}}
		\left(\sum_{i=1}^n \frac{1}{|z_i|^{2(r_0+r_\ell+1)}}\right) \cdot \prod_{j=1}^{\ell-1} \left(\sum_{i=1}^n \frac{1}{|z_i|^{2(r_j+1)}}\right).
	\end{align}
	Moreover, equality holds in
	\eqref{eq:general-2m} if and only if $z_1=\cdots=z_n$.
\end{thm}

\begin{proof}

As in the proof of Theorem~\ref{thm:BBstar-bound}, we write
\[
B := D^{-1} + b\mathbf{1}^{\top}, \quad 
D=\operatorname{diag}(z_1,\dots,z_n),\quad
b =\left(b_1,\dots,b_n\right)^{\top} = D^{-1}\mathbf{1}.
\]For each integer $m\ge 1$, Lemma \ref{lem:weyl} and \eqref{eq:eiganvalues_of_B_1} give
	\begin{equation}\label{eq:Schur-general}
		\sum_{k=1}^n \frac{1}{|w_k|^{2m}}
		= \sum_{k=1}^n |\lambda_k(B)|^{2m}
		\le \sum_{k=1}^n \sigma_k^{2m}(B)
		= \operatorname{tr}\left((BB^*)^m\right),
	\end{equation}with equality if and only if $B$ is normal.

It remains to compute $\operatorname{tr}\left((BB^*)^m\right)$ explicitly.
	Recalling that $b_j = 1/z_j$, we write $F := D^{-1}(D^{-1})^* = \operatorname{diag}(|b_1|^2,\dots,|b_n|^2)$, $K := bb^*$ and $\alpha := n+2$. For $r\ge 1$ set $A_r := \sum_{j=1}^n \frac{1}{|z_j|^{2r}}$. By \eqref{eq:Thm_6.1_BBstar_11}, we have
	\[
	BB^* = F + \alpha K.
	\]We now compute $\operatorname{tr}\left((F+\alpha K)^m\right)$.
	Expanding the $m$th power,
	\[
	(F+\alpha K)^m
	= \sum_{\ell=0}^m \alpha^\ell
	\sum_{\substack{\text{words of length }m\\\text{with }\ell\text{ copies of }K}}
	\text{(corresponding word in }F\text{ and }K).
	\]
	Each word with $\ell$ copies of $K$ and $m-\ell$ copies of $F$ can be
	uniquely written as
	\[
	F^{r_0} K F^{r_1} K \cdots K F^{r_\ell},
	\]
	where $r_0,\dots,r_\ell\ge 0$ and
	$r_0+\dots+r_\ell = m-\ell$.
	Thus
	\[
	\operatorname{tr}\big((F+\alpha K)^m\big)
	= \operatorname{tr}(F^m)
	+ \sum_{\ell=1}^m \alpha^\ell
	\sum_{\substack{r_0,\dots,r_\ell\ge 0\\ r_0+\cdots+r_\ell = m-\ell}}
	\operatorname{tr}\!\big(F^{r_0} K F^{r_1} K \cdots K F^{r_\ell}\big).
	\]Using the rank-one structure $K=b b^*$, we have
	\[
	K F^r K = (b^* F^r b)\,K,\qquad r\ge 0,
	\]
	and by induction this yields
	\[
	\operatorname{tr}\!\big(F^{r_0} K F^{r_1} K \cdots K F^{r_\ell}\big)
	= \left(\prod_{j=1}^{\ell-1} b^* F^{r_j} b\right) \cdot (b^* F^{r_0+r_\ell} b).
	\]
	Since $F^s = \operatorname{diag}(|b_1|^{2s},\dots,|b_n|^{2s})$, we have
	\[
	b^* F^r b
	= \sum_{j=1}^n \overline{b_j} |b_j|^{2r} b_j
	= \sum_{j=1}^n |b_j|^{2(r+1)}
	= A_{r+1}.
	\]
	Therefore
	\[
	\operatorname{tr} \left(F^{r_0} K F^{r_1} K \cdots K F^{r_\ell}\right)
	= A_{r_0+r_\ell+1} \cdot \prod_{j=1}^{\ell-1} A_{r_j+1}.
	\]Putting these pieces together, we obtain
	\begin{equation}\label{eq:will_be_used_in_app_D1}
	\operatorname{tr}\left((BB^*)^m\right)
	= \operatorname{tr}\left((F+\alpha K)^m\right)
	= A_m
	+ \sum_{\ell=1}^{m} \alpha^\ell
	\sum_{\substack{r_0,\dots,r_\ell \ge 0\\ r_0+\cdots+r_\ell = m-\ell}}
	A_{r_0+r_\ell+1}\cdot \prod_{j=1}^{\ell-1} A_{r_j+1},
	\end{equation}
	which is exactly the right-hand side of \eqref{eq:general-2m} with
	$\alpha=n+2$. Combining this with \eqref{eq:Schur-general} proves
	\eqref{eq:general-2m}.
    
    By~\eqref{eq:Schur-general} and Lemma~\ref{lem:weyl}, we see that equality in~\eqref{eq:general-2m} holds if and only if $B$ is normal.
The stated equality case then follows directly from Claim~\ref{claim:B_normal_iff_11}.\end{proof}

As a direct consequence of Theorem~\ref{thm:general-2m}, we obtain the following
sharp inequality involving only 
\(\sum_{k=1}^n |w_k|^{-2m}\) and \(\sum_{j=1}^n |z_j|^{-2m}\). 
This also settles the case \(p = 2m\) of Conjecture~\ref{conj:negative_order_p_11_1}.

\begin{cor}\label{cor:sharp-2m}
	Let
	\[
	P(z) = z \prod_{j=1}^n (z - z_j), \qquad z_j \in \mathbb{C} \setminus \{0\},
	\]
	and let $w_1,\dots,w_n$ be the zeros of $P'(z)$. Then for every integer $m\ge 1$ we have
	\begin{equation}\label{eq:sharp-2m}
		\sum_{k=1}^n \frac{1}{|w_k|^{2m}}
		\le
		\frac{(n+1)^{2m} + n -1}{n}
		\sum_{j=1}^n \frac{1}{|z_j|^{2m}}.
	\end{equation}
    Moreover, equality holds in
	\eqref{eq:sharp-2m} if and only if $z_1=\cdots=z_n$.
\end{cor}

\begin{proof}
	Denote
	\[
	A_p := \sum_{j=1}^n \frac{1}{|z_j|^{2p}}, \qquad
	B_m := \sum_{k=1}^n \frac{1}{|w_k|^{2m}}.
	\]
	Then Theorem~\ref{thm:general-2m} states that
\begin{equation}\label{eq:Bm<=Am+xxx1}
		B_m
		\le
		A_m
		+
		\sum_{\ell=1}^{m} (n+2)^{\ell}
		\sum_{\substack{r_0,\dots,r_\ell \ge 0\\ r_0+\cdots+r_\ell = m-\ell}}
		A_{r_0 + r_\ell + 1} \cdot \prod_{j=1}^{\ell-1} A_{r_j + 1}.
\end{equation}We want to replace all lower-order sums $A_p$ ($1\le p\le m$) by expressions depending only on $A_m$ and $n$.

	Let $a_i := |z_i|^{-2} >0$. Then $A_p=\sum_{i=1}^n a_i^p$.  
	By H\"older's inequality we have, for every $1\le p\le m$,
	\begin{equation}\label{eq:A_p-bound}
		A_p
		=
		\sum_{i=1}^n a_i^p
		\le
		n^{1-p/m}\left(\sum_{i=1}^n a_i^m\right)^{p/m}
		=
		n^{1-p/m} A_m^{p/m}.
	\end{equation}
	Now fix $\ell\in\{1,\dots,m\}$ and a $(\ell+1)$-tuple $(r_0,\dots,r_\ell)$ of nonnegative integers with $r_0+\dots+r_\ell = m-\ell$. Define\[
	u_0 := r_0 + r_\ell + 1, \qquad
	u_j := r_j + 1 \quad (1\le j\le \ell-1).
	\] In the $\ell$th term of the right side of  \eqref{eq:Bm<=Am+xxx1}, we consider the factor $A_{r_0 + r_\ell + 1} \cdot \prod_{j=1}^{\ell-1} A_{r_j + 1}=A_{u_0} \cdot \prod_{j=1}^{\ell-1} A_{u_j}$. Since $r_0+\cdots+r_\ell = m-\ell$, we have $1\le u_0,u_j\le m$, so we can apply \eqref{eq:A_p-bound} to each $A_{u_0}, A_{u_j}$:
	\[
	A_{u_0} \le n^{1-u_0/m} A_m^{u_0/m}, \qquad
	A_{u_j} \le n^{1-u_j/m} A_m^{u_j/m}.
	\]
	Hence
	\begin{align*}
		A_{u_0}\cdot \prod_{j=1}^{\ell-1} A_{u_j}
		&\le
		n^{(1-u_0/m)+\sum_{j=1}^{\ell-1}(1-u_j/m)}
		A_m^{(u_0 + \sum_{j=1}^{\ell-1} u_j)/m}.
	\end{align*}
	We now compute the exponents. First,
	\begin{align*}
		u_0 + \sum_{j=1}^{\ell-1} u_j
		&=
		(r_0 + r_\ell + 1) + \sum_{j=1}^{\ell-1} (r_j+1)
		\\&=
		(r_0+\cdots+r_\ell) + 1 + (\ell-1)
		=
		(m-\ell) + 1 + (\ell-1) = m,
	\end{align*}
	so the exponent of $A_m$ is exactly $1$.  
	Next,
	\[
	(1-u_0/m) + \sum_{j=1}^{\ell-1}(1-u_j/m)
	=
	\ell - \frac{u_0 + \sum_{j=1}^{\ell-1} u_j}{m}
	=
	\ell - 1,
	\]
	so the exponent of $n$ is $\ell-1$. Therefore, for each $(r_0,\dots,r_\ell)$,
	\[
	A_{u_0}\prod_{j=1}^{\ell-1} A_{u_j}
	\;\le\;
	n^{\ell-1} A_m.
	\]
	This upper bound does not depend on the particular choice of $(r_0,\dots,r_\ell)$, only on $\ell$.
	
	Now we count how many $(\ell+1)$-tuples $(r_0,\dots,r_\ell)$ of nonnegative integers satisfy $r_0+\cdots+r_\ell = m-\ell$. This is the standard stars-and-bars number
	\[
\#\left\{(r_0,\dots,r_\ell)\in \mathbb{Z}_{\ge0}^{\ell+1} : r_0+\cdots+r_\ell = m-\ell \right\}
  = \binom{(m-\ell)+\ell}{\ell}
  = \binom{m}{\ell}.
\]
	Consequently, for each fixed $\ell$,
	\[
	(n+2)^{\ell}
		\sum_{\substack{r_0,\dots,r_\ell \ge 0\\ r_0+\cdots+r_\ell = m-\ell}}
		A_{r_0 + r_\ell + 1} \cdot \prod_{j=1}^{\ell-1} A_{r_j + 1}
	\le
	(n+2)^\ell \binom{m}{\ell} n^{\ell-1} A_m.
	\]
	Plugging this back into \eqref{eq:Bm<=Am+xxx1} gives
	\[
	B_m \le A_m + \sum_{\ell=1}^m (n+2)^\ell \binom{m}{\ell} n^{\ell-1}A_m
	= A_m \left[ 1 + \sum_{\ell=1}^m \binom{m}{\ell} (n+2)^\ell n^{\ell-1} \right].
	\]
	We simplify the bracket. Note that $(n+2)^\ell n^{\ell-1} = \frac{1}{n} (n^2+2n)^\ell$. Thus
	\begin{align*}
		1 + \sum_{\ell=1}^m \binom{m}{\ell} (n+2)^\ell n^{\ell-1}
		&= 1 + \frac{1}{n}\sum_{\ell=1}^m \binom{m}{\ell} (n^2+2n)^\ell
		\\&= 1 + \frac{1}{n}\left[(1+n^2+2n)^m - 1\right]
		\\&= \frac{(n+1)^{2m} + n - 1}{n}.
	\end{align*}
	Therefore
	\[
	B_m
	=
	\sum_{k=1}^n \frac{1}{|w_k|^{2m}}
	\;\le\;
	\frac{(n+1)^{2m} + n -1}{n}
	\sum_{j=1}^n \frac{1}{|z_j|^{2m}},
	\]
	which is exactly \eqref{eq:sharp-2m}.

	Finally, we verify the sharpness and the equality case in \eqref{eq:sharp-2m}.  
	If $z_1=\dots=z_n=z_0$, then
	\[
	P(z) = z(z-z_0)^n, \qquad
	P'(z) = (z-z_0)^{n-1}\bigl[(n+1)z - z_0\bigr],
	\]
	so $P'(z)$ has a zero at $z_0/(n+1)$ and $n-1$ zeros at $z_0$. Hence
	\[
	\sum_{k=1}^n \frac{1}{|w_k|^{2m}}
	=
	(n-1)\frac{1}{|z_0|^{2m}}
	+
	\frac{(n+1)^{2m}}{|z_0|^{2m}}
	=
	\frac{(n+1)^{2m} + n -1}{n} \sum_{j=1}^n \frac{1}{|z_j|^{2m}},
	\]
	so equality holds in \eqref{eq:sharp-2m} for these polynomials. Combined with the equality characterization in Theorem~\ref{thm:general-2m}, this shows that equality in~\eqref{eq:sharp-2m} holds if and only if $z_1=\cdots=z_n$ for every $m\ge 1$.\end{proof}

As simple applications of Theorem~\ref{thm:general-2m}, we now record the
Schoenberg type inequalities of orders $-4$ and $-6$.

Taking $m=2$ in \eqref{eq:general-2m} yields the following.
\begin{cor}\label{cor:order-4-from-general}
Let
	\[
	P(z) = z \prod_{j=1}^n (z - z_j), \qquad z_j \in \mathbb{C} \setminus \{0\},
	\]
	and let $w_1,\dots,w_n$ be the zeros of $P'(z)$. Then
	\begin{equation}\label{eq:order-4-from-general}
		\sum_{k=1}^n \frac{1}{|w_k|^{4}}
		\;\le\;
		(2n+5)\sum_{j=1}^n \frac{1}{|z_j|^{4}}
		\;+\;
		(n+2)^2 \left( \sum_{j=1}^n \frac{1}{|z_j|^{2}} \right)^{2}.
	\end{equation}
	Moreover, equality holds in~\eqref{eq:order-4-from-general} if and only if $z_1=\dots=z_n$. 
\end{cor}

\begin{remark}
	In general, inequality \eqref{eq:order-4-from-general}
	is weaker than the one in Theorem~\ref{thm:order-4-B2B2*}, since it is obtained
	by estimating
	\[
	\sum_{k=1}^n \frac{1}{|w_k|^{4}}
	\le \operatorname{tr}\left((BB^*)^{2}\right)
	\quad\text{rather than}\quad
	\operatorname{tr}\left(B^{2}(B^{2})^{*}\right),
	\]
	and we have $\operatorname{tr}\left(B^{2}(B^{2})^{*}\right)
	\le \operatorname{tr}\left((BB^*)^{2}\right)$ by Horn’s inequality; see \cite[pp.~177, 190]{HJ91}.
\end{remark}


Similarly, taking $m=3$ in \eqref{eq:general-2m} gives the following.
\begin{cor}\label{cor:order-6-from-general}
Let
	\[
	P(z) = z \prod_{j=1}^n (z - z_j), \qquad z_j \in \mathbb{C} \setminus \{0\},
	\]
	and let $w_1,\dots,w_n$ be the zeros of $P'(z)$. Then
	\begin{align}\label{eq:order-6}
		\sum_{k=1}^n \frac{1}{|w_k|^{6}}
		\;\le\;&\;
		(3n+7)\sum_{j=1}^n \frac{1}{|z_j|^{6}}
		\;+\; 3(n+2)^{2}
		\left(\sum_{j=1}^n \frac{1}{|z_j|^{2}}\right)
		\left(\sum_{j=1}^n \frac{1}{|z_j|^{4}}\right)\notag\\
		&\;+\; (n+2)^{3} \left(\sum_{j=1}^n \frac{1}{|z_j|^{2}}\right)^{3}.
	\end{align}
	Moreover, equality holds in \eqref{eq:order-6} if and only if $z_1=\dots=z_n$.
\end{cor}

\section{An interpolation approach to negative orders}\label{sec:negative_interpolation}

In this section we use complex interpolation to derive Schoenberg type inequalities of general negative orders, treating separately the ranges $(-\infty,-2]$ and $[-2,-1]$.

\begin{thm}\label{thm:negative-interpolation_for_p_2_infinity}
	Let
	\[
	P(z) = z \prod_{j=1}^n (z - z_j), \qquad z_j \in \mathbb{C}\setminus\{0\},
	\]
	and let $w_1,\dots,w_n$ be the zeros of $P'(z)$. Then for every 
	$p\ge 2$ we have
	\begin{equation}\label{eq:negative-interpolation-family}
		\sum_{k=1}^n \frac{1}{|w_k|^{p}}
		\;\le\;
		(n+1)^{p-2}(n+3)\sum_{j=1}^n \frac{1}{|z_j|^{p}}.
	\end{equation}
\end{thm}

\begin{proof}
As in the proof of Theorem~\ref{thm:BBstar-bound}, we write
\[
B := D^{-1} + b\mathbf{1}^{\top}, \quad 
D=\operatorname{diag}(z_1,\dots,z_n),\quad
b =\left(b_1,\dots,b_n\right)^{\top} = D^{-1}\mathbf{1}.
\]Set
	\[
	u_j := \frac1{z_j},\qquad u := (u_1,\dots,u_n)^\top.
	\]
	Then $D^{-1} = \diag(u)$ and $b=u$, so we may view $B$ as the image of
	$u$ under the linear map
	\begin{equation}\label{eq:definition_of_T_tilde_11}
	\widetilde{T} : \mathbb{C}^n \to \mathbb{C}^{n\times n},\qquad
	\widetilde{T}(u) := \diag(u) + u\mathbf{1}^\top.
\end{equation}
	
	For any $p\ge1$, Lemma \ref{lem:weyl} and \eqref{eq:eiganvalues_of_B_1} give
	\begin{equation}\label{eq:eig-vs-sing-B}
		\sum_{k=1}^n \frac{1}{|w_k|^{p}}
		= \sum_{k=1}^n \bigl|\lambda_k(B)\bigr|^{p}
		\le
		\|B\|_{S_p}^{p}
		= \|\widetilde{T}(u)\|_{S_p}^{p},
		\qquad p\ge1.
	\end{equation}We now estimate $\|\widetilde{T}(u)\|_{S_p}$ in terms of
	$\|u\|_{\ell^p}$ by interpolation between the endpoints $p=2$ and
	$p=\infty$.
	
	\smallskip\noindent\emph{Step 1: The case $p=2$.}
A direct computation gives
	\[
	\widetilde{T}(u) \widetilde{T}(u)^*
	= \diag(u) \diag(u)^*
	+ \diag(u) \mathbf{1}u^*
	+ u\mathbf{1}^\top \diag(u)^*
	+ u\mathbf{1}^\top\mathbf{1}u^*.
	\]
	Since $\diag(u) \diag(u)^* = \diag(|u_1|^2,\dots,|u_n|^2)$, $\diag(u) \mathbf{1}=u$, and
	$\mathbf{1}^\top\mathbf{1}=n$, we obtain
	\[
	\widetilde{T}(u) \widetilde{T}(u)^*
	= \diag(|u_1|^2,\dots,|u_n|^2) + (n+2)uu^*.
	\]Taking the trace, we find
	\[
	\|\widetilde{T}(u)\|_{S_2}^2
	= \operatorname{tr}\bigl(\widetilde{T}(u) \widetilde{T}(u)^*\bigr)
	= \sum_{j=1}^n |u_j|^2 + (n+2)\operatorname{tr}(uu^*)
	= (n+3)\sum_{j=1}^n |u_j|^2.
	\]
	Thus
	\begin{equation}\label{eq:Ttilde-l2-S2}
		\|\widetilde{T}(u)\|_{S_2}
		= \sqrt{n+3}\,\|u\|_{\ell^2},
		\qquad \|\widetilde{T}\|_{\ell^2(\mathbb{C}^n)\to S_2} = \sqrt{n+3}.
	\end{equation}

	\smallskip\noindent\emph{Step 2: The case $p=\infty$.} First, 
	\[
	\|\diag(u)\|_{S_\infty} = \max_{1\le j\le n} |u_j| = \|u\|_{\ell^\infty}.
	\] Next, the matrix $u\mathbf{1}^\top$ has rank one. Its nonzero singular
	value is the square root of the unique nonzero eigenvalue of
	$(u\mathbf{1}^\top)(u\mathbf{1}^\top)^*$. We compute
	\[
	(u\mathbf{1}^\top)(u\mathbf{1}^\top)^*
	= u\mathbf{1}^\top \mathbf{1}u^*
	= (\,\mathbf{1}^\top\mathbf{1}) u u^*
	= n u u^*.
	\]
	The matrix $uu^*$ has rank one with eigenvalue $\|u\|_{\ell^2}^2$, so
	$n uu^*$ has eigenvalues $n\|u\|_{\ell^2}^2$ (once) and $0$ (with
	multiplicity $n-1$). Therefore the unique nonzero singular value of
	$u\mathbf{1}^\top$ is
	\[
	\sigma_{\max}(u\mathbf{1}^\top)
	= \sqrt{n} \|u\|_{\ell^2},
	\]
	and hence
	\[
	\|u\mathbf{1}^\top\|_{S_\infty}
	= \sqrt{n} \|u\|_{\ell^2}.
	\]
	Using the elementary inequality $\|u\|_{\ell^2}\le \sqrt{n} \|u\|_{\ell^\infty}$
	we obtain
	\[
	\|u\mathbf{1}^\top\|_{S_\infty}
	\le \sqrt{n}\,\sqrt{n}\,\|u\|_{\ell^\infty}
	= n\,\|u\|_{\ell^\infty}.
	\] By the triangle inequality for $\|\cdot\|_{S_\infty}$ we now have
	\[
	\|\widetilde{T}(u)\|_{S_\infty}
	= \|\diag(u) + u\mathbf{1}^\top\|_{S_\infty}
	\le \|\diag(u)\|_{S_\infty} + \|u\mathbf{1}^\top\|_{S_\infty}
	\le (n+1)\|u\|_{\ell^\infty}.
	\]
	Thus
\begin{equation}\label{eq:tilde_T_S_infity_norm_11111}
\|\widetilde{T}\|_{\ell^\infty(\mathbb{C}^n) \to S_\infty}
	= \sup_{u\neq 0}
	\frac{\|\widetilde{T}(u)\|_{S_\infty}}{\|u\|_{\ell^\infty}}
	\le n+1.
	\end{equation} To see that the constant $n+1$ is optimal, we consider the special
	vectors $u=(t,\dots,t)^\top \in\C^n$ with $t\neq 0$. In this case
	\[
	\widetilde{T}(u)
	= \diag(t,\dots,t) + t\mathbf{1}\mathbf{1}^\top
	= t(I+J).
	\]The
	eigenvalues of $J$ are $n$ (once) and $0$ (with multiplicity $n-1$), so
	$I+J$ has eigenvalues $n+1$ and $1$. In particular,
	\[
	\|\widetilde{T}(u)\|_{S_\infty}
	= \sigma_{\max}\bigl(t(I+J)\bigr)
	= |t|(n+1)
	= (n+1)\|u\|_{\ell^\infty}.
	\]
	This shows that
	\[
	\|\widetilde{T}\|_{\ell^\infty(\mathbb{C}^n)\to S_\infty} \ge n+1,
	\]
	and combined with \eqref{eq:tilde_T_S_infity_norm_11111} we conclude that
	\begin{equation}\label{eq:Ttilde-linf-Sinf}
		\|\widetilde{T}\|_{\ell^\infty(\mathbb{C}^n)\to S_\infty} = n+1.
	\end{equation}

	\smallskip\noindent\emph{Step 3: Interpolation for $2<p<\infty$.}
	Consider the compatible Banach couples
	\[
	(\ell^\infty(\mathbb{C}^n),\ell^2(\mathbb{C}^n)),
	\qquad
	(S_\infty,S_2).
	\]
	By the Riesz--Thorin interpolation theorem for $L^p$ and for $S_p$ (see Theorem \ref{thm:RT_interpolation1} and Remark \ref{rmk:RT_interpolation1}), we have
	\[
	(\ell^\infty(\mathbb{C}^n),\ell^2(\mathbb{C}^n))_{[\theta]} = \ell^p(\mathbb{C}^n),
	\qquad
	(S_\infty,S_2)_{[\theta]} = S_p \quad \text{(with equal norms)},
	\]where $\theta = \frac{2}{p}$ and $p \in (2,\infty)$.

    From \eqref{eq:Ttilde-l2-S2}--\eqref{eq:Ttilde-linf-Sinf}, we apply the abstract complex interpolation theorem (see Theorem~\ref{thm:complex-interpolation}) to the two compatible Banach couples $(\ell^\infty(\mathbb{C}^n),\ell^2(\mathbb{C}^n))$ and $(S_\infty,S_2)$ and obtain\[
	\|\widetilde{T}(u)\|_{S_p}
	\;\le\; 
	\|\widetilde{T}\|_{\ell^\infty(\mathbb{C}^n)\to S_\infty}^{\,1-\theta}
	\|\widetilde{T}\|_{\ell^2(\mathbb{C}^n)\to S_2}^{\,\theta}
	\,\|u\|_{\ell^p}
	\le
	(n+1)^{1-\theta}(n+3)^{\frac{\theta}{2}}\,\|u\|_{\ell^p}.
	\]
	Substituting $\theta = \frac{2}{p}$ gives
	\begin{equation}\label{eq:Ttilde-lp-Sp}
		\|\widetilde{T}(u)\|_{S_p}
		\;\le\;
		(n+1)^{1-\frac{2}{p}}(n+3)^{\frac{1}{p}}\,\|u\|_{\ell^p},
		\qquad 2<p<\infty.
	\end{equation}
	For $p=2$ this inequality reduces to \eqref{eq:Ttilde-l2-S2}, so
	\eqref{eq:Ttilde-lp-Sp} holds for all $p\ge 2$. Combining \eqref{eq:eig-vs-sing-B} and \eqref{eq:Ttilde-lp-Sp} with
	$u_j=1/z_j$ yields, for every $p\ge 2$,
	\[
	\Biggl(\sum_{k=1}^n \frac{1}{|w_k|^{p}}\Biggr)^{\!1/p}
	\le \|\widetilde{T}(u)\|_{S_p}
	\le (n+1)^{1-\frac{2}{p}}(n+3)^{\frac{1}{p}}
	\Biggl(\sum_{j=1}^n \frac{1}{|z_j|^{p}}\Biggr)^{\!1/p}.
	\]
	Raising both sides to the $p$th power gives
	\[
	\sum_{k=1}^n \frac{1}{|w_k|^{p}}
	\;\le\;
	(n+1)^{p-2}(n+3)\sum_{j=1}^n \frac{1}{|z_j|^{p}},
	\]
	which is exactly \eqref{eq:negative-interpolation-family}. This completes the proof.
\end{proof}

\begin{remark}
	The family of inequalities in Theorem~\ref{thm:negative-interpolation_for_p_2_infinity}
	is obtained by complex interpolation of the endpoint bounds for the linear
	map $\widetilde{T}$, and the constants are therefore not expected to be
	optimal in general. This phenomenon already appears in the case $p=4$. Indeed, writing $A_r = \sum_{j=1}^n \frac{1}{|z_j|^{2r}}$, the inequality~\eqref{eq:negative-interpolation-family} with $p=4$ reads
	\[
	\sum_{k=1}^n \frac{1}{|w_k|^{4}}
	\;\le\;
	(n+1)^2(n+3)\,A_2.
	\]
	On the other hand, Corollary~\ref{cor:order-4-from-general} gives the
	sharper fourth-order inequality
	\[
	\sum_{k=1}^n \frac{1}{|w_k|^{4}}
	\le
	(2n+5)A_2 + (n+2)^2 A_1^2.
	\]
	Cauchy--Schwarz yields $A_1^2\le n A_2$, and hence
	\[
	(2n+5)A_2 + (n+2)^2 A_1^2
	\le
	\left(n(n+2)^2 + 2n + 5\right) A_2.
	\] A simple computation shows that \(n(n+2)^2 + 2n + 5<(n+1)^2(n+3)\) for every $n\ge 2$. Thus, for $n\ge 2$ and all choices of
	$z_1,\dots,z_n\in\C\setminus\{0\}$,
	\[
	(2n+5)A_2 + (n+2)^2 A_1^2
	< (n+1)^2(n+3) A_2,
	\]
	so Corollary~\ref{cor:order-4-from-general} yields a strictly better
	upper bound at order $-4$ than the interpolation inequality
	\eqref{eq:negative-interpolation-family}. In particular, the constants
	in Theorem~\ref{thm:negative-interpolation_for_p_2_infinity} are not
	optimal for $p=4$, and one should expect the same non-optimality for
	every finite $p>2$.
\end{remark}
\begin{remark}
	We emphasize that the endpoints are sharp: for $p=2$ the constant
	$n+3$ in \eqref{eq:negative-interpolation-family} coincides with the
	optimal one in Theorem~\ref{thm:BBstar-bound}, and in the limit $p\to\infty$
	the inequality recovers the sharp bound
	$\min_k |w_k| \ge (n+1)^{-1}\min_j |z_j|$, with equality whenever
	$z_1=\dots=z_n$.
\end{remark}

Theorem~\ref{thm:negative-interpolation_for_p_2_infinity} is obtained by interpolating between the two sharp endpoint cases $p=2$ and $p=\infty$. Moreover, by Corollary~\ref{cor:sharp-2m} we know that for every even exponent $p=2m$ one also has a sharp inequality. Thus, by employing a more refined (albeit more involved, and still not sufficient to fully resolve Conjecture~\ref{conj:negative_order_p_11_1}) interpolation scheme, we obtain the following result.

\begin{thm}\label{thm:refined-interpolation}
Let
\[
  P(z) = z\prod_{j=1}^n (z-z_j), \qquad z_j\in\C\setminus\{0\},
\]
and let $w_1,\dots,w_n$ be the zeros of $P'(z)$. Then for every $p\ge 2$ we have
\begin{equation}\label{eq:refined-interpolation}
  \sum_{k=1}^n \frac{1}{|w_k|^p}
  \;\leq\;
  C(n,p)\sum_{j=1}^n \frac{1}{|z_j|^p},
\end{equation}
where $\alpha:=\frac{p}{2} - \lfloor \frac{p}{2} \rfloor \in[0,1]$, and set
\[
  C(n,p)
  := C_m(n)^{1-\alpha} C_{m+1}(n)^{\alpha},
  \qquad
  C_m(n) := \frac{(n+1)^{2m}+n-1}{n}.
\]\end{thm}
Clearly, the inequality \eqref{eq:refined-interpolation} improves the bound obtained in Theorem~\ref{thm:negative-interpolation_for_p_2_infinity} for every $p>2$. The proof of Theorem~\ref{thm:refined-interpolation} is similar in spirit to that of Theorem~\ref{thm:negative-interpolation_for_p_2_infinity}. Since the argument is somewhat technical, we defer the full proof to Appendix~\ref{appendix:refined-interpolation}.

Next, we aim to obtain an analogue for orders in the range $[-2,-1]$ using a similar interpolation strategy. 
Before doing so, we establish the following lemma on the operator norm of $\widetilde{T}$.

\begin{lem}\label{lem:T-l1-S1}
	Let $\widetilde{T}$ be as defined in \eqref{eq:definition_of_T_tilde_11}. Then
	\[
	\|\widetilde{T}\|_{\ell^1(\mathbb{C}^n)\to S_1}
	= \sup_{u\neq 0} \frac{\|\widetilde{T}(u)\|_{S_1}}{\|u\|_{\ell^1}}
	= \sqrt{n+3}.
	\]
	In particular, for every $u\in\C^n$,
	\[
	\|\widetilde{T}(u)\|_{S_1} \le \sqrt{n+3}\,\|u\|_{\ell^1}.
	\]
\end{lem}

\begin{proof}
	We first recall a standard fact: if $L:\ell^1(\mathbb{C}^n)\to X$ is a linear map
	into a Banach space $X$, then
	\begin{equation}\label{eq:l1-opnorm-basis}
		\|L\|_{\ell^1(\mathbb{C}^n)\to X}
		= \max_{1\le j\le n} \|L(e_j)\|_X,
	\end{equation}
	where $e_j$ is the $j$th standard basis vector of $\C^n$.
	For completeness we include a short proof of~\eqref{eq:l1-opnorm-basis}. Indeed, let $u=\sum_{j=1}^n \alpha_j e_j$ with
	$\|u\|_{\ell^1}=\sum_{j=1}^n|\alpha_j|\le 1$. By linearity and the triangle
	inequality,
	\[
	\|L(u)\|_X
	= \left\| \sum_{j=1}^n \alpha_j L(e_j)\right\|_X
	\le \sum_{j=1}^n |\alpha_j| \|L(e_j)\|_X
	\le \left(\max_{1\le j\le n} \|L(e_j)\|_X\right)\sum_{j=1}^n|\alpha_j|
	\le \max_{1\le j\le n} \|L(e_j)\|_X.
	\]
	Taking the supremum over all $u$ with $\|u\|_{\ell^1}\le 1$ yields
	$\|L\|_{\ell^1(\mathbb{C}^n)\to X} \le \max_j\|L(e_j)\|_X$. The reverse inequality is
	obtained by testing \eqref{eq:l1-opnorm-basis} on $u=e_j$, which shows
	$\|L(e_j)\|_X \le \|L\|_{\ell^1(\mathbb{C}^n)\to X}$ for each $j$, and hence
	$\max_j\|L(e_j)\|_X \le \|L\|_{\ell^1(\mathbb{C}^n)\to X}$.
	
	Applying \eqref{eq:l1-opnorm-basis} to $L=\widetilde{T}$ and $X=S_1$,
	we obtain
	\[
	\|\widetilde{T}\|_{\ell^1(\mathbb{C}^n)\to S_1}
	= \max_{1\le j\le n} \|\widetilde{T}(e_j)\|_{S_1}.
	\]
	By symmetry it suffices to compute $\|\widetilde{T}(e_1)\|_{S_1}$,
	where $e_1=(1,0,\dots,0)^\top$. We have
	\[
	\diag(e_1) = \diag(1,0,\dots,0),\qquad
	e_1\mathbf{1}^\top =
	\begin{pmatrix}
		1 & 1 & \cdots & 1\\
		0 & 0 & \cdots & 0\\
		\vdots & \vdots & & \vdots\\
		0 & 0 & \cdots & 0
	\end{pmatrix},
	\]
	so
	\[
	\widetilde{T}(e_1)
	= \diag(e_1) + e_1\mathbf{1}^\top
	=
	\begin{pmatrix}
		2 & 1 & 1 & \cdots & 1\\
		0 & 0 & 0 & \cdots & 0\\
		\vdots & \vdots & \vdots & & \vdots\\
		0 & 0 & 0 & \cdots & 0
	\end{pmatrix}
	= e_1 r^\top,
	\]
	where $r=(2,1,\dots,1)^\top\in\C^n$. Thus $\widetilde{T}(e_1)$ has rank
	one. For a rank-one matrix $uv^\top$ the singular values are
	$\{\|u\|_{\ell^2}\|v\|_{\ell^2},0,\dots,0\}$, so 
	$\|uv^\top\|_{S_1} = \|u\|_{\ell^2}\|v\|_{\ell^2}$. Therefore
	\[
	\|\widetilde{T}(e_1)\|_{S_1}
	= \|e_1 r^\top\|_{S_1}
	= \|e_1\|_{\ell^2} \|r\|_{\ell^2}.
	\]
	We have $\|e_1\|_{\ell^2}=1$ and
	\[
	\|r\|_{\ell^2}^2 = |2|^2 + (n-1)\cdot|1|^2 = 4 + (n-1) = n+3,
	\]
	so $\|r\|_{\ell^2} = \sqrt{n+3}$ and hence \(\|\widetilde{T}(e_1)\|_{S_1} = \sqrt{n+3}\). By symmetry, the same computation shows
	$\|\widetilde{T}(e_j)\|_{S_1} = \sqrt{n+3}$ for all $j=1,\dots,n$, and
	therefore
	\[
	\|\widetilde{T}\|_{\ell^1(\mathbb{C}^n)\to S_1}
	= \max_j \|\widetilde{T}(e_j)\|_{S_1}
	= \sqrt{n+3}.
	\]
	This proves the lemma.
\end{proof}

\begin{thm}\label{thm:negative-1<p<2}
	Let
	\[
	P(z) = z\prod_{j=1}^n (z-z_j), \qquad z_j\in\C\setminus\{0\},
	\]
	and let $w_1,\dots,w_n$ be the zeros of $P'(z)$. Then for every
	$1\le p\le 2$ we have
	\begin{equation}\label{eq:negative-p-between-1-and-2}
		\sum_{k=1}^n \frac{1}{|w_k|^{p}}
		\le
		(n+3)^{\frac{p}{2}}\sum_{j=1}^n \frac{1}{|z_j|^{p}}.
	\end{equation}
\end{thm}

\begin{proof}
As in the proof of Theorem~\ref{thm:negative-interpolation_for_p_2_infinity}, set $u_j := 1/z_j$, $u := (u_1,\dots,u_n)^\top$, and define the linear map
	\begin{equation}\label{eq:definition_of_T_tilde_22}
	\widetilde{T} : \mathbb{C}^n \to \mathbb{C}^{n\times n},\qquad
	\widetilde{T}(u) := \diag(u) + u\mathbf{1}^\top.
\end{equation}
By~\eqref{eq:eig-vs-sing-B}, we have
	\begin{equation}\label{eq:eig-vs-sing-B-2222}
		\sum_{k=1}^n \frac{1}{|w_k|^{p}}
		= \sum_{k=1}^n \bigl|\lambda_k(B)\bigr|^{p}
		\le
		\|B\|_{S_p}^{p}
		= \|\widetilde{T}(u)\|_{S_p}^{p},
		\qquad p\ge1.
	\end{equation}We now estimate $\|\widetilde{T}(u)\|_{S_p}$ in terms of
	$\|u\|_{\ell^p}$ by interpolation between the endpoints $p=1$ and
	$p=2$.

By Lemma~\ref{lem:T-l1-S1} we have \(\|\widetilde{T}\|_{\ell^1(\mathbb{C}^n)\to S_1} = \sqrt{n+3}\). On the other hand, \eqref{eq:Ttilde-l2-S2} gives \(\|\widetilde{T}\|_{\ell^2(\mathbb{C}^n)\to S_2} = \sqrt{n+3}\). Consider now the compatible Banach couples
	\[
	(\ell^1(\mathbb{C}^n),\ell^2(\mathbb{C}^n)),\qquad (S_1,S_2).
	\]
	By complex interpolation (see Theorem \ref{thm:RT_interpolation1} and Remark \ref{rmk:RT_interpolation1}), for each $0<\theta<1$ we have
	\[
	(\ell^1(\mathbb{C}^n),\ell^2(\mathbb{C}^n))_{[\theta]} = \ell^p(\mathbb{C}^n),\qquad
	(S_1,S_2)_{[\theta]} = S_p \quad \text{(with equal norms)},
	\]
	where $p$ is determined by
	\[
	\frac{1}{p} = \frac{1-\theta}{1} + \frac{\theta}{2}
	= 1 - \frac{\theta}{2},
	\quad\text{so that}\quad
	p\in[1,2].
	\]
	Since $\widetilde{T}:\ell^1(\mathbb{C}^n)\to S_1$ and $\widetilde{T}:\ell^2(\mathbb{C}^n)\to S_2$
	are bounded with
	\[
	\|\widetilde{T}\|_{\ell^1(\mathbb{C}^n)\to S_1}
	=\|\widetilde{T}\|_{\ell^2(\mathbb{C}^n)\to S_2} = \sqrt{n+3},
	\]
	the abstract interpolation theorem yields, for all $1\le p\le 2$,
	\begin{equation}\label{eq:T-lp-Sp}
		\|\widetilde{T}(u)\|_{S_p}
		\;\le\; (\sqrt{n+3})^{1-\theta}(\sqrt{n+3})^{\theta}\,\|u\|_{\ell^p}
		= \sqrt{n+3}\,\|u\|_{\ell^p}.
	\end{equation}
Combining \eqref{eq:eig-vs-sing-B-2222} and \eqref{eq:T-lp-Sp} with
	$u_j=1/z_j$ now gives
	\[
	\Biggl(\sum_{k=1}^n \frac{1}{|w_k|^{p}}\Biggr)^{\!1/p}
	\le \|\widetilde{T}(u)\|_{S_p}
	\le \sqrt{n+3}\,\Biggl(\sum_{j=1}^n \frac{1}{|z_j|^{p}}\Biggr)^{\!1/p},
	\qquad 1\le p\le 2.
	\]
	Raising both sides to the $p$th power yields
	\[
	\sum_{k=1}^n \frac{1}{|w_k|^{p}}
	\;\le\;
	(n+3)^{\,p/2}\sum_{j=1}^n \frac{1}{|z_j|^{p}},
	\qquad 1\le p\le 2,
	\]
	which is exactly \eqref{eq:negative-p-between-1-and-2}. This completes the proof.
\end{proof}

\begin{remark}
	It is tempting to hope that the sharp inequality
	\[
	\sum_{k=1}^n \frac{1}{|w_k|}
	\;\le\;
	2\sum_{j=1}^n \frac{1}{|z_j|}
	\]
	from Theorem~\ref{cor:minus-one} might be interpreted as an improved
	endpoint estimate of the form
	\[
	\|\widetilde{T}(u)\|_{S_1} \le 2\|u\|_{\ell^1}
	\]
	and then interpolated with the $\ell^2(\mathbb{C}^n)\to S_2$ bound to produce  stronger
	inequalities for $1<p<2$. However, this is not the case. Indeed, in the operator norm
	\[
	\|\widetilde{T}\|_{\ell^1(\mathbb{C}^n)\to S_1}
	= \sup_{u\neq 0} \frac{\|\widetilde{T}(u)\|_{S_1}}{\|u\|_{\ell^1}},
	\]
	the supremum is taken over all vectors $u\in\C^n$, and
	Lemma~\ref{lem:T-l1-S1} shows that the exact value is $\sqrt{n+3}$.
	In contrast, Theorem~\ref{cor:minus-one} only concerns those vectors
	$u=(1/z_1,\dots,1/z_n)$ arising from polynomials $P(z)$ with
	$z_j\in\C\setminus\{0\}$; in other words, $u$ is restricted to the
	nonlinear set $\widetilde{V}:=\C^n\setminus\{0\} \subset\C^n$. One may view the constant $2$
	as a Lipschitz-type bound for $\widetilde{T}$ on $\widetilde{V}$, but $\widetilde{V}$ is neither
	a linear subspace nor a Banach space (it does not contain $0$ and is not
	closed under addition), so it cannot serve as an endpoint space in the
	standard complex interpolation theory. Thus the sharp factor $2$ at order $-1$ cannot be used as an endpoint
	operator norm in a Banach couple $(\ell^1(\mathbb{C}^n),\ell^2(\mathbb{C}^n))$, and the interpolation
	argument above must rely on the coarser but genuine operator bound
	$\|\widetilde{T}\|_{\ell^1(\mathbb{C}^n)\to S_1} = \sqrt{n+3}$.
\end{remark}

\section{Dual Schoenberg type inequalities}\label{sec: dual cases for order -2}
Conjecture~\ref{conj1.2} prompts us to seek lower bound estimates for $\sum_k 1/|w_k|^2$ in terms of $z_1,\ldots, z_n$. In this section, we consider the dual scenario, namely, we provide some upper bounds for $\sum_j 1/|z_j|^2$ in  terms of $w_1, \ldots, w_n$, where we consider $P(z)$ as a full integral of $ P'(z)$.

\subsection{Matrix integrability and full integrability of polynomials} 

Let \( B \in M_n(\mathbb{C}) \), \( A \in M_{n+1}(\mathbb{C}) \). We adopt the notation and terminology of \cite[Definition~1.3]{DGN21} and \cite[Definition~3]{Bha07}. We say that \(A\) is an \emph{integral of \(B\)} if
\begin{eqnarray}\label{mA}
	A = \begin{bmatrix}
		B & u^T \\
		v & \tau(B)
	\end{bmatrix},
\end{eqnarray}
and also \( p_B(x) = \frac{1}{n+1}p_A'(x) \). In this case, $\tau(B)=\frac{\tr B}{n}$ is the normalized trace of $B$,  the pair of vectors \((u,v)\) is called an \textit{integrator of \( B \)} and the element \(\det(A)\) is called a \textit{constant of integration}. 

The polynomial $P\in \mathbb{C}[x]$ is called a \emph{full integral} of the polynomial $q\in \mathbb{C}[x]$, if $P'=q$ and for any $\lambda\in \mathbb{C}$ satisfying $(x-\lambda)^2|q$ we have that $(x-\lambda)|P$; see \cite[Definition~2.3]{DGN21} or \cite[Definition~1.3]{DG20}. In other words, any multiple zero of the polynomials $q$ is a zero of $P$. We say that $q$ is a polynomial of \emph{type} $(k,m)$, if $q$ has $k$ different simple zeros and $m$ different multiple zeros, where $k$ and $m$ are nonegative integers; see \cite[Definition~2.1]{DGN21} or \cite[Definition 1.1]{DG20}. Specially, if $m\le 1$, then $q$ has a full integral; if $m>k+1$, then  $q$ doesn't have a full integral \cite[Theorem 2.14]{DGN21}. More details about full integrability of polynomials, we refer the readers to \cite{DG20, DGN21}. 

Danielyan,  Guterman and  Ng \cite[Theorem 3.8]{DGN21} asserts that $B$ in \eqref{mA} is integrable if and only if its characteristic polynomial $p_B$ has a full integral. As an application, they derived a dual version of Schoenberg inequality \cite[Theorem 5.1]{DGN21}.

\subsection{Dual Schoenberg type inequality of order $-2$}
We now illustrate our results by obtaining upper bounds for 
\(\sum_{j} 1/|z_j|^2\) in terms of \(w_1,\ldots,w_n\).

\begin{thm} \label{thm:(u,v), upper bounds}
	Let $A$ defind by~\eqref{mA} be an integral of $B=\diag(w_1,\ldots, w_n)$ with an integrator $(u,v)$. Let $p_A=\prod_{j=1}^{n+1}(z-z_j)$ and $p_B=\prod_{k=1}^n(z-w_k)$ be the characteristic polynomials of $A, B$ respectively, where all $z_i, w_k\in \mathbb{C}\setminus \{0\}$. Then  
	\begin{equation}\label{dual1}
		\sum_{j=1}^{n+1} \frac{1}{|z_j|^{2}}
		\le\sum_{k=1}^n\frac{1}{|w_k|^2}+|c|^2 \left(\sum_{k=1}^n \frac{|u_k|^2}{|w_k|^2}  \sum_{l=1}^n \frac{|v_l|^2}{|w_l|^2} + \sum_{k=1}^n\frac{|u_k|^2+|v_k|^2}{|w_k|^2}+1\right) + 2\Re\left( c \sum_{k=1}^n \frac{u_k v_k}{w_k|w_k|^2}  \right),
	\end{equation} where $c:=\left(\tfrac{\sum_{k=1}^n w_k}{n}-\sum_{k=1}^n\tfrac{u_k v_k}{w_k}\right)^{-1}$. Moreover, equality in~\eqref{dual1} holds if and only if $A$ is normal.
\end{thm}
\begin{proof}
	By \eqref{mA}, the inverse of the partitioned matrix $A$ can be written (see \cite[Theorem~2.4]{Zh11}) as
	\begin{align*}
		A^{-1} = \begin{bmatrix}
			B^{-1} + B^{-1} u^\top S^{-1} v B^{-1} & -B^{-1} u^\top S^{-1} \\
			-S^{-1} v B^{-1} & S^{-1}
		\end{bmatrix}, \quad S = \tau(B) - v B^{-1} u^\top.
	\end{align*}
	The eigenvalues of $A^{-1}$ are precisely $\{1/z_1,\ldots,1/z_{n+1}\}$. By Lemma~\ref{lem:weyl} we obtain
	\begin{align}\label{Sch}
		\sum_{j=1}^{n+1}\frac{1}{|z_j|^2}=\sum_{j=1}^{n+1} |\lambda_j(A^{-1})|^2
		\le
		\sum_{j=1}^{n+1} \sigma_j(A^{-1})^2
		&=\|A^{-1}\|_F^2.
	\end{align} with equality if and only if $A^{-1}$ is normal, equivalently, $A$ is normal.
    
Next we compute the Frobenius norm:
	\begin{align}\label{Fro}
		\|A^{-1}\|_F^2=&\|B^{-1} + B^{-1} u^\top S^{-1} v B^{-1}\|_F^2+\|B^{-1} u^\top S^{-1}\|_F^2+\|S^{-1} v B^{-1}\|_F^2+\|S^{-1}\|_F^2.
	\end{align}
	Let $D=B^{-1}=\diag(d_1,\ldots,d_n)$, $u=(u_1,\ldots,u_n)$ and $v=(v_1,\ldots,v_n)$. Then $c:=S^{-1}=\left(\tfrac{\sum_{k=1}^n w_k}{n}-\sum_{k=1}^n u_k d_k v_k\right)^{-1}$. A direct computation yields
	\begin{align*}
		\|B^{-1} + B^{-1} u^\top S^{-1} v B^{-1}\|_F^2&=\|D + c D u^\top  v D\|_F^2\\
		&=\|D\|_F^2+|c|^2\|Du^\top v D\|_F^2+2\Re \tr \left(c D^* Du^\top vD \right)\\
		&=\sum_{k=1}^n |d_k|^2 + |c|^2 \left( \sum_{k=1}^n |d_k u_k|^2 \right) \left( \sum_{l=1}^n |d_l v_l|^2 \right) + 2\Re\left( c \sum_{k=1}^n u_k v_k |d_k|^2 d_k \right),\\
		\|B^{-1} u^\top S^{-1}\|_F^2&=|c|^2\|Du^\top\|_F^2=|c|^2\sum_{k=1}^n|d_k|^2|u_k|^2,\\
		\|S^{-1} v B^{-1}\|_F^2&=|c|^2\|vD\|_F^2=|c|^2\sum_{k=1}^n|d_k|^2|v_k|^2,\\
		\|S^{-1}\|_F^2&=|c|^2.
	\end{align*}
Combining these with \eqref{Sch} and \eqref{Fro}, and using $d_k=1/w_k$, we obtain exactly~\eqref{dual1}.
\end{proof}
We now specialize the choice of $B$ in~\eqref{mA}. For simplicity, we write
\begin{align*}
	B &=\diag (w_1,\ldots, w_n)\\ 
	&=\operatorname{diag} \, (\underbrace{b_1, \ldots, b_1}_{\alpha_1} , \ldots, 
	\underbrace{b_m, \ldots, b_m}_{\alpha_m}, a_1, \ldots, a_k),
\end{align*}where $a_1, \ldots, a_k, b_1, \ldots, b_m \in \C$ are pairwise distinct, 
$k,m \geq 0$, and $\alpha_1, \ldots, \alpha_m \geq 2$.

Danielyan, Guterman and Ng~\cite[Corollary~3.12]{DGN21} showed that if $p_B$ has a full integral $\tfrac{1}{n+1}p_A$, then $(u,v)$ can be chosen as
\begin{align}
u=v=\Big(
	\underbrace{0,\ldots,0}_{\sum_{s=1}^m\alpha_s}, \sqrt{\tfrac{p_A(a_1)}{h_{a_1}(a_1)}}, \ldots,\sqrt{\tfrac{p_A(a_k)}{h_{a_k}(a_k)}}\Big), \tag{Ch1}\label{cho1}
\end{align}
where $h(x)=\prod_{l=1}^k(x-a_l)$, $h_{a_i}(x)=\frac{h(x)}{x-a_i}$ and $\sqrt{\frac{p_A(a_i)}{h_{a_i}(a_i)}}$ represtents  any square root of $\frac{p_A(a_i)}{h_{a_i}(a_i)}$ for any $1\le i\le k$. It should be noted that the choice \eqref{cho1} minimizes the Frobenius norm of $A$; see \cite[Corollary~3.12]{DGN21}. 

In Theorem~\ref{thm:(u,v), upper bounds}, taking $(u,v)$ as in Choice~\eqref{cho1} yields the following corollary.

\begin{cor}
	Let
\[
  P(z) =  \prod_{j=1}^{n+1} (z - z_j),
\]
be a full integral of $P'(z)$, and let the zeros of $P'(z)$ be 
$a_1, \ldots, a_k, b_1, \ldots, b_m$, with multiplicities 
\(\underbrace{1,\ldots, 1}_k,\alpha_1,\ldots, \alpha_m\), respectively, 
where $z_j,a_j,b_j\in \mathbb{C}\setminus \{0\}$, $k, m \geq 0$, and $\alpha_1, \ldots, \alpha_m \geq 2$. 
Then 
	\begin{align*}
		\sum_{j=1}^{n+1} \frac{1}{|z_j|^{2}}
		\le&\sum_{j=1}^m \frac{\alpha_j}{|b_j|^2} + \sum_{j=1}^k \frac{1}{|a_j|^2} + |c|^2 \left( 1 + \sum_{j=1}^k \frac{|r_j|}{|a_j|^2} \right)^2 + 2\Re \left( c \sum_{j=1}^k \frac{r_j}{a_j |a_j|^2} \right),
	\end{align*}where $c:=\left(G-\sum_{j=1}^k\frac{r_j}{a_j}\right)^{-1}$, $r_j=\frac{P(a_j)}{\prod_{l\neq j}(a_j-a_l)}$ and $G=\frac{\sum_{j=1}^m\alpha_jb_j+\sum_{j=1}^ka_j}{n}$. Moreover, equality holds if and only if $r_j(\overline{a_j-G})^2$ is non-negative for every $j=1,\ldots, k$.
\end{cor}


\begin{proof}Taking $(u,v)$ as in Choice~\eqref{cho1} in Theorem~\ref{thm:(u,v), upper bounds} yields the desired inequality. For the equality condition, note that equality holds if and only if $A$ is normal, i.e., $AA^* = A^* A$. 
Here and in what follows, for a complex matrix $M$ and a complex vector $x$ we write $\overline{M}$ and $\overline{x}$ for their entrywise complex conjugates. 
A direct computation shows that  
\begin{eqnarray}\label{e:AA^*}
		A^* A = \begin{pmatrix} 
			\overline{B}B + \overline{v}^T v & \overline{B}u^T + \tau(B)\overline{v}^T \\ 
			\overline{u}B + \overline{\tau(B)}v & \overline{u}u^T + |\tau(B)|^2 
		\end{pmatrix}, \quad 
		AA^* = \begin{pmatrix} 
			B\overline{B} + u^T\overline{u} & B\overline{v}^T + \overline{\tau(B)}u^T \\ 
			v\overline{B} + \tau(B)\overline{u} & v\overline{v}^T + |\tau(B)|^2 
		\end{pmatrix}.
\end{eqnarray}
Since $B\overline{B} = \overline{B}B$ and $u=v$ in Choice~\eqref{cho1}, the difference simplifies to  
	$$
	AA^* - A^* A = \begin{pmatrix} 
		O & \left(B - \tau(B)\right) \overline{v}^T - \left(\overline{B} - \overline{\tau(B)}\right)v^T \\ 
		v \left(\overline{B} - \overline{\tau(B)}\right) - \overline{v}\left(B - \tau(B)\right) & O 
	\end{pmatrix}.
	$$Thus $A$ is normal if and only if this difference is the zero matrix. Since $B$ is diagonal, this is equivalent to  
\[
  v_{d_i} (\overline{a_i - \tau(B)}) = \overline{v_{d_i}} (a_i - \tau(B)) \quad \text{for all } i = 1, \ldots, k,
\]
where $d_i=\sum_{j=1}^m\alpha_j+i$.  
Substituting the expression for $v_{d_i}$, we obtain the equivalent condition  
\[
  \sqrt{\frac{P(a_i)}{h_{a_i}(a_i)}}  (\overline{a_i - \tau(B)}) 
  = \overline{\left( \sqrt{\frac{P(a_i)}{h_{a_i}(a_i)}}\right)} (a_i - \tau(B))  
  \quad \text{for all } i = 1, \ldots, k,
\]
which means $\sqrt{r_i}(\overline{a_i-G}) $ is real, equivalently, $r_i(\overline{a_i-G})^2$ is non-negative for all $1\le i\le k$.
\end{proof}


Specially,  if all the critical points of a polynomial $P$ are distinct, then $P$ is a full integral of $P'$. It then follows from
\cite[Theorem 3.13, case 1]{DGN21} that we have the following dual Schoenberg type inequality for order $-2$.
\begin{cor}\label{cor:dual -2}
Let
\[
  P(z) = \prod_{j=1}^{n+1} (z - z_j),
\]
and let $w_1, \ldots, w_n$ be the zeros of $P'(z)$.
Assume that all $z_j,w_k$ lie in $\C \setminus \{0\}$ and that $w_1,\ldots,w_n$ are distinct. Then
\[
\sum_{j=1}^{n+1} \frac{1}{|z_j|^{2}}
\le
\sum_{k=1}^n \frac{1}{|w_k|^2}
+ |c|^2 \left( 1 + \sum_{k=1}^n \frac{|r_k|}{|w_k|^2} \right)^2
+ 2\Re \left( c \sum_{k=1}^n \frac{r_k}{w_k |w_k|^2} \right), 
\]
where
\[
c := \left(G - \sum_{k=1}^n\frac{r_k}{w_k}\right)^{-1}, \qquad
G = \frac{1}{n}\sum_{k=1}^n w_k, \qquad
r_k = \frac{(n+1)P(w_k)}{P''(w_k)}.
\]
Moreover, equality is attained if and only if $r_k(\overline{w_k-G})^2$ is non-negative for every $k=1,\ldots,n$.
\end{cor}

\section*{Acknowledgments}
The second author  was supported by Young Elite Scientists Sponsorship Program for PhD Students (China Association for Science
and Technology)  and the Fundamental Research Funds for the Central Universities at Xi'an Jiaotong University (Grant No. xzy022024045). The authors are grateful to Prof.~Minghua Lin and Jiachen Shi for numerous insightful discussions and valuable suggestions, which greatly improved the presentation of this paper. The authors are also thankful to Zehao Zhang for providing an elegant proof of Lemma~\ref{lem:real2} via a completing-the-square argument. Finally, the authors would like to thank Shengtong Zhang for his helpful remarks on the results.

\appendix

\section{Proof of the sum-of-squares identity~\eqref{eq:sos-identity}}\label{app:identity-proof}

Throughout this section we set $m:=n-1$ and work with
\[
S_1=\sum_{i=1}^{m}a_i,\qquad
S_2=\sum_{i=1}^{m}a_i^2,\qquad
S_3=\sum_{i=1}^{m}a_i^3,\qquad
S_4=\sum_{i=1}^{m}a_i^4.
\]
All pairwise sums $\sum_{i<j}$ are taken over $1\le i<j\le m$, 
and by $\sum_{i\ne j}$ we mean the sum over all ordered pairs $(i,j)$ with $1\le i,j\le m$ and $i\ne j$. 
Similarly, $\sum_{k\ne i,j}$ denotes the sum over all indices $k$ satisfying $1\le k\le m$ and $k\ne i,j$.

\medskip
First, we prove the following five useful identities~(A1)--(A5).

\medskip
\noindent\textbf{(A1)} \(\displaystyle \sum_{i<j}(a_i-a_j)^2 = m S_2 - S_1^2\).

\begin{proof}
We use the standard symmetrization:
\[
\sum_{i<j}(a_i-a_j)^2
=\frac12\sum_{i\ne j}(a_i-a_j)^2
=\frac12\sum_{i\ne j}(a_i^2+a_j^2-2a_i a_j).
\]
The three terms give, respectively,
\(\sum_{i\ne j}a_i^2=(m-1)\sum_i a_i^2=(m-1)S_2\),
the same for \(\sum_{i\ne j}a_j^2\), and
\(\sum_{i\ne j}2a_i a_j=2\left((\sum_i a_i)^2-\sum_i a_i^2\right)=2(S_1^2-S_2)\).
Altogether,
\[
\sum_{i<j}(a_i-a_j)^2
=\tfrac12\left(2(m-1)S_2-2(S_1^2-S_2)\right)=mS_2-S_1^2.\qedhere
\]
\end{proof}

\medskip
\noindent\textbf{(A2)} \(\displaystyle \sum_{i<j}(a_i-a_j)^2(a_i+a_j) = m S_3 - S_1 S_2\).

\begin{proof}
Expand
\((a_i-a_j)^2(a_i+a_j)=(a_i^3+a_j^3)-a_i a_j(a_i+a_j)\).
Hence
\[
\sum_{i<j}(a_i-a_j)^2(a_i+a_j)
=\underbrace{\sum_{i<j}(a_i^3+a_j^3)}_{=(m-1)S_3}
-\underbrace{\sum_{i<j}(a_i^2 a_j+a_i a_j^2)}_{=\sum_{i\ne j}a_i^2 a_j}.
\]
Now \(\sum_{i\ne j}a_i^2 a_j=\sum_i a_i^2\sum_{j\ne i}a_j
=\sum_i a_i^2(S_1 -a_i)=S_1 S_2-S_3\).
Therefore the sum equals
\((m-1)S_3-(S_1 S_2-S_3)=mS_3-S_1 S_2\).
\end{proof}

\medskip
\noindent\textbf{(A3)} \(\displaystyle \sum_{i<j}(a_i-a_j)^2(a_i+a_j)^2 = m S_4 - S_2^{2}\).

\begin{proof}
Use \((a_i-a_j)^2(a_i+a_j)^2=(a_i^2-a_j^2)^2=(a_i^2+a_j^2)^2-(2a_i a_j)^2\).
Thus
\[
\sum_{i<j}(a_i-a_j)^2(a_i+a_j)^2
=\underbrace{\sum_{i<j}(a_i^2+a_j^2)^2}_{\mathbf{I}}
-\underbrace{\sum_{i<j}4a_i^2 a_j^2}_{\mathbf{II}}.
\]
For \(\mathbf{I}\):
\[
\sum_{i<j}(a_i^2+a_j^2)^2
=\sum_{i<j}(a_i^4+a_j^4)+2\sum_{i<j}a_i^2 a_j^2
=(m-1)S_4+\left(S_2^2-S_4\right)=(m-2)S_4+S_2^2.
\]
For \(\mathbf{II}\): since \(S_2^2=S_4+2\sum_{i<j}a_i^2 a_j^2\),
we have \(\sum_{i<j}4a_i^2 a_j^2=2(S_2^2-S_4)\).
Subtracting gives \(mS_4-S_2^2\).
\end{proof}

\medskip
\noindent\textbf{(A4)} \(\displaystyle
\sum_{i<j}(a_i-a_j)^2\sum_{k\ne i,j} a_k^2
= (m-1)S_2^{2} - S_1^2 S_2 - m S_4 + 2 S_1 S_3.
\)

\begin{proof}
Note \(\sum_{k\ne i,j}a_k^2=S_2-a_i^2-a_j^2\). Hence
\[
\sum_{i<j}(a_i-a_j)^2 \sum_{k\ne i,j}a_k^2
=S_2 \sum_{i<j}(a_i-a_j)^2
-\sum_{i<j}(a_i-a_j)^2(a_i^2+a_j^2).
\]
The first term is \(S_2(mS_2-S_1^2)\) by (A1). For the second term, use
\[
(a_i-a_j)^2(a_i^2+a_j^2)=(a_i^2+a_j^2)^2-2a_i a_j(a_i^2+a_j^2),
\]
so
\[
\sum_{i<j}(a_i-a_j)^2(a_i^2+a_j^2)
=(m-2)S_4+S_2^2-\underbrace{\sum_{i<j}2(a_i^3 a_j+a_i a_j^3)}_{=2\sum_{i\ne j}a_i^3 a_j}.
\]
But \(\sum_{i\ne j}a_i^3 a_j=\sum_i a_i^3\sum_{j\ne i}a_j
=\sum_i a_i^3(S_1 -a_i)=S_1 S_3-S_4\). Therefore the second term equals
\((m-2)S_4+S_2^2-2(S_1 S_3-S_4)=mS_4+S_2^2-2S_1 S_3\).
Subtracting from the first term gives
\[
(mS_2^2-S_1^2S_2)-\left(mS_4+S_2^2-2S_1 S_3\right)
=(m-1)S_2^2-S_1^2S_2-mS_4+2S_1 S_3.
\qedhere\]
\end{proof}

\medskip
\noindent\textbf{(A5)} \(\displaystyle
\sum_{i<j}(a_i-a_j)^2 \left(\sum_{k\ne i,j} a_k\right)^{2}
= m S_4 - S_2^{2} + (m+2) S_1^2 S_2 - S_1^4 - 2m S_1 S_3.
\)

\begin{proof}
Let \(T_{ij}:=\sum_{k\ne i,j}a_k=S_1 -a_i-a_j\). Then
\[
\sum_{i<j}(a_i-a_j)^2 T_{ij}^2
=S_1^2 \sum_{i<j}(a_i-a_j)^2
-2S_1 \sum_{i<j}(a_i-a_j)^2(a_i+a_j)
+\sum_{i<j}(a_i-a_j)^2(a_i+a_j)^2.
\]
Apply (A1)--(A3) to the three sums:
\[
S_1^2(mS_2-S_1^2)-2S_1(mS_3-S_1 S_2)+(mS_4-S_2^2)
= mS_4-S_2^2+(m+2)S_1^2 S_2-S_1^4-2m S_1 S_3.
\qedhere\]
\end{proof}

We now proceed to establish the sum-of-squares identity~\eqref{eq:sos-identity}. In terms of $m$, the identity \eqref{eq:sos-identity} reads
\begin{equation}\label{eq:sos-identity-m}
\frac{m^2-m+1}{m(m+1)}(S_2+S_1^2)^2-(S_4+S_1^4)
=\frac{2}{(m+1)m(m-2)}\sum_{i<j}(a_i-a_j)^2
\left[Q_{ij}^2+\frac{m}{2}\,R_{ij}\right],
\end{equation}
where
\[
Q_{ij}:=a_i+a_j+(1-m)S_1,\qquad
R_{ij}:=(m-2) \sum_{k\ne i,j} a_k^2-\left(\sum_{k\ne i,j}a_k\right)^{2}.
\]
Set
\[
\Sigma_Q:=\sum_{i<j}(a_i-a_j)^2 Q_{ij}^2,\qquad
\Sigma_R:=\sum_{i<j}(a_i-a_j)^2 R_{ij}.
\]
Expanding $Q_{ij}^2$ and invoking (A1)--(A3) yields
\begin{align}
\Sigma_Q
&=\sum_{i<j}(a_i-a_j)^2(a_i+a_j)^2
+2(1-m)S_1\sum_{i<j}(a_i-a_j)^2(a_i+a_j)
+(1-m)^2S_1^2\sum_{i<j}(a_i-a_j)^2 \notag\\
&=(mS_4-S_2^2)+2(1-m)S_1 (mS_3-S_1S_2)+(1-m)^2S_1^2 (mS_2-S_1^2) \notag\\
&=mS_4-S_2^2-2m(m-1)S_1 S_3+(m^2-m+2)(m-1)S_1^2 S_2-(m-1)^2S_1^4 .
\label{eq:SigmaQ}
\end{align}
For the $R_{ij}$--part, using its definition together with (A4)--(A5) gives
\begin{align}
\Sigma_R
&=(m-2)\sum_{i<j}(a_i-a_j)^2\sum_{k\ne i,j}a_k^2
-\sum_{i<j}(a_i-a_j)^2\left(\sum_{k\ne i,j}a_k\right)^2 \notag\\
&=(m-2)\left[(m-1)S_2^2-S_1^2S_2-mS_4+2S_1 S_3\right]
-\left[mS_4-S_2^2+(m+2)S_1^2S_2-S_1^4-2mS_1 S_3\right] \notag\\
&=(m^2-3m+3)S_2^2-(m^2-m)S_4-2mS_1^2S_2+4(m-1)S_1 S_3+S_1^4 .
\label{eq:SigmaR}
\end{align}
Therefore the right-hand side of \eqref{eq:sos-identity-m} equals
\[
\frac{2}{(m+1)m(m-2)} \Sigma_Q
+\frac{m}{(m+1)m(m-2)} \Sigma_R
=\frac{1}{(m+1)m(m-2)}\left(2\Sigma_Q+m\Sigma_R\right).
\]
Substituting \eqref{eq:SigmaQ}--\eqref{eq:SigmaR} and collecting like terms, we obtain
\begin{align*}
2\Sigma_Q+m\Sigma_R
&=\left[2m-(m^3-m^2) \right]S_4
+\left[-2+(m^3-3m^2+3m) \right]S_2^2 +\left[2\left((m-1)^3-1\right) \right]S_1^2 S_2 \\
&\qquad
+\left[-4m(m-1)+4m(m-1) \right]S_1 S_3
+\left[-2(m-1)^2+m \right]S_1^4 \\
&= -m(m-2)(m+1) S_4
+\left(m^3-3m^2+3m-2\right) S_2^2
+2\left(m^3-3m^2+3m-2\right) S_1^2 S_2 \\
&\qquad+\left(-2m^2+5m-2\right) S_1^4,
\end{align*}
where the $S_1 S_3$--terms cancel exactly. Dividing by $(m+1)m(m-2)$ gives
\[
\frac{2\Sigma_Q+m\Sigma_R}{(m+1)m(m-2)}
= -S_4
+\frac{m^2-m+1}{m(m+1)} S_2^2
+\frac{2(m^2-m+1)}{m(m+1)} S_1^2 S_2
+\frac{-2m+1}{m(m+1)} S_1^4.
\]
Finally,
\[
\frac{m^2-m+1}{m(m+1)}(S_2+S_1^2)^2-(S_4+S_1^4)
= -S_4+\frac{m^2-m+1}{m(m+1)}\left(S_2^2+2S_1^2S_2+S_1^4\right)-S_1^4,
\]
whose right-hand side is exactly the expression we just obtained. This proves
\eqref{eq:sos-identity-m}, hence~\eqref{eq:sos-identity}. \qed

\section{Proof of the identities~\eqref{eq:reciprocal-1} and~\eqref{eq:reciprocal-2}}\label{sec:appendix_reciprocal}

In this section we derive the identities
\begin{equation*}
  \sum_{k=1}^{n} \frac{1}{w_k}
  = 2 \sum_{j=1}^{n} \frac{1}{z_j},
  \qquad
  \sum_{k=1}^{n} \frac{1}{w_k^{2}}
  = 3 \sum_{j=1}^{n} \frac{1}{z_j^{2}}
  + \left( \sum_{j=1}^{n} \frac{1}{z_j} \right)^{2},
\end{equation*}
where
\[
  P(z) = z \prod_{j=1}^{n} (z - z_j), \qquad z_j \in \C \setminus \{0\},
\]
and $w_1,\dots,w_n$ are the zeros of $P'(z)$.

Set
\[
  Q(z) := \prod_{j=1}^{n} (z - z_j),
\]
so that $P(z) = z Q(z)$.
Then
\[
  P'(z) = Q(z) + z Q'(z), \qquad
  P''(z) = 2 Q'(z) + z Q''(z), \qquad
  P^{(3)}(z) = 3 Q''(z) + z Q^{(3)}(z).
\]
Evaluating at $z=0$ gives
\begin{equation}\label{eq:P-derivatives-at-0}
  P'(0) = Q(0), \qquad
  P''(0) = 2 Q'(0), \qquad
  P^{(3)}(0) = 3 Q''(0).
\end{equation}On the other hand, the logarithmic derivative of $Q$ is
\[
  \frac{Q'(z)}{Q(z)} = \sum_{j=1}^{n} \frac{1}{z - z_j}.
\]
Evaluating at $z=0$ yields
\begin{equation}\label{eq:Qprime-over-Q}
  \frac{Q'(0)}{Q(0)} = \sum_{j=1}^{n} \frac{1}{0 - z_j}
  = - \sum_{j=1}^{n} \frac{1}{z_j}.
\end{equation}
Differentiating the logarithmic derivative, we obtain
\[
  \left( \frac{Q'}{Q} \right)'(z)
  = \frac{Q''(z)}{Q(z)} - \left( \frac{Q'(z)}{Q(z)} \right)^{2}
  = - \sum_{j=1}^{n} \frac{1}{(z - z_j)^2}.
\]
Setting $z=0$ gives
\begin{equation}\label{eq:Qdouble-over-Q}
  \frac{Q''(0)}{Q(0)} - \left( \frac{Q'(0)}{Q(0)} \right)^2
  = - \sum_{j=1}^{n} \frac{1}{z_j^2},
\end{equation}
and combining \eqref{eq:Qprime-over-Q} and \eqref{eq:Qdouble-over-Q} yields
\begin{equation}\label{eq:Qdouble}
  \frac{Q''(0)}{Q(0)}
  = \left( \sum_{j=1}^{n} \frac{1}{z_j} \right)^2
    - \sum_{j=1}^{n} \frac{1}{z_j^2}.
\end{equation}

Next, since $P'(z)$ has zeros $w_1,\dots,w_n$, we have
\[
  P'(z) = (n+1)\prod_{k=1}^{n} (z - w_k),
\]
so that
\[
  \frac{P''(z)}{P'(z)} = \sum_{k=1}^{n} \frac{1}{z - w_k}.
\]
Evaluating at $z=0$ gives
\begin{equation}\label{eq:sum-1-over-w-basic}
  \frac{P''(0)}{P'(0)} = \sum_{k=1}^{n} \frac{1}{0 - w_k}
  = - \sum_{k=1}^{n} \frac{1}{w_k}.
\end{equation}
Differentiating $\frac{P''}{P'}$ yields
\[
  \left( \frac{P''}{P'} \right)'(z)
  = \frac{P^{(3)}(z)}{P'(z)} - \left( \frac{P''(z)}{P'(z)} \right)^2
  = - \sum_{k=1}^{n} \frac{1}{(z - w_k)^2}.
\]
Setting $z=0$ we obtain
\begin{equation}\label{eq:sum-1-over-w2-basic}
  \frac{P^{(3)}(0)}{P'(0)} - \left( \frac{P''(0)}{P'(0)} \right)^2
  = - \sum_{k=1}^{n} \frac{1}{w_k^2},
\end{equation}
or equivalently
\begin{equation}\label{eq:sum-1-over-w2}
  \sum_{k=1}^{n} \frac{1}{w_k^2}
  = \left( \frac{P''(0)}{P'(0)} \right)^2 - \frac{P^{(3)}(0)}{P'(0)}.
\end{equation}

We now substitute \eqref{eq:P-derivatives-at-0}, \eqref{eq:Qprime-over-Q}, and \eqref{eq:Qdouble}
into \eqref{eq:sum-1-over-w-basic} and \eqref{eq:sum-1-over-w2}.
From \eqref{eq:P-derivatives-at-0} and \eqref{eq:Qprime-over-Q} we obtain
\[
  \frac{P''(0)}{P'(0)}
  = \frac{2 Q'(0)}{Q(0)}
  = -2 \sum_{j=1}^{n} \frac{1}{z_j},
\]
and hence, by \eqref{eq:sum-1-over-w-basic},
\[
  \sum_{k=1}^{n} \frac{1}{w_k}
  = 2 \sum_{j=1}^{n} \frac{1}{z_j},
\]
which is \eqref{eq:reciprocal-1}.

Similarly, from \eqref{eq:P-derivatives-at-0} and \eqref{eq:Qdouble} we get
\[
  \frac{P^{(3)}(0)}{P'(0)}
  = \frac{3Q''(0)}{Q(0)}
  = 3 \left( \left( \sum_{j=1}^{n} \frac{1}{z_j} \right)^2
            - \sum_{j=1}^{n} \frac{1}{z_j^2} \right).
\]
Substituting this and $\frac{P''(0)}{P'(0)} = -2 \sum_{j=1}^{n} \frac{1}{z_j}$ into
\eqref{eq:sum-1-over-w2}, we obtain
\begin{align*}
  \sum_{k=1}^{n} \frac{1}{w_k^2}
  &= \left( -2 \sum_{j=1}^{n} \frac{1}{z_j} \right)^2
     - 3 \left( \left( \sum_{j=1}^{n} \frac{1}{z_j} \right)^2
               - \sum_{j=1}^{n} \frac{1}{z_j^2} \right) \\
  &= \left( \sum_{j=1}^{n} \frac{1}{z_j} \right)^2
     + 3 \sum_{j=1}^{n} \frac{1}{z_j^2},
\end{align*}
which is precisely \eqref{eq:reciprocal-2}.\qed

\section{Verification of Conjecture~\ref{conj: sum z_j=0} in the case $p=6$}\label{sec:appendix_case_p_6}

In this section we prove that Conjecture~\ref{conj: sum z_j=0} holds in the case $p = 6$.
We begin by recalling the following Schoenberg type inequality of order~$6$, established in~\cite[Theorem~2.4]{Tang25}.

\begin{thm}[\cite{Tang25}]\label{thm:inequality_of_order_6_1}
Let \( z_1, z_2, \ldots, z_n \) be the zeros of \( p(z) \), and let \( w_1, w_2, \ldots, w_{n-1} \) be its critical points. Assume further that $\sum_{j=1}^n z_j=0$. Then 
\begin{align*}
\sum_{k=1}^{n-1} |w_k|^6 \leq\ 
&\frac{n - 6}{n} \sum_{j=1}^n |z_j|^6
+ \frac{6}{n^2} \left( \sum_{j=1}^n |z_j|^4 \right) \left( \sum_{j=1}^n |z_j|^2 \right) \\
&+ \frac{3}{n^2} \left| \sum_{j=1}^n z_j |z_j|^2 \right|^2
- \frac{2}{n^3} \left( \sum_{j=1}^n |z_j|^2 \right)^3,
\end{align*}
with equality if and only if all \( z_j \) lie on a straight line in the complex plane.
\end{thm}

We also need the following auxiliary inequality for the configuration of zeros.

\begin{thm}\label{thm:inequality_of_order_6_22}
Let $n\ge 1$ and let $z_1,\dots,z_n\in\mathbb{C}$ satisfy $\sum_{j=1}^n z_j = 0$. Then the following inequality holds:
\begin{equation}\label{eq:main-ineq}
\frac{4}{n} \sum_{j=1}^n |z_j|^6
- \frac{6}{n^2} \left( \sum_{j=1}^n |z_j|^4 \right) \left( \sum_{j=1}^n |z_j|^2 \right)
- \frac{3}{n^2} \left| \sum_{j=1}^n z_j |z_j|^2 \right|^2
+ \frac{2}{n^3} \left( \sum_{j=1}^n |z_j|^2 \right)^3
\ge 0.
\end{equation}
\end{thm}

\begin{proof}
We first rewrite the inequality in a probabilistic form. Let $\Omega=\{1,\dots,n\}$ and equip $\Omega$ with the probability measure
\[
\mathbb{P}(\{j\}) = \frac{1}{n}\qquad (j=1,\dots,n).
\]
Define a complex-valued random variable $Z$ on $\Omega$ by $Z(j) = z_j$. Then
\[
\mathbb{E}Z = \frac{1}{n}\sum_{j=1}^n z_j = 0
\]
by the hypothesis $\sum_{j=1}^n z_j = 0$.


Introduce the moments $m_2 := \mathbb{E}|Z|^2$, $m_4 := \mathbb{E}|Z|^4$, $m_6 := \mathbb{E}|Z|^6$ and $u := \mathbb{E}\left(Z|Z|^2\right)$. By definition of $Z$, we have
\[
m_2 = \frac{\sum_{j=1}^n |z_j|^2}{n},\qquad
m_4 = \frac{\sum_{j=1}^n |z_j|^4}{n},\qquad
m_6 = \frac{\sum_{j=1}^n |z_j|^6}{n},\qquad
u = \frac{\sum_{j=1}^n z_j |z_j|^2}{n}.
\]Substituting these into \eqref{eq:main-ineq} and multiplying by $n$ gives the equivalent form
\[
4 m_6
- 6 m_4 m_2
- 3 |u|^2
+ 2 m_2^3
\ge 0.
\]
For convenience, set
\[
F_0 := 4 m_6 - 6 m_4 m_2 - 3 |u|^2 + 2 m_2^3.
\]
Thus it suffices to prove
\[
F_0  \ge  0.
\]

If $S_2=0$, then $m_2=0$ and hence $|z_j|=0$ for all $j$, so all quantities in
\eqref{eq:main-ineq} vanish and the inequality holds with equality.
Therefore we may and do assume $S_2>0$, i.e. $m_2>0$.

Consider the Hilbert space $L^2(\Omega,\mathbb{P})$ with the usual inner product
\[
\langle f,g\rangle = \mathbb{E}\bigl(f\,\overline{g}\bigr).
\]
Define the three functions
\[
f_0 := 1,\qquad f_1 := Z,\qquad f_2 := Z|Z|^2.
\]We compute:
\begin{align*}
\langle f_0,f_0\rangle &= \mathbb{E}(1\cdot\overline{1}) = 1,\\
\langle f_0,f_1\rangle &= \mathbb{E}(1\cdot\overline{Z})
= \mathbb{E}(\overline{Z}) = \overline{\mathbb{E}Z} = 0,\\
\langle f_0,f_2\rangle
&= \mathbb{E}\bigl(1\cdot\overline{Z|Z|^2}\bigr)
= \mathbb{E}\bigl(\overline{Z}|Z|^2\bigr)
= \overline{\mathbb{E}\bigl(Z|Z|^2\bigr)} = \overline{u},\\
\langle f_1,f_1\rangle &= \mathbb{E}(Z\overline{Z}) = \mathbb{E}|Z|^2 = m_2,\\
\langle f_1,f_2\rangle
&= \mathbb{E}\bigl(Z\,\overline{Z|Z|^2}\bigr)
= \mathbb{E}\bigl(|Z|^4\bigr) = m_4,\\[2pt]
\langle f_2,f_2\rangle
&= \mathbb{E}\bigl(Z|Z|^2\,\overline{Z|Z|^2}\bigr)
= \mathbb{E}|Z|^6 = m_6.
\end{align*}
Hence the Gram matrix $G = ( \langle f_i,f_j\rangle )_{i,j=0}^2$ is
\[
G =
\begin{pmatrix}
\langle f_0,f_0\rangle & \langle f_0,f_1\rangle & \langle f_0,f_2\rangle\\
\langle f_1,f_0\rangle & \langle f_1,f_1\rangle & \langle f_1,f_2\rangle\\
\langle f_2,f_0\rangle & \langle f_2,f_1\rangle & \langle f_2,f_2\rangle
\end{pmatrix} =
\begin{pmatrix}
1 & 0 & \overline{u}\\
0 & m_2 & m_4\\
u & m_4 & m_6
\end{pmatrix}.
\]A direct computation of the determinant gives
\[
\det G
= m_2 m_6 - m_4^2 - m_2 |u|^2.
\]Since $G$ is a Gram matrix in a Hilbert space, it is positive semidefinite. In particular, its determinant is nonnegative:
\begin{equation}\label{eq:J-ge0}
J_0 := \det G = m_2 m_6 - m_4^2 - m_2 |u|^2 \ge 0.
\end{equation}

We now relate $F_0$ and $J_0$ via an algebraic identity. Since $m_2>0$, we may consider $\frac{3}{m_2}J_0 = 3m_6 - \frac{3}{m_2}m_4^2 - 3|u|^2$. Subtracting this from $F_0$ yields
\begin{equation}\label{eq:B2-def}
B_0 := F_0 - \frac{3}{m_2}J_0
= m_6 - 6m_2m_4 + 2m_2^3 + \frac{3}{m_2}m_4^2.
\end{equation}Thus we can write
\[
F_0 = \frac{3}{m_2}J_0 + B_2.
\]
By \eqref{eq:J-ge0} and $m_2>0$, we already know that $\frac{3}{m_2}J_0\ge 0$. Thus it remains to show that
\[
B_0 \ge 0.
\]

Observe that $m_2,m_4,m_6$ depend only on the nonnegative random variable
\[
X := |Z|^2 \ge 0.
\]
Indeed,
\[
m_2 = \mathbb{E}X,\qquad
m_4 = \mathbb{E}X^2,\qquad
m_6 = \mathbb{E}X^3.
\]The function $\varphi(x)=x^3$ is convex on $[0,\infty)$, so Jensen's inequality implies
\begin{equation}\label{eq:Jensen}
m_6 = \mathbb{E}X^3 \ge \bigl(\mathbb{E}X\bigr)^3 = m_2^3.
\end{equation}
Using \eqref{eq:Jensen} in \eqref{eq:B2-def}, we obtain
\[
B_0
= m_6 - 6m_2m_4 + 2m_2^3 + \frac{3}{m_2}m_4^2
\ge
m_2^3 - 6m_2m_4 + 2m_2^3 + \frac{3}{m_2}m_4^2 = \frac{3}{m_2}\left(m_2^2 - m_4\right)^2.
\]
Hence $B_0 \ge 0$. Therefore $F_0 \ge 0$. Recalling that $F_0\ge 0$ is equivalent to \eqref{eq:main-ineq}, this completes the proof.
\end{proof}

Combining Theorems~\ref{thm:inequality_of_order_6_1} and~\ref{thm:inequality_of_order_6_22}, we obtain the following corollary, which confirms Conjecture~\ref{conj: sum z_j=0} in the case $p = 6$.

\begin{cor}\label{cor:inequality_of_order_6_11}
Let \( z_1, z_2, \ldots, z_n \) be the zeros of \( p(z) \), and let \( w_1, w_2, \ldots, w_{n-1} \) be its critical points. Assume further that $\sum_{j=1}^n z_j=0$. Then for all $n\ge 4$, \[ \sum_{k=1}^{n-1} |w_k|^6 \le \frac{n-2}{n} \sum_{j=1}^n |z_j|^6 . \]
\end{cor}

\section{Proof of Theorem~\ref{thm:refined-interpolation}}
\label{appendix:refined-interpolation}

\begin{lem}\label{lem:Te-2m-endpoint}
Let $n\geq 2$ and let
\[
\widetilde{T} : \C^n \to \C^{n\times n}, \qquad \widetilde{T}(u) := \diag(u) + u\mathbf{1}^\top,
\]For every integer $m\geq 1$ and every $u=(u_1,\dots,u_n)^\top\in\C^n$,  we have
\begin{equation}\label{eq:Te-2m-pointwise}
  \|\widetilde{T}(u)\|_{S_{2m}}^{2m}
  \leq
  C_m(n)\sum_{j=1}^n |u_j|^{2m},
\end{equation}
where
\[
  C_m(n) := \frac{(n+1)^{2m}+n-1}{n}.
\]Moreover, equality holds in \eqref{eq:Te-2m-pointwise} whenever $|u_1|=\cdots=|u_n|$. In particular,
\[
  \|\widetilde{T}\|_{\ell^{2m}(\mathbb{C}^n)\to S_{2m}} = C_m(n)^{\frac{1}{2m}}.
\]
\end{lem}

\begin{proof}
Fix $m\geq 1$ and $u=(u_1,\dots,u_n)^\top\in\C^n$. Set $F := \diag(d_1,\dots,d_n)$, $K := u u^*$, $\alpha := n+2$ and write
\[
  A_r := \sum_{j=1}^n |u_j|^{2r}, \qquad r\geq 1.
\]
By \eqref{eq:Thm_6.1_BBstar_11}, we have
\[
  \widetilde{T}(u)\widetilde{T}(u)^\ast
  = \diag(u) \diag(u)^* + \diag(u)\mathbf{1}u^*
    + u\mathbf{1}^\top\diag(u)^* + u\mathbf{1}^\top\mathbf{1}u^*
  = F + \alpha K.
\]
Hence
\[
  \|\widetilde{T}(u)\|_{S_{2m}}^{2m}
  = \tr\left( |\widetilde{T}(u)|^{2m} \right)
  = \tr\left( (\widetilde{T}(u)\widetilde{T}(u)^\ast)^m \right)
  = \tr\left( (F+\alpha K)^m \right).
\]We now expand $(F+\alpha K)^m$ using the rank-one structure of $K$, similarly to the proof of \eqref{eq:will_be_used_in_app_D1}, and obtain
\begin{equation}\label{eq:Te-trace-Sr-formula}
  \|\widetilde{T}(u)\|_{S_{2m}}^{2m}
  = A_m +
    \sum_{\ell=1}^m \alpha^\ell
    \sum_{\substack{r_0,\dots,r_\ell\geq 0\\ r_0+\cdots+r_\ell = m-\ell}}
      A_{r_0+r_\ell+1} \cdot \prod_{j=1}^{\ell-1} A_{r_j+1}.
\end{equation}
The argument in the proof of Corollary~\ref{cor:sharp-2m} (based only on H\"{o}lder's inequality and elementary
estimates for power sums of positive numbers) shows that
\[
  \sum_{\substack{r_0,\dots,r_\ell\geq 0\\ r_0+\cdots+r_\ell = m-\ell}}
      A_{r_0+r_\ell+1}\cdot\prod_{j=1}^{\ell-1} A_{r_j+1}
  \leq
  \binom{m}{\ell} n^{\ell-1} A_m
  \qquad (1\leq \ell\leq m).
\]
Combining this with \eqref{eq:Te-trace-Sr-formula} and using $\alpha=n+2$ we obtain
\[
  \|\widetilde{T}(u)\|_{S_{2m}}^{2m}
  \leq A_m\left(
    1 + \sum_{\ell=1}^m
      \binom{m}{\ell} (n+2)^\ell n^{\ell-1}
  \right).
\]
The binomial theorem (as in the proof of Corollary~\ref{cor:sharp-2m}) gives
\[
  1 + \sum_{\ell=1}^m
      \binom{m}{\ell} (n+2)^\ell n^{\ell-1}
  = \frac{(n+1)^{2m}+n-1}{n}
  = C_m(n),
\]
and therefore
\[
  \|\widetilde{T}(u)\|_{S_{2m}}^{2m}
  \leq C_m(n)\,A_m
  = C_m(n)\sum_{j=1}^n |u_j|^{2m}.
\]
This proves \eqref{eq:Te-2m-pointwise}. 

To see that the constant is sharp, take any $u$ with $|u_1|=\cdots=|u_n|=r>0$.
Then $F = r^2 I$ and $K = u u^\ast$ has rank one with
$\|u\|_{\ell^2}^2 = nr^2$. Hence
\[
  \widetilde{T}(u)\widetilde{T}(u)^\ast = F + (n+2)K
  = r^2\bigl( I_n + (n+2)\tfrac{1}{n} \mathbf{1}\mathbf{1}^\ast\bigr),
\]
so $\widetilde{T}(u)\widetilde{T}(u)^\ast$ is unitarily diagonalizable with one eigenvalue
$r^2(1+(n+2)n)$ and $n-1$ eigenvalues equal to $r^2$. Consequently,
\[
  \|\widetilde{T}(u)\|_{S_{2m}}^{2m}
  = (n-1)r^{2m} + \left((n+1)^{2m}\right)r^{2m}
  = \left((n+1)^{2m}+n-1\right)r^{2m},
\]
while
\[
  \sum_{j=1}^n |u_j|^{2m} = n r^{2m}.
\]
Thus
\[
  \frac{\|\widetilde{T}(u)\|_{S_{2m}}^{2m}}
       {\sum_{j=1}^n |u_j|^{2m}}
  = \frac{(n+1)^{2m}+n-1}{n}
  = C_m(n),
\]
showing that equality is attained in \eqref{eq:Te-2m-pointwise}. Taking the supremum over $u\neq 0$ then gives
$\|\widetilde{T}\|_{\ell^{2m}(\mathbb{C}^n)\to S_{2m}} = C_m(n)^{1/(2m)}$, this completes the proof.
\end{proof}

We are now ready to present the following.
\begin{proof}[Proof of Theorem~\ref{thm:refined-interpolation}]

As in the proof of Theorem~\ref{thm:negative-interpolation_for_p_2_infinity}, set $u_j := 1/z_j$, $u := (u_1,\dots,u_n)^\top$, and define the linear map
	\begin{equation}\label{eq:definition_of_T_tilde_33}
	\widetilde{T} : \mathbb{C}^n \to \mathbb{C}^{n\times n},\qquad
	\widetilde{T}(u) := \diag(u) + u\mathbf{1}^\top.
\end{equation}
By~\eqref{eq:eig-vs-sing-B}, we have
	\begin{equation}\label{eq:refined-Weyl}
		\sum_{k=1}^n \frac{1}{|w_k|^{p}}
		= \sum_{k=1}^n \bigl|\lambda_k(B)\bigr|^{p}
		\le
		\|B\|_{S_p}^{p}
		= \|\widetilde{T}(u)\|_{S_p}^{p},
		\qquad p\ge1.
	\end{equation} We now estimate $\|\widetilde{T}(u)\|_{S_p}$ in terms of $\|u\|_{\ell^p}$ for $p\geq 2$
by complex interpolation between the sharp even endpoints provided by
Lemma~\ref{lem:Te-2m-endpoint}. Fix $p\geq 2$. Choose the
integer $m:=\lfloor \frac{p}{2}\rfloor$ so that $2m\leq p\leq 2m+2$, and set
\[
  \alpha := \frac{p}{2}-m \in [0,1].
\]
Then there exists $\theta\in[0,1]$
such that
\[
  \frac{1}{p} = \frac{1-\theta}{2m} + \frac{\theta}{2m+2}.
\]
The families $(\ell^r(\mathbb{C}^n))_{1\leq r\leq\infty}$ and $(S_r)_{1\leq r\leq\infty}$
form compatible complex interpolation scales; see, for example,
Theorem \ref{thm:RT_interpolation1} and Remark \ref{rmk:RT_interpolation1}. Therefore,
\[
  \left(\ell^{2m}(\mathbb{C}^n),\ell^{2m+2}(\mathbb{C}^n)\right)_{[\theta]} = \ell^p(\mathbb{C}^n),
  \qquad
  \left(S_{2m},S_{2m+2}\right)_{[\theta]} = S_p \quad \text{(with equal norms)}.
\]
The abstract interpolation theorem (see Theorem~\ref{thm:complex-interpolation}) then yields
\[
  \|\widetilde{T}\|_{\ell^p(\mathbb{C}^n)\to S_p}
  \leq
  \|\widetilde{T}\|_{\ell^{2m}(\mathbb{C}^n)\to S_{2m}}^{1-\theta}
  \|\widetilde{T}\|_{\ell^{2m+2}(\mathbb{C}^n)\to S_{2m+2}}^{\theta}.
\]
Using Lemma~\ref{lem:Te-2m-endpoint} at $m$ and $m+1$, we obtain
\[
  \|\widetilde{T}\|_{\ell^p(\mathbb{C}^n)\to S_p}
  \leq
  C_m(n)^{\frac{1-\theta}{2m}}\,
  C_{m+1}(n)^{\frac{\theta}{2m+2}},
\]
and hence
\begin{equation}\label{eq:refined-op-norm-p}
  \|\widetilde{T}(u)\|_{S_p}^p
  \leq
  C_m(n)^{\frac{p(1-\theta)}{2m}}
  C_{m+1}(n)^{\frac{p\theta}{2m+2}}
  \|u\|_{\ell^p}^p.
\end{equation}A straightforward simplification shows that the exponents of $C_m(n)$ and
$C_{m+1}(n)$ in \eqref{eq:refined-op-norm-p} can be written in the
form
\[
  \frac{p(1-\theta)}{2m} = m+1-\frac{p}{2} = 1-\alpha,
  \qquad
  \frac{p\theta}{2m+2} = \frac{p}{2}-m = \alpha.
\]
Thus \eqref{eq:refined-op-norm-p} becomes
\[
  \|\widetilde{T}(u)\|_{S_p}^p
  \leq
  C_m(n)^{1-\alpha} C_{m+1}(n)^\alpha
  \|u\|_{\ell^p}^p.
\]
Since $u_j=1/z_j$, we have $\|u\|_{\ell^p}^p = \sum_{j=1}^n |u_j|^p
= \sum_{j=1}^n |z_j|^{-p}$, and combining this with
\eqref{eq:refined-Weyl} yields
\[
  \sum_{k=1}^n \frac{1}{|w_k|^p}
  \leq
  C_m(n)^{1-\alpha}\,C_{m+1}(n)^\alpha
  \sum_{j=1}^n \frac{1}{|z_j|^p},
\]
which is precisely \eqref{eq:refined-interpolation} with
\[
  C(n,p) = C_m(n)^{1-\alpha} C_{m+1}(n)^\alpha.
\]This completes the proof.
\end{proof}


\begin{thebibliography}{99}


\bibitem{Barvinok25}
A.~Barvinok, \textit{A Course in Convexity}, Graduate Studies in Mathematics, Vol.~54, American Mathematical Society, Providence, RI, 2025.


\bibitem{Bergh76}
J.~Bergh and J.~L\"ofstr\"om, \emph{Interpolation Spaces: An Introduction}, Vol.~223, Springer, Berlin--Heidelberg, 1976.


\bibitem{Bha07}
B.~V.~R.~Bhat and M.~Mukherjee, Integrators of matrices, \emph{Linear Algebra Appl.} \textbf{426} (2007), 71--82.


\bibitem{Bhatia09}
R.~Bhatia, \textit{Positive Definite Matrices}, Princeton University Press, 2009.


\bibitem{Bhatia13}
R.~Bhatia, \textit{Matrix Analysis}, Vol.~169, Springer Science \& Business Media, 2013.


\bibitem{BE95}
P.~Borwein, T.~Erd\'elyi, \textit{Polynomials and Polynomial Inequalities}, Vol.~161, Springer Science \& Business Media, 1995.


\bibitem{BIS99}
M.~G.~de Bruin, K.~G.~Ivanov, A.~Sharma, A conjecture of Schoenberg, \textit{J. Inequal. Appl.} \textbf{4} (1999): 183--213.


\bibitem{dS99}
M.~G.~de Bruin, A.~Sharma, On a Schoenberg-type conjecture, in: \textit{Continued Fractions and Geometric Function Theory (CONFUN)}, Trondheim, 1997, \textit{J. Comput. Appl. Math.} \textbf{105} (1999): 221--228.


\bibitem{CN06}
W.~S.~Cheung, T.~W.~Ng, A companion matrix approach to the study of zeros and critical points of a polynomial, \textit{J. Math. Anal. Appl.} \textbf{319} (2006): 690--707.


\bibitem{DG20}
S.~Danielyan and A.~Guterman, On integral of polynomial with multiple roots, \emph{Zap. Nauchn. Sem. POMI} \textbf{482} (2019), 28--44; English transl., \emph{J. Math. Sci. (N.~Y.)} \textbf{249} (2020), no.~2, 128--138.


\bibitem{DGN21}
S.~V.~Danielyan, A.~E.~Guterman, T.~W.~Ng, Integrability of diagonalizable matrices and a dual Schoenberg type inequality, \textit{J. Math. Anal. Appl.} \textbf{498} (2021): 124909.


\bibitem{GGTW25}
B.~Georgiev, J.~G\'{o}mez-Serrano, T.~Tao, A.~Z.~Wagner, Mathematical exploration and discovery at scale. \textit{arXiv preprint}, arXiv:2511.02864v1, 2025.


\bibitem{GohbergKrein69}
I.~C.~Gohberg, M.~G.~Kre\u{\i}n, \textit{Introduction to the Theory of Linear Nonselfadjoint Operators}, American Mathematical Society, Providence, 1969.


\bibitem{Hay19}
W.~K.~Hayman, \emph{Research Problems in Function Theory}, University of London, London, 1967; The fiftieth anniversary edition, Springer Nature, 2019.


\bibitem{HJ91}
R.~A.~Horn, C.~R.~Johnson, \textit{Topics in Matrix Analysis}, Cambridge University Press, 1991.


\bibitem{HJ13}
R.~A.~Horn, C.~R.~Johnson, \textit{Matrix Analysis}, Cambridge University Press, 2nd ed., 2013.


\bibitem{Kat99}
E.~S.~Katsoprinakis,
On the complex role set of a polynomial,
in: N.~Papamichael, St.~Ruscheweyh and E.~B.~Saff (eds.),
\emph{Computational Methods and Function Theory 1997 (Nicosia)},
Series in Approximations and Decompositions, Vol.~11,
World Scientific, River Edge, NJ, 1999.


\bibitem{KR03}
 N.~L.~Komarova, I.~Rivin, Harmonic mean, random polynomials and stochastic matrices, \textit{Adv. Appl. Math.} \textbf{31} (2003): 501--526.

 
\bibitem{Kos84}
A.~Kostrikin, Conservative polynomials, in: \emph{Stud. Algebra}, Tbilisi, 1984, pp.~115--129.


\bibitem{KT16}
O.~Kushel, M.~Tyaglov, Circulants and critical points of polynomials, \textit{J. Math. Anal. Appl.} \textbf{439} (2016): 634--650.


\bibitem{LXZ21}
M.~Lin, M.~Xie, J.~Zhang, Remarks on circulant matrices and critical points of polynomials, \textit{J. Math. Anal. Appl.} \textbf{502} (2021): 125233.


\bibitem{Mal04}
S.~Malamud, Inverse spectral problem for normal matrices and the Gauss--Lucas theorem, \textit{Trans. Amer. Math. Soc.} \textbf{357} (2004): 4043--4064.


\bibitem{Mar66}
M.~Marden, \textit{Geometry of Polynomials}, 2nd ed., Mathematical Surveys, No.~3, American Mathematical Society, Providence, RI, 1966.


\bibitem{MOA11}
A.~W.~Marshall, I.~Olkin, B.~Arnold, \textit{Inequalities: Theory of Majorization and Its Applications}, 2nd ed., Springer, New York, 2011.


\bibitem{NocedalWright2006} J.~Nocedal and S.~J.~Wright, \textit{Numerical Optimization}, 2nd ed., Springer, New York, 2006.


\bibitem{Per03}
R.~Pereira, Differentiators and the geometry of polynomials, \textit{J. Math. Anal. Appl.} \textbf{285} (2003): 336--348.


\bibitem{Pisier03}
G.~Pisier, Q.~Xu, Non-commutative \(L^p\)-spaces, in: W.~B.~Johnson, J.~Lindenstrauss (Eds.), \textit{Handbook of the Geometry of Banach Spaces}, Vol.~2, Elsevier Science~B.V., 2003, pp.~1459--1517.


\bibitem{QS02}
Q.~I.~Rahman, G.~Schmeisser, \textit{Analytic Theory of Polynomials}, Oxford University Press, Oxford, 2002.


\bibitem{Rov12}
I.~Roventa, A note on Schur-concave functions, \textit{J. Inequal. Appl.} (2012): 159.


\bibitem{Sch77}
G.~Schmeisser,
On Ilieff's conjecture,
\emph{Math. Z.} \textbf{156} (1977), 165--173.


\bibitem{Sch86}
I.~J.~Schoenberg, A conjectured analogue of Rolle's theorem for polynomials with real or complex coefficients, \textit{Amer. Math. Monthly} \textbf{93} (1986): 8--13.


\bibitem{Simon05}
B.~Simon, \textit{Trace Ideals and Their Applications}, No.~120, American Mathematical Society, 2005.


\bibitem{Sma81}
S.~Smale,
The fundamental theorem of algebra and complexity theory,
\emph{Bull. Amer. Math. Soc. (N.\,S.)} \textbf{4} (1981), 1--36.


\bibitem{Tang25}
Q.~Tang, Schoenberg type inequalities, \textit{arXiv preprint}, arXiv:2504.09837v2, 2025.


\bibitem{Tao22}
T.~Tao, Sendov's conjecture for sufficiently high degree polynomials, \textit{Acta Math.} \textbf{229} (2022): 347--392.


\bibitem{Tri78}
H.~Triebel,
\emph{Interpolation Theory, Function Spaces, Differential Operators}, Deutscher Verlag Wissensch., Berlin, 1978.


\bibitem{Zh11}
F.~Zhang, \emph{Matrix Theory: Basic Results and Techniques},
2nd ed., Springer, New York, 2011.


\bibitem{Zha25}
T.~Zhang, A refinement of Pawlowski's result,
\emph{Proc. Amer. Math. Soc.}, to appear; arXiv:2411.07105, 2025.


\end{thebibliography}
\end{document}